\documentclass[a4paper,11pt]{report}
\usepackage[utf8x]{inputenc}
\usepackage{amssymb}
\usepackage{amsthm}
\usepackage{amsmath}
\usepackage[english]{babel}
\usepackage{accents}
\usepackage{amsfonts}
\usepackage{textcomp}
\usepackage{stmaryrd}
\usepackage{tikz}
\usepackage{booktabs}
\usepackage{url}
\usepackage[scaled=0.90]{helvet}
\usepackage{graphicx}
\graphicspath{ {/home/stefan/Pictures/} }

\title{Projective Wellorders and the Nonstationary Ideal} \author{Stefan Hoffelner}
\date{29.6.2016}

\begin{document}

\newtheorem{thm}{Theorem}
\newtheorem{Claim}[thm]{Claim}
\newtheorem{Fact}[thm]{Fact}
\newtheorem{Proposition}[thm]{Proposition}
\newtheorem{Definition}[thm]{Definition}
\newtheorem{Assumption}[thm]{Assumption}
\newtheorem{Lemma}[thm]{Lemma}
\newtheorem{Question}{Question}

\newcommand{\ZFC}{\mathsf{ZFC}}
\newcommand{\ZF}{\mathsf{ZF}}
\newcommand{\ZFP}{\mathsf{ZF}^-}
\newcommand{\AC}{\mathsf{AC}}
\newcommand{\PFA}{\mathsf{PFA}}
\newcommand{\BPFA}{\mathsf{BPFA}}
\newcommand{\MRP}{\mathsf{MRP}}
\newcommand{\MM}{\mathsf{MM}}

\newcommand{\card}{\hbox{card}}
\newcommand{\Card}{\hbox{Card}}
\newcommand{\hcard}{\hbox{hcard}}
\newcommand{\Reg}{\hbox{Reg}}
\newcommand{\GCH}{\mathsf{GCH}}
\newcommand{\CH}{\mathsf{CH}}
\newcommand{\Ord}{\hbox{Ord}}
\newcommand{\Range}{\hbox{Range}}
\newcommand{\Dom}{\hbox{Dom}}
\newcommand{\Ult}{\hbox{Ult}}
\newcommand{\I}{\hbox{I}}
\newcommand{\II}{\hbox{II}}

\newcommand{\mouseM}{\mathcal{M}}
\newcommand{\mouseN}{\mathcal{N}}
\newcommand{\mouseP}{\mathcal{P}}
\newcommand{\treeT}{\mathcal{T}}

\newcommand{\forceQ}{\mathbb{Q}}
\newcommand{\forceP}{\mathbb{P}}
\newcommand{\forceR}{\mathbb{R}}
\newcommand{\seal}{\mathbb{S}(\vec{S})}

%Stat fuer die everywhere stationary class S
\newcommand{\Stat}{\mathcal{S}}
\newcommand{\NSA}{\hbox{NS}_{\omega_1} \upharpoonright A} \newcommand{\NS}{\hbox{NS}_{\omega_1}}

\maketitle

\begin{abstract}
We show that, under the assumption of the existence of $M_1^{\#}$, 
there exists a model on which the restricted nonstationary ideal $\NSA$ is $\aleph_2$-saturated, for $A$ a stationary
co-stationary subset of $\omega_1$, while the full nonstationary ideal $\NS$
can be made $\Delta_1$ definable with $K_{\omega_1}$ as a parameter. Further 
we show, again under the assumption of the existence of $M_1^{\#}$ that there is a model of set theory such that 
$\NS$ is $\aleph_2$-saturated and such that there is lightface $\Sigma^1_4$-definable well-order on the reals.
This result is optimal in the presence of a measurable cardinal.
\end{abstract}

\newpage
\begin{abstract}
 Unter der Annahme der Existenz von $M_1^{\#}$ wird ein mengentheoretisches Modell von $\ZFC$ konstruiert in dem 
das nonstation\"are Ideal $\NS$ auf $\omega_1$ saturiert ist und in dem eine $\Sigma_4^{1}$-definierbare
Wohlordnung auf den reellen Zahlen m\"oglich ist. Dieses Resultat ist optimal sobald eine messbare Kardinalzahl 
in dem Universum angenommen wird. Desweiteren wird, wieder unter der Annahme der
Existenz von $M_1^{\#}$ ein mengentheoretisches Modell von $\ZFC$ konstruiert in welchem das auf eine beliebige, vorher fixierte
station\"are, co-station\"are Teilmenge $A \subset \omega_1$ eingeschr\"ankte nonstation\"are Ideal $\NSA$ saturiert ist,
w\"ahrend $\NS$ selbst $\Delta_1$-definierbar ist mit $K_{\omega_1}$ als einzigem Parameter.

\end{abstract}

\section*{Introduction}

This thesis deals with two fundamental notions in set theory: definability and the saturation of the nonstationary ideal.
The significance of the first for set theoretic investigations was already known to the early
descriptive set theorists of the Russian and Polish schools who realized that the continuum hypothesis 
has a positive answer when looking at analytic (i.e. easily describable) sets only. Twenty years later K. G\"odel showed 
how immensely powerful the notion of definability can be when used to investigate the set theoretic universe itself
via his introduction of the constructible universe $L$. The quest of finding $L$-like universes which
at the same time allow large cardinal properties is still one of the central open problems
in set theory and is just another testimony of the importance of the concept.

The investigation of the saturation of the nonstationary ideal on $\omega_1$, $\NS$
has a very interesting history as well. K. Kunen showed that 
given a huge cardinal there is a model in which $\NS$ is $\aleph_2$-saturated, a result
which caused a considerable amount of attention as supercompact
cardinals served as an informal upper bound for all natural
consequences of large cardinal axioms, and hugeness is considerably stronger. Using 
a generic ultrapower argument he even argued that hugeness is the exact consistency strength of
the assertion ``$\NS$ is saturated''. Taking stationary subsets as conditions in a poset 
and picking a generic filter one arrives at a $V$-ultrafilter, which can be used to 
form the ultrapower $Ult(V,G)$ in $V[G]$. The resulting embedding $j: V \rightarrow M$ 
has $\omega_1$ as critical point and $M^{<j(\omega_1)} \subset M$ does hold, 
which resembles the definition of hugeness.
This intuition however turned out to be flawed as
by the work of M. Foreman, M. Magidor and S. Shelah
the forcing axiom $\MM$ (whose consistency follows from a supercompact cardinal) implies
that $\NS$ is $\aleph_2$-saturated, which was later improved by S. Shelah
who showed that already a Woodin cardinal suffices to construct a model
in which $\NS$ is saturated. 

A natural question to ask is whether the saturation of $\NS$ is consistent with $\CH$.
The original proof of Shelah for the saturation of $\NS$ does not indicate an answer
to this question. H. Woodin however, in \cite{W}, proved that if there is a measurable cardinal and
if $\NS$ is saturated then $\CH$ fails, in fact there exists a definable counterexample to 
$\CH$. This impressive result indicates that there might be a surprising connection between 
the statements of $\CH$ and $\NS$ being saturated. 

Definability enters the picture in the following way. By a result of G. Hjorth (see \cite{Hjorth}), a $\undertilde{\Sigma}_3^{1}$-definable 
wellorder of the reals together with the assumption that every real has a sharp implies $\CH$. An investigation 
of the possibility of a set theoretic universe where $\NS$ is saturated and the reals admit a projectively
definable wellorder will therefore help to clarify the actual connection between $\CH$ and the saturation of $\NS$.
One of the two main theorems of this thesis is the following: 
\begin{thm}
 Assume that $M_1^{\#}$ exists. Then there is a set theoretic universe in which
$\NS$ is saturated and there is a (lightface) $\Sigma^1_4$-definable wellorder on the reals.
\end{thm}
Very roughly speaking its proof takes Shelah's argument for making $\NS$ saturated and
adds forcings which will code reals by triples of ordinals, using a technique developed by
A. Caicedo and B. Velickovic in \cite{CV}. Twisting their original definition 
makes a further localization forcing possible which codes all the relevant information
for the wellorder in one real in such a way that any suitable, countable model can work with it, 
thus witnessing the relation.

This result is also interesting from a second perspective. It is often a challenging
task to consider set theoretic universes with certain features, usually obtained assuming some large cardinal,
and additionally equip them with definability properties. One can find a lot of
literature devoted to this type of investigations such as \cite{H}, \cite{FF}, and \cite{CF}.
Though very diverse in their goals
and means these results have in common that the ground model from which they start is $L$ or some 
slight extension of it. This approach will not work when investigating possible models for ``$\NS$ saturated
and definable wellorders on the reals''. By a result of R. Jensen and J. Steel the statement 
``$\NS$ is saturated'' is equiconsistent with the existence of a Woodin cardinal, thus one is compelled
to work in inner models which will differ from $L$ quite substantially in certain ways and render many
of the usual arguments harder if not impossible. So this thesis can be seen as an attempt of 
finding coding methods which still work for stronger inner models, making the definability of certain sets
possible. This also applies for the second result of this thesis.

The second theorem deals with the definability of $\NS$ in the presence of a normal, saturated ideal
on $\omega_1$. 

\begin{thm}
 Assume that $M_1^{\#}$ exists and that $A\subset \omega_1$ is a stationary, co-stationary set.
Then there is a model of $\ZFC$ in which the nonstationary ideal restricted to subsets of $A$ is saturated while
the full ideal
$\NS$ is $\Delta_1$ definable using $\omega_1$ as a parameter. 
\end{thm}
Its proof starts again with Shelah's argument of making $\NS$ saturated, and adds forcings which will
write the characteristic functions of any stationary subset of $\omega_1$ into a pattern of canonically
definable trees which have a cofinal branch or not. These codings are also used
in \cite{CF} to show that $\BPFA$ and a $\Sigma_4^{1}$ wellorder on the reals
are simultaneously possible. We define a suitable class of models which is stationary and
the fact that throughout the iteration there is always the stationary
complement of $A$ turns out to be crucial when ensuring that this class remains stationary along the way
of the iteration. Again we will force to localize the relevant information in such a way that
already any suitable model can read it off from a subset of $\omega_1$ only, yielding the 
desired definability of $\NS$.

This theorem grew out of an attempt to answer a question in \cite{FL}, where it is asked 
whether one could improve Woodin's result of $Con(\ZFC +$ ``there are infinitely many Woodin cardinals'')
implies $Con(\ZFC +$ ``$\NS$ is saturated + $\NS$ is $\Delta_1$-definable``), and the techniques of our proof were initially intended
to solve the problem for having the full $\NS$ saturated. It turned out however that
one faces big problems as the stationary class of nice models will lose its stationarity
during the iteration as its factors are semiproper only. This raised a serious conflict
with desired properties of our core model which served as the ground model during the iteration,
and so we decided to tackle the case for the restricted ideal first.

We end this introduction with a short outline of how this thesis is organized.
The first chapter very shortly defines the most important tools which are used in the proofs
of the two main theorems. $RCS$-iterations and inner model theory play a
fundamental role in the theorems so they are introduced. These belong to the more technical and abstract
parts of set theory, but we hope that the reader can fully understand the thesis
without being an expert in these subfields as long as he keeps in mind a couple
of properties which, in the inner model case, we tried to state as additional
axioms. Also a proof of Shelah's result of $Con(\ZFC +$ ''there is a Woodin cardinal``)
implies $Con(\ZFC +$ ''$\NS$ is saturated``) is included as it is the starting point for 
our investigations.

In the second chapter we prove Theorem 2 and in the third chapter we prove Theorem 1.
They can be read independently of each other.

\tableofcontents

\chapter{Preliminaries}

\section{$\Stat$-semiproperness and RCS-iterations} This section is a short reminder of the main definitions and properties
of notions centered around the fundamental concept of properness.
As is well known the definition of a proper notion of forcing was singled out by S. Shelah in the eighties.
Proper forcings represent a class of forcings which is closed under iterations
 with countable support and which contain all the forcings which have the ccc and which are $\omega$-closed.
There is a weaker version of properness which Shelah called semiproper which still enjoys
a lot of the nice properties of properness which however can not be iterated with countable support in a semiproperness preserving way.
Nevertheless there is a generalization of countable support iterations, the so-called revised countable
support iteration (RCS-iteration), which yields, applied to semiproper factors, always a semiproper notion of forcing.
For our purpose a further generalization of semiproperness is needed, namely the so called $\Stat$-semiproper forcings.
All these notions will be introduced now, the missing proofs can be found e.g. in M. Viale's notes \cite{Viale} or
in S. Shelah's \cite{S}, the definition I use can be found in Schlindwein's paper \cite{Sch}.

\begin{Definition}
A notion of forcing $\forceP$ is called proper if and only if there is a regular cardinal $\lambda$ such that
$\lambda > 2^{|\forceP|}$
and there is a closed unbounded $C \subset [H_{\lambda}]^{\omega}$ of elementary submodels $M \prec (H_{\lambda}, \in, <,..)$
(where $<$ defines a well-order. on $H_{\lambda}$ and the dots should indicate further possible second order parameters)
such that for every condition $p \in \forceP \cap M$ there exists a $q \le p$ such that $q$ is $(M,\forceP$)-generic; the latter
meaning that every maximal antichain $A \subset \forceP$, $A \in M$ induces a predense subset $A \cap M$ below the condition $q$.
\end{Definition}

It is very well known that there are other ways of defining what it means for a condition $q \in \forceP$
to be $(M, \forceP)$-generic, e.g. it is equivalent to the assertion that for every name for an ordinal 
$\dot{\alpha} \in M$, $q \Vdash \dot{\alpha} \in M$. Also for a partial
 order $\forceP$ to be proper is equivalent to preserve stationary subsets 
of $[\lambda]^{\omega}$ for every uncountable cardinal $\lambda$.

The next notion is crucial in the proof of the first main result.

\begin{Definition}
A class $\Stat$ consisting of elements $M \in [H_{\theta}]^{\omega}$ for arbitrary cardinals $\theta$
is called everywhere stationary if the following two things hold: \begin{enumerate}
\item for all regular cardinals $\theta$, $\Stat \cap [H_{\theta}]^{\omega}$ is a stationary subset of $[H_{\theta}]^{\omega}$
\item $\Stat$ is closed under truncation meaning that for all regular, uncountable cardinals $\theta$ and all $M \in \Stat$
$M \cap H_{\theta} \in \Stat$.
\end{enumerate}

\end{Definition}

Now we can introduce a generalized form of proper forcing:

\begin{Definition}
A notion of forcing $\forceP$ is $\Stat$-proper for an everywhere stationary class $\Stat$ if there exists a regular, uncountable
cardinal $\theta$ such that $\theta > 2^{|\forceP|}$ and there is a club $C \subset [H_{\theta}]^{\omega}$ such that
for every $M \in C \cap \Stat$ and every $p \in \forceP \cap M$ there is a $q \le p $ which is $(M,\forceP)$-generic.
\end{Definition}

This definition can also be used if $\Stat$ is not a class but a stationary subset of some $[H(\theta)]^{\omega}$.
It is well known due to Shelah that $\Stat$-properness is almost as good as full properness, I.e. it is preserved
under countable support iterations, preserves stationary subsets 
$T \subset [H_{\theta} \cap V]^{\omega}$ as long as $T \subset \Stat$ and 
does not collapse $\aleph_1$. In particular this means that 
during an iteration of $\Stat$-proper forcings with countable support, $\Stat$ remains an everywhere stationary class in 
$[H_{\theta} \cap V]^{\omega}$.

Let us now turn to semiproperness:

\begin{Definition}
Let $\forceP$ be a partial order, then $\forceP$ is semiproper if and only if there is a cardinal $\theta > 2^{|\forceP|}$
and there is a club $C \subset [H_{\theta}]^{\omega}$ of elementary submodels $M \prec (H_{\theta}, \in, < ,...)$
such that every condition $p \in \forceP \cap M$ has an $(M, \forceP)$-semigeneric condition $q$ below it; and a
condition $q$ is $(M, \forceP)$-semigeneric if and only if whenever $\dot{\alpha}$ is a name for a countable
ordinal in $M$ then $q \Vdash \dot{\alpha} \in M$.
\end{Definition}

This generalizes in a straightforward way:

\begin{Definition}
Let $\Stat$ be an everywhere stationary class. A notion of forcing 
$\forceP$ is $\Stat$-semiproper if for a regular $\theta > 2^{|\forceP|}$ 
there is a club $C \subset [H_{\theta}]^{\omega}$ of elementary submodels 
$M \prec (H_{\theta}, \in, <, ...)$ such that for every $M \in C \cap \Stat$ 
and for every $p \in M$ there is a $q<p$ which is $(M, \forceP)$-semigeneric.
\end{Definition}
It is well known that in fact something stronger is true: \begin{Fact}
If $\Stat$ is everywhere stationary then $\forceP$ is $\Stat$ semiproper 
if and only if for every sufficiently large $\theta$ and every 
$M \prec H_{\theta}$, $M \in \Stat$
such that $p, \forceP \in M$ there is a $q< p$ such that $q$ is $(M, \forceP)$-semigeneric.
\end{Fact}
\begin{proof}
Let $\lambda > 2^{|\forceP|}$ then
there is a club $C \subset H_{\lambda}$ such that for every 
$N \in C$ and every $p \in N$ there is a stronger $q$ which is $(N,\forceP)$-semigeneric.
Let $\theta > \lambda$ and let $M \prec H_{\theta}$, $M \in \Stat$ 
and let $<$ be a fixed well-order. of $H_{\theta}$ then $M$ will have 
the $<$-least club $C$ witnessing the above as an element. 
Hence $M \cap H_{\lambda} \in C$ and as $M \in \Stat$ 
we get $(M \cap H_{\lambda}) \in \Stat \cap C$. 
Thus we find for every $p \in M \cap H_{\lambda}$ a $q<p$ which is $(M \cap H_{\lambda})$-semigeneric.
As $\lambda > 2^{|\forceP|}$, $q$ is also $(M, \forceP)$-semigeneric.

\end{proof}

A most important feature of proper notions of forcing 
is that iterating them with countable support results 
in a proper forcing again. (Its proof relies heavily 
on the fact that countable subsets in the extension 
can be covered by countable subsets of the groundmodel. 
This has as a consequence that whenever we break an 
iteration with countable support 
$(\forceP_{\alpha}, \dot{\forceQ}_{\alpha} \, : \, \alpha < \delta)$ into 
two halves $(\forceP_{\alpha}, \dot{\forceQ}_{\alpha} \, : \, \alpha < \beta)$ 
and the according tail forcing, 
the second half is still a countable support iteration)

Contrary to proper forcings, semiproperness will not be preserved 
when using a countable support iteration. In order to give an example 
illustrating this fact we have to be able to break an iteration 
$(\forceP_{\alpha}, \dot{\forceQ}_{\alpha} \, : \, \alpha < \delta)$ 
into two halves $(\forceP_{\alpha}, \dot{\forceQ}_{\alpha} \, : \, \alpha < \beta)$ 
for an ordinal $\beta < \delta$, and $\dot{\forceP}_{[\beta,\delta)}$, where 
the second half is an iteration over the ground model $V^{\forceP_{\beta}}$ of 
length $\delta - \beta$ with factors which correspond to the $\dot{\forceQ}_{\alpha}$.
Intuitively it is clear that such a construction should be possible, though on 
the one hand the factors $\dot{\forceQ}_{\alpha}$ are $\forceP_{\alpha}$-names of 
partial orders, while in the yet not explicitly defined iteration 
$\dot{\forceP}_{[\beta,\delta)}$ the factors of the iteration should be 
$\forceP_{\beta}$-names of $\dot{\forceP}_{[\beta,\alpha)}$-names of partial 
orders (where $\dot{\forceP}_{[\beta,\alpha)}$ is a $\forceP_{\beta}$-name for an iteration of length $\alpha - \beta$).
Further the appropriate definition of
the second half $\dot{\forceP}_{[\beta,\delta)}$
should have the consequence that the two step iteration $(\forceP_{\alpha}, \dot{\forceQ}_{\alpha} \, : \, \alpha < \beta)$ $\ast$
$\dot{\forceP}_{[\beta,\alpha)}$
is forcing equivalent to $(\forceP_{\alpha}, \dot{\forceQ}_{\alpha} \, : \, \alpha < \delta)$.
A precise analysis of the situation leads quickly
to a rather tedious definition:

\begin{Definition}
Let $(\forceP_{\beta}, \dot{\forceQ}_{\beta}  \, : \, \beta < \alpha)$ be a forcing iteration of length $\alpha$, let $\eta < \alpha$ then
we let $\dot{\forceP}_{[\eta, \alpha)}$ be a $\forceP_{\eta}$-name defined like this: \begin{enumerate}
\item[] $\forall p \in \forceP_{\eta}$ $\forall \dot{s} \in V^{\forceP_{\eta}}$ $( p \Vdash_{\eta} \dot{s} \in
\dot{\forceP}_{[\eta, \alpha)}$ if and only if $\forall q <_{\eta} p$ $\exists r \in \forceP_{\alpha}$ $(r \upharpoonright \eta
< q$ and $ r \upharpoonright \eta \Vdash_{\eta} r \upharpoonright [\eta, \alpha) = \dot{s} )$.
\end{enumerate}

\end{Definition}

Here the term $r \upharpoonright [\eta, \alpha)$ does not mean the 
usual check-name for the tail sequence as seen in $V^{\forceP_{\eta}}$, 
but rather the $\forceP_{\eta}$-name for a function with domain 
$[\eta, \alpha)$ and values at stage $\xi$ are
$\dot{\forceP}_{[\eta, \xi)}$-names according to the $\forceP_{\xi}$-name $r(\xi)$. 
Thus the above definition has the more readable characterization that for a 
$\forceP_{\eta}$-generic filter $G_{\eta}$ over $V$ we have that 
$$ V[G_{\eta}] \models \dot{\forceP}_{[\eta, \alpha)} = \{ p \upharpoonright [\eta, \alpha) \, : \, p \in \forceP_{\alpha}, \,
p \upharpoonright \eta \in G_{\eta} \} $$
With these clarified notions one can show that

\begin{Fact}
Every iteration
$(\forceP_{\beta}, \dot{\forceQ}_{\beta}  \, : \, \beta < \alpha)$ 
is forcing equivalent to the according two step iteration 
$(\forceP_{\beta}, \dot{\forceQ}_{\beta} \, : \, \beta < \eta)$ $\ast$ $\dot{\forceP}_{[\eta,\alpha)}$.
And the forcing  $\dot{\forceP}_{[\eta,\alpha)}$ is forced to be isomorphic to an iteration of length $\alpha- \eta$, i.e.
$$ 1 \Vdash_{\eta} \dot{\forceP}_{[\eta,\alpha)} \text{ is isomorphic to an iteration of length } \alpha- \eta$$
       
\end{Fact}

One must be careful with the second statement of the last fact. 
Although for an iteration $(\forceP_{\beta}, \dot{\forceQ}_{\beta}  \, : \, \beta < \alpha)$ the tail $\dot{\forceP}_{[\eta,\alpha)}$ 
can be seen as an iteration in $V^{\forceP_{\eta}}$, the type of the iteration can change.
As an example take a countable support iteration
$(\forceP_{\beta}, \dot{\forceQ}_{\beta}  \, : \, \beta < \alpha)$ and an intermediate
 stage $\eta < \alpha$ such that in $V^{\forceP_{\eta}}$ the cofinality of an ordinal 
$\delta < \alpha$, $\eta < \delta$ gets changed from uncountable to countable cofinality. 
Then the cut off iteration with ground model $V^{\forceP_{\eta}}$ will not be a countable 
support iteration anymore as at stage $\delta$, which has countable cofinality, 
nevertheless the direct limit is taken. This situation can not happen when the 
factors of the iteration are proper and a countable support iteration is taken, 
as in a generic extension obtained with a proper forcing all new countable sets 
of ordinals can be covered by old countable sets.
The above described phenomenon can be exploited to construct an example of 
a countable support iteration of forcings which does collapse $\omega_1$, 
yet its factors are semiproper. In order to change the cofinality of an 
uncountable regular cardinal to cofinality $\omega$ we use Namba's forcing:

\begin{Definition}
The partial order $Nm(\omega_2)$, the Namba forcing, 
is defined like this: conditions are perfect trees 
$T \subset (\omega_2)^{< \omega}$, where perfect 
means that every node has $\aleph_2$-many extensions. 
The ordering is given by $T_1 \le T_2$ if and only if $T_1 \subset T_2$.
\end{Definition}

One can show that under $\CH$ forcing with $Nm(\omega_2)$ 
preserves $\aleph_1$ and changes the cofinality of 
$(\omega_2)^{V}$ to $\omega$. Moreover if one assumes 
the existence of a measurable cardinal, then after forcing with an $\sigma$-closed, thus $\CH$ preserving notion of forcing
the Namba forcing is semiproper. Nevertheless countable support iterations of $\omega_1$-preserving
notions of forcing, starting with a Namba forcing
can collapse $\omega_1$:

Assume as a ground model a model where Namba forcing is semiproper 
and changes the cofinality of $\omega_2$ to $\omega$. 
Let $(\forceP_{\alpha}, \dot{\forceQ}_{\alpha} \, : \, \alpha < \omega_2 )$ 
be a countable support iteration of length $\omega_2$ such 
that the first step in the itertation $\forceP_0$ is the Namba forcing.
 Further assume that at each stage $\alpha < \omega_2$, $V^{\forceP_{\alpha}}$ 
thinks that $\dot{\forceQ}_{\alpha}$ is an antichain of size $\aleph_1$.
Work in $V^{\forceP_{0}}$, fix a cofinal function $f: \omega \rightarrow \omega_2$ and for each $\alpha < \omega_2$
a $\forceP_0$-name of an antichain $A_{\alpha} \subset \dot{\forceQ}_{\alpha}$ of size $\omega_1$ in $V^{\forceP_{\alpha}}$.
List $A_{\alpha} = (a^{\alpha}_{\beta} \, : \, \beta < \omega_1)$ and define a name for a function $\dot{g}$ as follows: \begin{enumerate}
\item[] Let $\dot{g}$ be chosen such that for every $n \in \omega$ and every $\beta < \omega_1$ the Boolean value
$\llbracket \dot{g}(n) = \beta \rrbracket = 1 \smallfrown a^{f(n)}_{\beta} \smallfrown 1$, I.e. at the $f(n)$-th
coordinate, the value is $a^{f(n)}_{\beta}$, while constantly $1$ everywhere else.

\end{enumerate}

Now $\dot{g}$ is forced to be a surjective function from $\omega$ to $\omega_1$. 
Indeed let $p \in \forceP_{\omega_2}$ be a condition with support $\gamma < \omega_2$ and let $\beta < \omega_1$ be an arbitrary ordinal.
Then there is
an $n \in \omega$ such that $f(n) > \gamma$ and so $p$ and $a^{f(n)}_{\beta}$ are compatible.
Therefore for every ordinal $\beta < \omega_1$ it is dense to be in the range of $\dot{g}$ and so it is surjective on $\omega_1$.

Thus one cannot hope in general that countable support iterations of semiproper 
forcings will preserve $\omega_1$. One has to iterate semiproper forcings in a 
more careful way if one wants to preserve $\omega_1$.
Obviously the above example relies on the fact that after the Namba forcing the 
cofinality of $\omega_2^V$ has changed to $\omega$, nevertheless at stage $\omega_2^V$, 
due to the definition of the countable support iteration, we take the direct limit of the 
previous factors, though we actually should have taken the inverse limit. If we allow for
 our conditions in the iteration
 $(\forceP_{\alpha}, \dot{\forceQ}_{\alpha} \, : \, \alpha < \omega_2 )$ $\forceP_{\beta}$-names 
of countable sets, instead of the usual countable support, thus using a revised countable support,
then the example above will not define a surjection of $\omega$ to $\omega_1$ anymore.
Indeed this modification will completely rule out the possibility of collapsing $\omega_1$ 
as revised countable support iterations of semiproper forcings result in a semiproper forcing again.

\begin{Definition}
Let $\forceP_{\beta}$ be a forcing for every $\beta < \alpha$, $\alpha$ a limit ordinal. Then $\forceP_{\alpha}$ 
is an RCS-limit (short for revised countable support) of   
$\forceP_{\beta}$, $\beta < \alpha$ if it is a subset of the inverse 
limit of the forcings $(\forceP_{\beta} \, :\, \beta < \alpha)$ such that each $ p \in \forceP_{\alpha}$ satisfies
\begin{itemize}
\item for each $q < p$ there is an ordinal $\gamma < \alpha$ and a 
$\forceP_{\gamma}$-condition $r$ such that 
$r \le q \upharpoonright \gamma$ and in the 
forcing $\forceP_{\gamma}$ it holds that
$r \Vdash_{\gamma} cf(\alpha)= \omega$ or for each 
$\beta \ge \gamma$ $p \upharpoonright [\gamma, \beta) \Vdash_{\forceP_{\gamma, \beta}} p(\beta)=1$ \end{itemize}
 
\end{Definition}

\begin{Fact}
Iterations with RCS-support whose factors are semiproper result in a semiproper forcing notion. Moreover
if we split an RCS iteration into two pieces then the tail iteration, as seen from the intermediate model
will look like an RCS iteration again. More precisely, if $(\forceP_{\alpha}, \dot{\forceQ}_{\alpha} \, : \, \alpha < \beta)$
is an RCS iteration then $1 \Vdash_{\gamma} \dot{\forceP}_{[\gamma , \beta)}$ is an RCS-itertaion, for every $\gamma < \beta$.
\end{Fact}
For $\Stat$-semiproper notions 
of forcing $\forceP$ we can still infer that they preserve stationary 
subsets of $\omega_1$ as long as for every stationary 
$S \subset \omega_1$, $\{ M \cap \omega_1 \, :\, M \in \Stat\} \cap S$ is stationary.
Indeed if $S \subset \omega_1$ is stationary, if $p \in \forceP$, 
and if $\dot{C}$ denotes a name of a club, then we pick an elementary 
submodel $M \in \Stat$ such that $M \cap \omega_1 \in S$ which contains $\dot{C}$.
Let $q < p$ be an $(M, \forceP)$-generic condition then $q \Vdash M \cap \omega_1 \in \dot{C}$ and
as $M \cap \omega_1 \in S$ we are done.

It is well known that properness is equivalent to the 
preservation of stationary subsets of $[\lambda]^{\omega}$. 
As a consequence $\Stat$-proper forcings preserve stationarity of
$\Stat$ in every $[H_{\theta} \cap V]^{\omega}$. The class $\Stat$ itself
will nevertheless lose its stationarity as soon as we only add one new element $x$ to the universe
as $\{ M \in H_{\theta} \, : \, x \in M \}$ will be a club disjoint from $\Stat$.
If we want to find a sufficient condition for a two step iteration $\forceP_0 \ast \forceP_1$ to be 
$\Stat$-proper, we therefore have to change $\Stat$ after the first stage.

\begin{Definition}
Let $S \subset [H_{\theta}]^{\omega}$ be stationary and let $\forceP$ be an 
arbitrary notion of forcing with $2^{|\forceP|} < \theta$. Then set $$ S[G] := \{ M[G] \, : \, \forceP \in M \in S \} $$
\end{Definition}

This set $S[G]$ will remain a stationary set in $V[G]$:

\begin{Lemma}
Let $S \subset [H_{\theta}]^{\omega}$ be stationary and let $\forceP$ be 
an arbitrary notion of forcing with $2^{|\forceP|} < \theta$ then $S[G]$ 
is a stationary subset of $V[G]$'s version of $[H_{\theta}]^{\omega}$.
\end{Lemma}

\begin{proof}
Fix a name $\dot{C}$ for a club in $[H_{\theta}]^{\omega}$. Pick a countable $M \prec H_{\theta^+}$ such that $M$ contains $\dot{C}$
and such that $M \cap H_{\theta} \in S$. Then $\dot{C}^{G}$ is a club and is an element of $M[G]$, therefore
$M[G] \cap H_{\theta} = (M \cap H_{\theta})[G] \in \dot{C}^{G} $, yet $(M \cap H_{\theta})[G] \in S[G]$.
\end{proof}

Now this enables us to handle two step iterations:

\begin{Fact}
Let $S \subset H_{\theta}$ be stationary and assume that $\forceP$ is a 
forcing notion with $2^{|\forceP|} < \theta$ which is $S$-proper. 
Let $\forceQ$ be a notion of forcing in $V[G]$ which is $S[G]$-proper, then the iteration $\forceP \ast \forceQ$ is $S$-proper.
\end{Fact}

The last Proposition leads to

\begin{Fact}
Let $(\forceP_i, \dot{\forceQ}_i \, :\, i \in \lambda)$ be an RCS-iteration of forcings for which \begin{enumerate}
\item[] $1 \Vdash_i \forceQ_i$ is $\Stat[\dot{G}_{i}]$-semiproper \end{enumerate}
holds at each stage $i < \lambda$.
Then $\forceP_{\lambda}$ is an $\Stat$-semiproper forcing notion.
  
\end{Fact}

\section{The canonical inner Model with one Woodin Cardinal}

Constructing universes of set theory with definable wellorders on the reals 
becomes a reasonable task only when the underlying ground model already
satisfies a certain amount of definability which can be exploited.
In the presence of sufficiently 'small' large cardinals 
G\"odel's constructible universe $L$ is the ultimate candidate
for such a ground model, handing to the mathematician 
a lot of well studied tools to examine its structure thoroughly.
However as soon as we investigate properties of the set theoretic universe 
which are only achievable using stronger large cardinal hypotheses 
$L$ will not suffice anymore as these hypotheses are inconsistent with
$V=L$. The statement $\NS$ is $\aleph_2$-saturated is exactly one of such
properties which need a quite large cardinal, namely a Woodin cardinal,
whose existence contradicts $V=L$. We are thus compelled to
use a model which is not $L$, yet having some of its nice
properties, plus containing a Woodin cardinal. 
Of course Inner Model Theory equips the set theorist with such a model, named $M_1$ which 
shall be introduced here briefly. 

As Inner Model Theory relies heavily on a lot of 
nontrivial notions and definitions, whose introduction would
soak up a lot of space we will skip them and assume the reader is already familiar with
concepts such as 
\begin{itemize}
 \item the notion of an extender, its length and support, 
 ultrapower constructions and their relation to elementary embeddings of the universe,
 \item Jensen's fine structural hierarchy,
 \item the notion of acceptability,
 \item fine extender sequences,
 \item the defintion of active and passive premice,
 \item the notion of projecta, universality, solidity and soundness,
 \item iteration games played on sufficiently iterable premice, thus generating an iteration tree
 and the notion of branches in the iteration tree which drop in model or degree,
 \item the comparison process of sufficiently iterable premice,
 \item the definition of a Woodin cardinal,
 \item plus other notions we probably have forgot to mention.
\end{itemize}
One can find these concepts introduced e.g. in \cite{FSIT},
\cite{Steel1}, \cite{LS}.
It should be noted however that a reader who does not know any of these can skip to the end of the 
section while  
still being able to fully understand the thesis, keeping in mind a couple of 
properties of $M_1$ which are mentioned as a list of properties stated as Fact \ref{axioms}, and which can be seen 
as additional axioms we will use during our proofs.

Now to the definition of $M_1$: recall first that
a premouse $\mouseM$ is called tame whenever
its extenders do not overlap a local Woodin cardinal, 
i.e. whenever $E$ is an extender on the $\mouseM$ sequence 
and $\lambda= lh(E)$ then $$\mathcal{J}^{\mouseM}_{\lambda} \models \forall \xi > crit(E) (\xi \text{ is not Woodin}).$$
We define recursively a sequence of premice 
$\mouseN_{\xi}$ starting with $$\mouseN_0 := (V_{\omega}, \in, \emptyset, \emptyset).$$ 
Suppose now that the premouse $\mouseN_{\xi}$ has already been defined. Then we consider 
the $\omega$th core of $\mouseN_{\xi}$ and stop the construction if it does not exist.
If it exist we set $\mouseM_{\xi} := \mathcal{C}_{\omega} (\mouseN_{\xi})$ and split into cases:
\begin{enumerate}
 
\item if $\mouseM_{\xi} =(J^{\vec{E}}_{\gamma}, \in, \vec{E}, \emptyset)$ is passive and 
there is an extender $F^{\ast}$ over $V$, further an extender $F$ over $\mouseM_{\xi}$ and 
an ordinal $\nu < \gamma$ such that
$$V_{\nu + \omega} \subset Ult(V,F^{\ast} \text{ and } 
F\upharpoonright \nu = F^{\ast} \cap ([\nu]^{<\omega} \times J^{\vec{E}}_{\gamma}).$$
Assume further that the structure $(J^{\vec{E}}_{\gamma}, \in , \vec{E}, F)$ is a 
tame premouse then pick the least such $\nu$
and set $$\mouseN_{\xi+1} := (J^{\vec{E}}_{\gamma}, \in , \vec{E}, F)$$

\item if $\mouseM_{\xi}=(J^{\vec{E}}_{\gamma}, \in, \vec{E}, H)$ is not passive, (i.e. $H \ne \emptyset$) 
or there does not exist a background extender $F^{\ast}$, just continue in the $J$-hierarchy. 
Set $$\mouseN_{\xi+1} = (J^{\vec{E}^\smallfrown H}_{\gamma+1}, \in, \vec{E}^\smallfrown H, \emptyset).$$

\end{enumerate}
In the limit steps $\lambda$, we let $\omega \eta =lim$ $inf_{\xi < \lambda} (\rho_{\omega}^{+})^{\mouseM_{\xi}}$
and let $\mouseN_{\lambda}$ be the unique passive premouse $\mouseP$ of height $\omega \eta$ such that
for every $\beta < \eta$, $\mathcal{J}^{\mouseP}_{\beta}$ is the eventual value of 
$\mathcal{J}^{\mouseM_{\xi}}_{\beta}$ as $\xi \rightarrow \lambda$.

We can use the just described sequence of premice $$\mathbb{C}^t:=(\mouseN_{\xi} \, : \, \mouseN_{\xi} \text{ exists})$$
to define $M_1$, assuming the existence of a Woodin cardinal.
As a reminder let us introduce the notion of 1-smallness:
\begin{Definition}
Let $\mouseM$ be a premouse, then we say that $\mouseM$ is 1-small 
if the following holds: whenever $\lambda$ is the critical point 
of an extender on the $\mouseM$-sequence
then $$\mathcal{J}^{\mouseM}_{\lambda} \vDash \text{ there is no Woodin cardinal. }$$
\end{Definition}

Suppose first that for every $\xi$, $\mouseN_{\xi}$ is 1-small, then it is a result 
of J.Steel \cite{SteelManyWoodins} that for every $\xi < \Ord$, $\mouseN_{\xi}$ is defined and so is $\mouseN_{\infty}$. 
In this case we let $M_1$ be $\mouseN_{\infty}$ which is a class sized model which can be shown 
to contain exactly one Woodin cardinal.

Otherwise if there is a $\xi < \Ord$ such that $\mouseN_{\xi}$ is not 1-small then fix the 
least such $\xi$. $\mouseN_{\xi}$ is an active premouse. Let $M_1^{\#}$ be $\mathcal{C}_{\omega}(\mouseN_{\xi})$ 
and let $\mouseP$ be the result of iterating the last extender of $M_1^{\#}$ out of the universe. 
Set $M_1 := \mathcal{J}^{\mouseP}_{\infty}$, then again $M_1$ is a class sized model with one Woodin cardinal.
To summarize:
\begin{thm}
If there is a Woodin cardinal then in both cases $M_1$ is a class sized model with exactly one Woodin cardinal and
all its initial segments $\mathcal{J}^{M_1}_{\beta}$ are $\omega$-sound.
\end{thm}

We will always assume that $M_1^{\#}$ exists, i.e. $M_1^{\#}$ is the least mouse 
which is not 1-small and $M_1$ is the result of iterating away the last extender of $M_1^{\#}$. 

It is a wellknown fact that reals which are elements in a sufficiently
iterable premouse $\mouseM$ admit an easy (i.e. $\Delta^2_2$) definition
with a countable ordinal as parameter, using comparability of
mice. The formula for 
$r \in \mouseM$ just reads like this: $r$ is the $\alpha$-th
real in some sufficiently, i.e. $(\omega, \omega_1 +1)$-iterable premouse.
Thus the definability of reals in mice is tied to
the expressibility of 'sufficient' iterability conditions in the sense of iteration games.
For the model $M_1$, as was observed by
J. Steel in \cite{Steel2} an even easier definition is possible, due to a weakening
of the usual iteration game which is still enough to guarantee 
a certain amount of comparability. We say that 
a premouse $\mouseM$ is $\Pi^1_2$-iterable if player II has a 
winning strategy for $\mathcal{I}(\mouseM)$, where the latter denotes the 
new iteration game played on the premouse $\mouseM$.
Roughly speaking J. Steel showed that
\begin{itemize}
 \item if $\mouseM$ is a premouse which is embeddable into a model of the $\mathbb{C}^t$-sequence 
 then player II has a winning strategy for the $\mathcal{I}(\mouseM)$,
 \item if on the other hand II has a winning strategy for $\mathcal{I}(\mouseM)$
 played on the premouse $\mouseM$ then $\mouseM$ can be compared
 with any premouse $\mathcal{N}$ which embedds into an element of the $\mathbb{C}^t$-sequence,
 \item and finally the set of premice
 $\{ \mouseM \, : \,  \text{II has a winning strategy in } \mathcal{I}(\mouseM) \}$
 is $\Pi_1$ definable.
 \end{itemize}

We can use this to observe that a low complexity definition of countable (countable in $M_1$ that is)
initial segments of $M_1$ is possible in generic extensions $M_1[G]$ which preserve $\omega_1$. We consider
the set of countable premice which are $\Pi^1_{2}$-iterable, $\omega$-sound, 1-small, and which project to $\omega$.
$$B:= \{ \mouseM \text{ ctbl premouse} \, : \,\mouseM \text{ is } 
\Pi^{1}_2\text{-iterable}, \, \omega\text{-sound, } \text{1-small } \rho_{\omega}(\mouseM)= \omega\}.$$
If we are in a forcing extension $M_1[G]$
which preserves $\omega_1$ then using Shoenfield absoluteness, the set of countable 
premice which are initial segments of $M_1$ and which project to $\omega$ form a set 
of premice which are still $\Pi^1_{2}$-iterable in $M_1[G]$, $\omega$-sound and 1-small. 
If we consider in $M_1[G]$ an arbitrary element $\mouseM$ of $B$
and let $\mouseN = \mathcal{J}^{M_1}_{\eta}$, $\eta < \omega_1$ be an initial segment 
of $M_1$ which projects to $\omega$ then we can compare these two as $\mouseM$
is $\Pi_2^{1}$-iterable which suffices for comparison as mentioned above.
As both $\mouseM $ and $\mouseN$ are $\omega$-sound and $\omega$-projecting they actually 
do not move during the iteration and therefore we get that $\mouseM \trianglelefteq \mouseN$ 
or $\mouseN \trianglelefteq \mouseM$ must hold.
If we let the height of $\mouseN= \mathcal{J}^{M_1}_{\eta}$ vary we see that there is certainly 
an $\eta < \omega_1$ such that
$\mouseM \trianglelefteq \mouseN=\mathcal{J}^{M_1}_{\eta}$. Thus the set $B$ defines in $M_1[G]$ 
a set of initial segments of $M_1$, which is cofinal in that for every countable $\mouseM \trianglelefteq M_1$ 
there is an $\mouseN \in B$ such that $\mouseM \trianglelefteq \mouseN$. As $\Pi^{1}_2$-iterability 
is a $\Pi^{1}_2$-notion, $B$ is itself $\Pi_2^{1}$-definable. Thus we have shown: 
\begin{Lemma}
Let $M_1[G]$ be an $\omega_1$-preserving forcing extension of $M_1$. Then in $M_1[G]$ there is
$\Pi^{1}_2$-definable set $B$ of premice  which are of the form $\mathcal{J}^{M_1}_{\eta}$ for 
some $\eta< \omega_1$. $B$ is defined as
$$B:= \{ \mouseM \text{ ctbl premouse} \, : \,\mouseM \text{ is } \Pi^{1}_2\text{-iterable}, \, 
\omega\text{-sound} \text{ and projects to } \omega \},$$
and the set
$$\{ \eta < \omega_1 \, : \, \exists \mouseN \in B (\mouseN = \mathcal{J}^{M_1}_{\eta})\}$$ is cofinal in $\omega_1$.
\end{Lemma}

Next we want turn to the concept of generic absoluteness. The motivation
for this will become clear later once we turn to the proofs of the thesis.
For now we just state that we would want our ground model $V$ 
to be definable in generic extensions $V[G]$, for 
$G$ a generic filter for a poset $\forceP$, i.e.
we need a first order formula $\Phi(x)$ such that 
the class $\{ x \,:\, V[G] \models \Phi(x) \}$ outputs $V$ again.
To achieve this we proceed indirectly using 
Steel's core model $K$. As its definition is 
very involved we will skip it (the interested reader can find it in \cite{Schimmerling}) and just state
that it is so defined that the
class $K^c$ is an iterate of $K$, or dually put 
$K$ is a Skolem hull of $K^c$. The proof
of its existence initially used a measurable
cardinal as well, but eventually
R. Jensen and J. Steel found a way around that 
additional hypotheses.
\begin{thm}
Assume that there is no transitive class model satisfying $\ZFC$ and ``there is a Woodin cardinal''. Then
there is a $\Sigma_2$-formula $\Phi(x)$ such 
that $$K:=\{ x \, : \, \Phi(x) \text{ holds} \}$$ is a 
class model of $\ZFC$, which is iterable. If $\forceP$ 
is a notion of forcing of set size and $G$ is a $V$-generic 
filter then $K^{V[G]} = K^{V}$, thus $K$ is absolute for set sized forcing extensions.
\end{thm}

Note that as soon as we consider $M_1$, cut it at the Woodin cardinal
$\delta$ we are in the situation of the anti large cardinal assumption
and thus can build $K$ in $\mathcal{J}^{M_1}_{\delta}$.
The model we end up with is again $\mathcal{J}^{M_1}_{\delta}$ which
was proved by J.Steel (see \cite{Schindler2} for a proof of this).
Thus we do have a certain amount of generic absoluteness in 
$M_1$ namely there is a formula $\Phi(x)$ which
defines $K$ and thus $\mathcal{J}^{M_1}_{\delta}$
for notions of forcing which have size less than $M_1$'s Woodin
cardinal $\delta$.

We turn now to condensation of sufficiently iterable premice. 
The central result is the following which was proved by J. Steel 
and I. Neeman building on the works of Jensen, S.Friedman, Dodd, 
Mitchell and Schimmerling: 
\begin{thm}
Let $\mouseM$ be an $\omega$-sound, $(\omega, \omega_1, \omega_1 +1)$-iterable premouse.
Let $\pi :\mathcal{H} \rightarrow \mouseM$ be a fully elementary map such that its critical point $cp(\pi)=\rho_{\omega}^{\mathcal{H}}$.
Then either
\begin{enumerate}
\item $\mathcal{H}$ is a proper initial segment of 
$\mouseM$ \item there is an extender $E$ on the $\mouseM$-sequence 
such that $lh(E) = \rho_{\omega}^{\mathcal{H}}$ and $\mathcal{H}$ is a proper initial segment of $Ult_{0}(\mouseM, E)$
\end{enumerate}

\end{thm}
A useful observation is the following:
\begin{Lemma}\label{condensation}
Let $\mouseM$ be as in the theorem above. Consider the set $S$ of countable elementary submodels
of $\mouseM$ whose transitive collapses are proper initial segments of $\mouseM$.
Then $S$ is stationary in $[\mouseM]^{\omega}$.
\end{Lemma}
\begin{proof}
Assume for a contradiction that $S$ is not stationary then 
there is a club $C$ which is $<$-least in the definable wellorder of the mouse, such that every 
element $H$ of it does not collapse to an initial segment of $\mouseM$. 
Let $\mouseN \vartriangleright \mouseM$ be a bigger
$\omega$-sound, $(\omega, \omega_1, \omega_1 +1)$-iterable premouse which sees that $C$ is the $<_{\mouseN}$-least such club. 
Let $P \prec_1 \mouseN$ be a countable $\Sigma_1$-substructure which
is $\omega$-sound, created as the $\Sigma_1$-hull of $\{ \mouseM \}$ in
$\mouseN$. Then $C \in P$ by elementarity and $P \cap [\mouseM]^{\omega} \in C$ as
$C$ is closed and unbounded. But now the collapse of $P \cap [\mouseM]^{\omega}$
is an initial segment of $\pi(P)$. Yet $\pi(P)$ does collapse nicely as a comparison argument shows.
Once we compare $\pi(P)$ with an $\omega$-sound, $\omega$-projecting initial segment of $\mouseM$
nothing will move during the comparison process and so $\pi(P)$ is an initial segment of $\mouseM$.
Thus $\pi(P) \cap [\mouseM]^{\omega}$ is an initial segment, yet $\pi(P) \cap [\mouseM]^{\omega} \in C$
which is a contradiction.      

\end{proof}

To summarize the above results we formulate an explicit list of 
axioms which $M_1$ satisfies and which a reader who does not know about
inner model theory should keep in mind when reading this thesis.

\begin{Fact}\label{axioms}
 $M_1$ satisfies the following list of axioms
\begin{enumerate}
 \item There is a $\Pi^1_{2}$-definable set of reals $I$
 whose elements are codes for countable initial segments of $M_1$. Moreover these codes
are cofinal meaning that for every countable initial segment $\mouseP$ of $M_1$ there is a code $c$ in $I$ such that
$\mouseN$ is an initial segment of $c$. Further this 
 set still works in all $\omega_1$-preserving forcing extensions
 $M_1[G]$ of $M_1$.
 \item In $M_1$ below its Woodin cardinal Steel's core model $K$ can be constructed
 and coincides with $\mathcal{J}^{M_1}_{\delta}$. The definition
 of $K$ is generically absolute for set forcing extensions of size less than $\delta$.
\item If $\mathcal{J}^{M_1}_{\eta}$ is an initial segment of $M_1$, hence 
      an $\omega$-sound, $(\omega, \omega_1, \omega_1 +1)$-iterable premouse then
      there is a stationary subset $S \subset [\mathcal{J}^{M_1}_{\eta}]^{\omega}$ such that every $H \in S$ condenses
      to an initial segment of $\mathcal{J}^{M_1}_{\eta}$.
\end{enumerate}

\end{Fact}

\section{How to make $\NS$ $\aleph_2$-saturated}

The investigation of the nonstationary ideal on a regular cardinal 
has a long history in set theory which is no wonder as stationarity 
represents one of its most fundamental notions.
The question of the length of antichains of stationary subsets modulo 
nonstationarity generated particular interest due to their central role in generic ultrapower arguments.

\begin{Definition}
Let $\kappa$ be a regular cardinal and $I$ an ideal on $\kappa$.
For a regular cardinal $\lambda$ we say that $I$ is $\lambda$ saturated 
if there are no antichains of length $\lambda$ in $P(\kappa) \backslash I$,
where antichains are meant to be modulo $I$-small intersections of their elements.
\end{Definition}
An equivalent way of saying that $I$ is $\lambda$-saturated 
is therefore the statement that the Boolean
algebra $P(\kappa) \slash I$ has the $\lambda$-cc, 
which highlights the importance of the notion in the context of generic ultrapowers where conditions are
elements of $I$-positive sets ordered by the
subset relation.

The question of possible values for antichains modulo $I$ is 
intimately tied with large cardinals as was already implicitly 
evident in S. Ulam's work on measurability, showing that for no $\lambda$ there is a $\lambda^+$-saturated ideal on $\lambda^+$.
Forty years later R. Solovay showed that
for a regular uncountable cardinal $\kappa$, every stationary 
$S \subset \kappa$ can be partitioned into $\kappa$ many stationary 
sets which implies that for every stationary $S\subset \kappa$ the 
restricted nonstationary ideal $\hbox{NS}_{\kappa} \upharpoonright S$ can not be $\kappa$-saturated.
Thus naturally the question arises
whether there are successor cardinals $\kappa$ such that the nonstationary ideal on $\kappa$ is $\kappa^+$-saturated.
Here a crucial difference between the nonstationary ideal on $\aleph_1$ and 
the nonstationary ideal on other regular $\kappa > \aleph_1$ shows up, which has its
deeper reasons in the trivial fact that below $\omega_1$ limit ordinals have only one possible cofinality,
while at bigger cardinals more possibilities occur.
It is a $\ZFC$-theorem of M. Gitik and S. Shelah that $\hbox{NS}_{\kappa}$ can 
not be $\kappa^+$-saturated for any $\kappa > \omega_1$, for $\kappa =\omega_1$ however, $\NS$ can be
$\aleph_2$-saturated, which was shown first by K. Kunen assuming the existence of a huge cardinal.

The ultimate solution to the problem of the consistency of the statement 
$\NS$ is $\aleph_2$-saturated from optimal large cardinal assumptions was eventually found by S. Shelah who showed
around 1985 that already a Woodin cardinal suffices for the consistency 
of $\NS$ being saturated. As this result and its proof are essential for our work we will
give the proof in this section in detail.
As a concluding remark we mention that  it was only in 2006 that 
R. Jensen and J. Steel proved that the assumption of a Woodin cardinal 
is in fact sharp in terms of consistency strength via showing that 
if the theory $\ZFC + \text{\textquotedblleft $\NS$ on } \omega_1 
\text{ is saturated\textquotedblright}$ is consistent then so 
is $\ZFC + \text{\textquotedblleft there is a Woodin cardinal \textquotedblright }$

We start with the preparations for Shelah's result.
The material in this section draws heavily from
R. Schindler's notes \cite{SchindlerNS} on the problem.
Recall first a couple of definitions and facts.

\begin{Definition}
Let A be an arbitrary set then a cardinal $\kappa$ is $A$-strong up to the cardinal $\delta$ iff $\forall \gamma < \delta \exists j :
V \rightarrow M$ which is elementary
such that
\begin{enumerate}
\item crit $j$ = $\kappa \land \gamma < j(\kappa)$ \item $V_{\kappa + \gamma} \subset M$
\item $A \cap V_{\kappa + \gamma} = j(A) \cap V_{\kappa + \gamma}$ \end{enumerate}
 
\end{Definition}

The following fact can be used to define Woodin cardinals:

\begin{Fact}
The following are equivalent
\begin{itemize}
\item $\delta$ is Woodin
\item For any $A\subset V_{\delta}$,
$$ \{ \alpha < \delta \,:\, \alpha \text{ is $A$-strong up to $\delta$ } \}$$ is stationary in $\delta$.
\end{itemize}

\end{Fact}

We will need a bit more, namely a Woodin cardinal with a $\diamondsuit$-sequence living below it:

\begin{Definition}
Let $\delta$ be a Woodin cardinal then we say that $\delta$ is Woodin 
with $\diamondsuit$ iff there is a sequence $(a_{\kappa} \, :\, \kappa < \delta)$ 
such that for each $\kappa$, $a_{\kappa} \subset V_{\kappa}$ and for every 
$A \subset V_{\delta}$ the set $$\{ \kappa < \delta \,:\,
 A \cap V_{\kappa}= a_{\kappa} \land \kappa \text{ is $A$-strong up to } 
\delta \}$$ is stationary in $\delta$.
 
\end{Definition}

In terms of consistency strength this adds 
nothing to being a Woodin cardinal. 
If $\delta$ is the Woodin cardinal 
and we force with $\delta$-Cohen forcing 
then in the resulting generic extension we 
have that $\delta$ is Woodin with $\diamondsuit$: 
indeed first note that $\delta$-Cohen forcing is 
the same as forcing with conditions of the form 
$(a_{\alpha} \, : \, \alpha < \kappa < \delta \, \land \, a_{\alpha} \subset V_{\alpha})$, ordered by end extension.
Assume now that there is a condition $p$ such 
that $$p \Vdash \tau \subset V_{\delta} \, \land \, \sigma \subset \delta \text{ is club in } \delta.$$ 
We have to show that there is a stronger condition $q < p$, $q = (a_{\alpha} \, : \, \alpha < \lambda)$ 
and a cardinal $\kappa < \delta$ for which

$$q \Vdash  \kappa \in \sigma \text{ is } \tau\text{-strong up to } \delta \, \land \tau \cap \kappa = a_{\kappa}.$$

We construct by recursion a descending sequence of conditions $(p_{\kappa} \, : \, \kappa < \delta )$
such that the length of each $p_{\kappa}$ is $\mu_{\kappa}$ and such that the following points are obeyed:
\begin{enumerate}
\item $\{ \mu_{\kappa} \,: \, \kappa < \delta\}$ is a club in $\delta$.
\item For every $\kappa$ there is some $C_{\kappa} \subset \mu_{\kappa}$ which is unbounded in $\mu_{\kappa}$
such that $p_{\kappa} \Vdash \sigma \cap \mu_{\kappa} = C_{\kappa}$ and consequently $p_{\kappa} \Vdash \mu_{\kappa} \in \sigma$.
\item For every $\kappa$ there is some $A_{\kappa} \subset V_{\kappa}$ such that $p_{\kappa} \Vdash \tau \cap V_{\mu_{\kappa}} = A_{\kappa}$
\item for every $\kappa$, $p_{\kappa} \Vdash a_{\mu_{\kappa}} = 
A_{\kappa}$ \item If $p_{\kappa +1} \nVdash $ $\kappa$ is $\tau$-strong up to $\delta$, then there is an
$\alpha < \mu_{\kappa +1}$ such that
$$p_{\kappa +1} \Vdash \kappa \text{ is not } \tau \text{-strong up to } \alpha $$ (Note here that $\kappa+1$
makes such a choice always possible. It is impossible for $\kappa$ limit) \end{enumerate}
That such a sequence exists is easily seen. Now set $A := \bigcup_{\kappa < \delta} A_{\kappa}$, and as $\delta$ is Woodin
and the set $\{ \mu_{\kappa} \,: \, \kappa < \delta\}$ is a club in $\delta$ we find a point $\kappa = \mu_{\kappa}$
such that $\kappa$ is $A$-strong up to $\delta$. Now pick the condition $q:= (a_{\lambda} \, : \, \lambda < \kappa +1)$
which by 2,3 and 4 satisfies
$$q \Vdash \kappa \in \sigma \land \tau \cap \kappa = a_{\kappa}.$$ 
To finish we need to show that also $q \Vdash \kappa \text{ is } \tau \text{-strong up to } 
\delta$. Assume for a contradiction the opposite then by property 5 there 
exists an $\alpha < \mu_{\kappa +1}$ such that $$p_{\kappa+1} \Vdash \kappa 
\text{ is not } \alpha \text{-strong up to } \delta.$$ But the elementary 
embedding $j: V \rightarrow M$ which witnesses that $\kappa$ is $A$-strong 
up to $\delta$ can be lifted to $j' : V[G] \rightarrow M'$, and by the 
$\delta$-closure of the forcing it will witness that $\kappa$ is $A$-strong up to $\delta$, a contradiction.

The usage of the $\diamondsuit$-sequence at the Woodin cardinal is crucial 
for the proof of the existence of a model where $\NS$ is $\aleph_2$-saturated.
We will use it as a guideline for the iteration. Whenever we hit a stage 
$\alpha$ such that the $\diamondsuit$-sequence $(a_{\beta} \, : \, \beta < \delta)$ 
at stage $\alpha$ is the name of a maximal antichain of stationary subsets of 
$\omega_1$ of length $\omega_2$ we want to change its length to 
$\aleph_1$, and further ensure that this maximal antichain remains maximal 
in all stationary subsets of $\omega_1$-preserving outer models. The next 
forcing notion does exactly what we demand:

\begin{Definition} Assume that $\vec{S}$ is an antichain of stationary 
subsets of $\omega_1$. Then the so called sealing forcing $\seal$ 
consists of conditions of the form $(p,c)$ where $p : \alpha+1 \rightarrow \vec{S}$ 
is a function and $c : \alpha + 1 \rightarrow \omega_1$ is a function with closed image and such that
$$ \forall \xi \le \alpha (c(\xi) \in \bigcup_{i \in \xi} p(i))$$ holds. We 
let $(q,d) < (p,c)$ if $q$ and $d$ end-extend $p$ and $c$ respectively.
\end{Definition}
It is well known that the sealing forcing $\seal$ is $\omega$-distributive 
and preserves all stationary subsets of $\vec{S}$, thus $\seal$ is stationary 
subsets of $\omega_1$ preserving if $\vec{S}$ is maximal.

We can also consider a stationary, co-stationary set $A \subset \omega_1$ 
and the nonstationary ideal restricted to subsets of $A$, $$\NSA := 
\{ B \subset A \,:\, B \text{ is stationary }\}.$$ It is natural to 
ask the same question we asked for the full ideal $\NS$, for the 
restricted version $\NSA$, namely whether it can be $\aleph_2$-saturated.
Surprisingly its positive answer has a simpler structure than for full $\NS$.
The reason behind this is that the set $A$ enables us to see that 
the according sealing forcings are particularly nice, namely $A^c$-proper,
which can be used to see that the usual proof for making $\NS$ saturated can be 
done even without adding reals.
\begin{Definition}
Assume that $A\subset \omega_1$ is a
stationary, co-stationary subset of $\omega_1$ and that $\vec{S}_{A}$ 
is an antichain of stationary subsets of $A$. Then we can also seal it 
off using the straightforward generalization, denoted by $\seal_A$ of 
the already introduced sealing forcing $\seal$. Conditions of $\seal_A$ 
are pairs $(p,c)$ where $p: \alpha+1 \rightarrow \vec{S}_A $ is a function 
and $c: \alpha+1 \rightarrow \omega_1$ is a function with range a closed subset of $\omega_1$ such that
$$ \forall \xi \le \alpha (c(\xi) \in A \rightarrow  c(\xi) \in \bigcup_{i \in \xi} p(i))$$

\end{Definition}

The idea now to force $\NS$ to be $\aleph_2$-saturated is to seal 
off one by one all the long antichains of stationary subsets with an iteration.
The first obstacle one immediately encounters is how we catch our 
tail during this process, for which the usage of the Woodin cardinal 
will be crucial. Another difficulty is the following:
One has to ensure that the iteration stays stationary set preserving 
to avoid utter chaos. The sealing forcing $\seal$ is stationary set
 preserving as long as $\vec{S}$ is maximal, but there does not 
exist a theory for iterations of stationary set preserving forcings
 in a stationary set preserving way. Thus we are compelled to work 
with semiproper notions of forcing and $RCS$-iterations instead.
We demand to only seal off a maximal antichain $\vec{S}$ when the forcing $\seal$ is
also semiproper. But this leaves us with the possibility of not sealing off every long antichain during the iteration.
That these difficulties still do not ruin the proof is outlined here in detail:
\begin{thm} Assume that $\delta$ is a Woodin cardinal with $\diamondsuit$. 
Then there exists a semiproper forcing $\forceP$ of size $\delta$ such 
that in $V[G]$, $\NS$ is $\aleph_2$-saturated and $\delta = \aleph_2$.
\end{thm}
\begin{proof}

We use a $\diamondsuit$-sequence on $\delta$ to determine at each stage $\alpha< \delta$ which forcing to use.
Assume that during our iteration we have arrived at stage $\alpha$. Then let $\forceQ_{\alpha}$ be \begin{enumerate}
\item the sealing off forcing of a maximal antichain $\sigma^{G_{\alpha}}$ 
if the diamond sequence at stage $\alpha$ is the $\forceP_{\alpha}$-name 
$\sigma$ for a maximal antichain of stationary subsets of $\omega_1$ and the sealing forcing is semiproper.

\item the collapse of $2^{\aleph_2}$ to $\aleph_1$ else.
\end{enumerate}
At limit stages we use the $RCS$-limit. This already suffices.
Assume for a contradiction. that $\NS$ is not $\aleph_2$-saturated in 
$V[G]$, I.e. there is a maximal antichain $\vec{S}=(S_i : i< \omega_2)$ 
in $P(\omega_1)/ \NS$. Let $\tau$ be a $\mathbb{P}$-name for the sequence. 
As $V[G] \models \aleph_2 = \delta$ for our Woodin cardinal $\delta$, we 
claim that it is possible to find an inaccessible $\kappa$ below $\delta$ 
such that the following three properties hold: \begin{enumerate}
\item $\kappa$ is $\mathbb{P} \oplus \tau$-strong up to $\delta$ in 
$V$ \item $\kappa = \omega_2^{V[G\upharpoonright \kappa]}$ \item 
$\vec{S}\upharpoonright \kappa = (S_i \, : \, i < \kappa) = 
(\tau \cap V_{\kappa})^{G\upharpoonright \kappa}$ is the maximal antichain
in $V[G\upharpoonright \kappa]$ which is picked by the $\diamondsuit$-sequence at stage $\kappa$.
\end{enumerate}
This is clear as we can assume that our $\diamondsuit$-sequence 
lives on the stationary subset of inaccessible cardinals below $\delta$, 
and for all inaccessible $\kappa$ property $2$ automatically holds. Moreover the sets

$$C_1:=\{ \kappa < \delta \, :\, \vec{S}\upharpoonright \kappa = (S_i \, 
:\, i< \kappa)= (\tau \cap V_{\kappa})^{V[G\upharpoonright \kappa]} \}$$ and

$$C_2:=\{ \kappa < \delta \, : \forall \alpha < \kappa \forall S \in P(\omega_1) \cap V^{\mathbb{P_{\alpha}}} \text{ stationary }
\exists \bar{S} \in \vec{S}\upharpoonright
\kappa (S \cap \bar{S} \notin NS) \}$$
are both clubs, therefore hitting the stationary
set $T$ consisting of the points $\kappa < \delta$
where $\tau \cap V_{\kappa} = a_{\kappa}$ (remember: $a_{\kappa}$ is the 
$\kappa$-th element of the $\diamondsuit$-sequence $(a_{\alpha} : \alpha < \delta )$) 
and $\kappa$ is $\tau$-strong up to $\delta$. Thus if $\kappa$ is in the nonempty intersection
$C_1 \cap C_2 \cap T$ then $1$ and $2$ are satisfied, and the recursive definition 
of our forcing $\mathbb{P}$ yields that at stage
$\kappa$, as $a_{\kappa}= \tau \cap V_{\kappa}$, the sealing forcing $\mathbb{S} 
((\tau \cap V_{\kappa} )^{G\upharpoonright \kappa})$ is at least considered, and 
in order to show property $3$, it suffices to show that 
$ (\tau \cap V_{\kappa} )^{G\upharpoonright \kappa})= \bar{S}\upharpoonright \kappa$ 
is maximal in $V[G\upharpoonright \kappa]$. But this is clear as by the definition 
of RCS iteration and as $|\forceP_{\alpha}| < \kappa$ we take at inaccessible 
$\kappa$'s the direct limit of the $\mathbb{P_{\alpha}}$'s, thus each stationary 
$S\subset \omega_1$ in $V^{\mathbb{P_{\kappa}}}$ is already included in a 
$V^{\mathbb{P_{\alpha}}}$ for $\alpha < \kappa$. So we have ensured the 
existence of a $\kappa$ with all the 3, above stated properties.

Now the forcing $\mathbb{S}(\vec{S}\upharpoonright \kappa)$ can not be semiproper at stage $\kappa$, as otherwise we would have to
force with it, therefore killing the antichain $\vec{S}$. 
So there exists a condition $(p,c) \in \mathbb{S}(\vec{S}\upharpoonright \kappa)$ such that the set
\begin{multline*}
\bar{T}:= \{ X \prec (H_{\kappa^+})^{V[G\upharpoonright \kappa]} \,:\, |X|=\aleph_0 \land (p,c) \in X \land \nexists Y \supset X (Y \prec
(H_{\kappa^+})^{V[G\upharpoonright \kappa]} \\ \land |Y|= \aleph_0 \land (X \cap \omega_1 = Y \cap \omega_1) \land \exists (q,d) \le (p,c)
\, ((q,d) \text{ is Y-semigeneric }))\}.
\end{multline*}
is stationary in $V[G\upharpoonright \kappa]$, and by construction of our iteration, the $\kappa$-th forcing in $\mathbb{P}$ is
$Col(\omega_1, 2^{\aleph_2})$, so in $V[G\upharpoonright \kappa+1]$ 
there is a surjection $f: \omega_1 \rightarrow (H_{\kappa^+})^{V[G\upharpoonright \kappa]}$.
As $Col(\omega_1, 2^{\aleph{_2}})$ is proper the set $\bar{T}$ 
remains stationary in $V[G\upharpoonright \kappa+1]$ which 
implies that $$ T:= \{ \alpha < \omega_1 \, :\, f\text{\textquotedblright}
\alpha \in \bar{T} \land \alpha= f\text{\textquotedblright}\alpha \cap \omega_1\}$$
 is stationary in $V[G\upharpoonright \kappa+1]$. As the tail $\mathbb{P}_{[\kappa +2 , \delta)}$ remains semiproper, seen as an iteration
with $V[G\upharpoonright \kappa+1]$ as ground model, we can infer that $T$ remains stationary in $V[G]$
and hence there
exists an $i_0< \delta$ such that
$$(\ast\ast) \quad T \cap S_{i_0} \text{ is stationary in } V[G].$$

Let us shortly reflect the situation we are in. The idea is to find a model $X \in \bar{T}$ such that we $can$
find a $(X,\mathbb{S}(\vec{S}\upharpoonright \kappa))$-semigeneric condition $(q,d) < (p,c)$, thus arriving at a contradiction.
In order to do so we have to ensure that $\alpha = X \cap \omega_1$ is in 
some $S_i \in \vec{S}\upharpoonright \kappa$. As $\vec{S}$ was assumed to 
be maximal there is indeed an index $i_0 < \delta$ which is as desired, 
this index however might be bigger than $\kappa$. This is where the large 
cardinal assumption comes into play. We can find an elementary embedding $j: V \rightarrow M$ 
such that $j(\kappa) > i_0$, thus it seems that $j(\vec{S}\upharpoonright \kappa)$ is now long
enough to have $S_{i_0}$ as an element. But this is not correct as $\vec{S}$ was not assumed to be definable and therefore
$j{\vec{S}\upharpoonright \kappa)} \ne \vec{S}\upharpoonright j(\kappa)$. We have to use more than just the elementary embedding, namely
that Woodiness fixes even a predicate with $j$. Indeed if we let $\lambda > i_0$ such that $(\tau \cap
V_{\lambda})^{G\upharpoonright \lambda} = \vec{S}\upharpoonright\lambda$, and let $j: V\rightarrow M $ be such that
$j(\tau) \cap V_{\lambda} = \tau \cap V_{\lambda}$ then $j(\tau \cap V_{\kappa}) = j(\tau) \cap V_{j(\kappa)}$
and $V_{j(\kappa)}= V_{\lambda} \cup V_{j(\kappa)} – V_{\lambda}$, thus $j(\vec{S}\upharpoonright \kappa)=
j(\tau \cap V_{\kappa})^{G\upharpoonright \kappa})$ contains $(\tau \cap V_{\lambda})$ and thus $S_{i_0}$.

First let $\lambda < \delta$, $\lambda >$ max$(i_0, \kappa +1)$ be such that $(\tau \cap V_{\lambda})^{G\upharpoonright \lambda}= \vec{S}
\upharpoonright \lambda$, so we have $(\tau \cap V_{\lambda})^{G\upharpoonright \lambda} (i_0) = S_{i_0}$. As $\kappa$
was chosen to be $\mathbb{P}\oplus \tau$-strong up to $\delta$ we let $j: V \rightarrow M$ be an 
elementary embedding with critical point $\kappa$, such that $M$ is transitive, 
$M^{\kappa} \subset M$, $V_{\lambda + \omega} \subset M$, $j(\mathbb{P}) \cap V_{\lambda} = \mathbb{P} 
\cap V_{\lambda}$, and $j(\tau) \cap V_{\lambda} = \tau \cap V_{\lambda}$.

$H$ should denote the generic filter for the segment 
$(\mathbb{P}_{[\lambda +1, j(\kappa)]})^{M[G\upharpoonright \lambda]}$ of $j(\mathbb{P})$
over $M[G\upharpoonright \lambda]$. Then we lift $j$ 
to an elementary embedding $$j^{\ast} : V[G\upharpoonright \kappa] \rightarrow M[G\upharpoonright \lambda, H].$$ 
Notice that $(V_{\lambda + \omega})^{V[G\upharpoonright \lambda]} = (V_{\lambda+ \omega})^{M[G\upharpoonright \lambda]}$.

Now we let $(X_i \,:\, i < \omega_1) \in V[G\upharpoonright \kappa +1]$ be an increasing continuous 
chain of countable elementary substructures
of $(H_{j(\kappa)^{+}})^{M[G\upharpoonright \kappa +1]}$ with $\{ \tau \cap V_{\lambda}, i_0\} \subset X_0$ satisfying for all
$i < \omega_1$ the following three properties:
\begin{enumerate}
\item[(a)] $i \in X_{i+1}$
\item[(b)] $f$\textquotedblright $(X_i \cap \omega_1) \subset X_i$ \item[(c)] 
$j^{\ast}$\textquotedblright$(X_i \cap (H_{\kappa^{+}})^{V[G\upharpoonright \kappa]} \subset X_i$ \end{enumerate}

Let $\bar{G} := G \upharpoonright [\kappa +2, \lambda]$, then we have 
that $$ \{X_i[\bar{G}] \cap \omega_1 \, :\, i < \omega_1 \} \in V[G\upharpoonright \lambda]$$ 
is a club in $\omega_1$ so intersecting it with the stationary set defined in $(\ast\ast)$ 
we find some $i < \omega_1$ such that $X_i[\bar{G}] \cap \omega_1 = X_i \cap \omega_1 \in T \cap S_{i_0}.$

Write $X:=X_i, \alpha:= X \cap \omega_1$. As at stage $\kappa$ we had to force 
with the $\omega$-closed $Col(2^{\aleph_2}, \aleph_1)$
we know that $X \cap (H_{\kappa^{+}})^{V[G\upharpoonright \kappa]} \in V[G\upharpoonright \kappa]$. 
Remember that $f \in V[G\upharpoonright \kappa+1]$ was chosen as
a surjection of $\omega_1$ onto $(H_{\kappa^{+}})^{V[G\upharpoonright \kappa]}$, so as $\alpha \in T$ 
by definition of $T$ $f$\textquotedblright $\alpha
\in \bar{T}$ and $\alpha = f$\textquotedblright $\alpha \cap \omega_1$, and hence by (b) 
$$f\text{\textquotedblright}\alpha \subset X \cap (H_{\kappa^{+}})^{V[G\upharpoonright \kappa]} \in V[G\upharpoonright \kappa].$$ 
As $\alpha = f$\textquotedblright$\alpha \cap \omega_1$, $f$\textquotedblright$\alpha \in \bar{T}$ 
and $f$\textquotedblright$\alpha \subset X \cap (H_{\kappa^{+}})^{V[G\upharpoonright \kappa]}$ 
we get that $X \cap (H_{\kappa^{+}})^{V[G\upharpoonright \kappa]} \in \bar{T}$ and therefore

$$(\ast\ast\ast) \quad j^{\ast}(X \cap (H_{\kappa^{+}})^{V[G\upharpoonright \kappa]}) \in j^{\ast} (\bar{T}).$$

Note that our second generic $H$, denoting the generic filter for the segment 
$(\mathbb{P}_{[\lambda +1, j(\kappa)]})^{M[G\upharpoonright \lambda]}$ of $j(\mathbb{P})$
over $M[G\upharpoonright \lambda]$ has not been specified yet.
As the segment $(\mathbb{P}_{[\lambda+1, j(\kappa)]})^{M[G\upharpoonright \lambda]}$ 
of $j(\mathbb{P})$ over $M[G\upharpoonright \lambda]$
is semi-proper we have that there is a condition $q$ in the segment 
$(\mathbb{P}_{[\lambda +1, j(\kappa)]})^{M[G\upharpoonright \lambda]}$ of
$j(\mathbb{P})$ which is $(X[\bar{G}],\mathbb{P}_{[\lambda+1, j(\kappa)]})$-semigeneric. 
If we pick $H$ such that $q \in H$ then by semigenericity
of $q$ we obtain
$X[\bar{G},H] \cap \omega_1 = X[\bar{G}] \cap \omega_1 = X \cap \omega_1 = \alpha \in S_{i_0}= (\tau \cap V_{\lambda})^{G\upharpoonright
\lambda} (i_0) \in X[\bar{G}, H]$.
But also due to (c) we have that

$$j^{\ast} (X \cap (H_{\kappa^{+}})^{V[G\upharpoonright \kappa]})=  j^{\ast} \text{\textquotedblright} (X \cap
(H_{\kappa^{+}})^{V[G\upharpoonright \kappa]}) \subset X[\bar{G},H].$$

This gives us the desired contradiction as we can find an $(X[\bar{G},H],j(\mathbb{S}(\vec{S}\upharpoonright \kappa)))$-semigeneric
condition below $j(p,c)=(p,c)$. Indeed we can just list the countably many names for countable ordinals in $X[\bar{G},H]$ along with
conditions of $j(\mathbb{S}(\vec{S}\upharpoonright \kappa))$ deciding them below $(p,c)$ and let $(p',c') \in
j(\mathbb{S}(\vec{S}\upharpoonright \kappa))$ be just the condition with dom$(c')=$dom$(d')= \alpha + 1$, $c'(\alpha)=\alpha$ and
$p'(i) = S_{i_0}$ for some $i< \alpha$. So $X[\bar{G},H]$ together with $(p',c')<(p,c)$ witness that
$j^{\ast}(X \cap (H_{\kappa^{+}})^{V[G\upharpoonright \kappa]}) \notin j^{\ast} (\bar{T})$, contradicting $(\ast\ast\ast)$.

\end{proof}

\chapter{$\NS$ is $\Delta_1$-definable and $\NSA$ is saturated}
\section{Introduction}
In this section we want to give a proof of the following theorem

\begin{thm}
Assume that $M_1^{\#}$ exists, and let 
$A \in M_1$ be a stationary, co-stationary subset of $\omega_1$. Then there is a generic extension $M_1[G]$ via 
a set sized forcing such that in $M_1[G]$ 
$\NS$ is $\Delta_1$-definable in $H(\omega_3)$ with 
parameter $\omega_1$ and such that $\NSA$ is $\aleph_2$-saturated.
\end{thm}
Before starting the actual proof we outline its route roughly: The main idea is to use 
Shelah's proof of the saturation of $\NS$ from a Woodin cardinal as a starting point and try to add 
certain forcings which code additional information of the universe. This additional 
information can be used to obtain a nicer definition of stationarity on $\omega_1$.
The extra coding forcings should be chosen in such a way that they will not 
interfere with the usual sealing forcings which push the saturation of 
$\NS$ down to $\aleph_2$, and should be robust enough in that a once 
coded information should be preserved in all future further generic extensions of the universe.

To be a little more precise the proof is via a $\delta$-long ($\delta$ the Woodin cardinal) countable support iteration of semiproper
and $\Stat$-proper forcings. We will use a $\diamondsuit$-sequence $(a_{\alpha})_{\alpha < \delta}$
to guide the
iteration. At each stage $\alpha$ of the iteration we look at the $\alpha$-th entry of the sequence.
We will distinguish two different cases, a coding stage and a sealing stage and start to describe the first of 
the two stages:

If $a_{\alpha}$ is the $\forceP_{\alpha}$-name of a stationary subset 
$S\subset A$ then we code the characteristic function of $S$ into a 
pattern of nicely definable trees which should have a cofinal branch 
or be bounded, meaning in that particular context that
there are no cofinal branches through the tree in no $\omega_1$-preserving
outer model. This is followed by a bunch of forcings which add a 
more local version of this information to the universe. We shall 
see that the iteration of these forcings is $\Stat$-proper.

If $a_{\alpha}$ is a $\forceP_{\alpha}$-name of a long (i.e. of length $\aleph_2$) maximal 
antichain in the structure $P(A) / \NS $ then we seal off this 
antichain after collapsing its length to $\aleph_1$ but only if this forcing
is semiproper. This seemingly redundant move will enable us to 
argue for stationary set preservation for the resulting iteration, as, contrary to the semiproper case, there is no iteration theory of
stationary set preserving notions of forcing.
This forcing is seen to be $A^c$-proper and semiproper.
In the remaining case we just do nothing and force with the trivial forcing.

To summarize we arrive at a countable support iteration of $\Stat$-proper 
and $A^c$-proper forcings, therefore the iteration is
$\Stat(A^c)$-proper where $\Stat(A^c)$ should denote the
class $\{ M \in \Stat \, : \, M \cap \omega_1 \in A^c \}$.
We will see soon that the class $\Stat$ is projective stationary, i.e. $\Stat(A^c)$ of models remains 
everywhere stationary, hence the iteration with countable
support yields a $\Stat(A^c)$-proper extension of the ground model $M_1$.
Consequently all new countable sets of ordinals can 
be covered by countable sets of ordinals living in the ground model. 
We use this observation to argue that stationary subsets of $\omega_1$ are preserved in the 
just described iteration:

We can look at the iteration in a different way. 
As we only seal off when the sealing forcing is semiproper, and as the forcing 
from coding stages is $\Stat$-proper (as to be seen soon), hence $\Stat$-semiproper 
we drop the countable support iteration for a second and use an $RCS$-iteration 
instead, i.e. we use all the factors of the iteration but iterate them
using the revised countable support instead of the plain countable support.
We arrive at a $\Stat$-semiproper generic extension of the ground model $M_1$ which is stationary
set preserving as $\Stat$ is projective stationary. 
Yet, as $\Stat(A^c)=\{ M \in \Stat \, : \, M \cap \omega_1 \in A^c \}$
remains everywhere stationary, countable subsets of ordinals in the generic
extension can be covered by countable sets in the ground model.
Consequently the $RCS$-iteration is in fact 
just a countable support iteration. Indeed at every stage $\alpha < \delta$ 
we see that $V^{\forceP_{\alpha}}$ is a $\Stat(A^c)$-proper 
forcing extension of $V$, thus whenever there is an ordinal $\beta > \alpha$ 
such that $cof(\beta)$ is countable in $V^{\forceP_{\alpha}}$,
it has already been of countable cofinality in $V$. Thus the $RCS$-iteration is just a countable support
iteration, and we can iterate with countable support to end up with a $\Stat$-semiproper extension of the ground model.
As a consequence the final model, when using countable support iteration
for the iterands, preserves stationary subsets of $\omega_1$.

\section{The coding forcing}

As already mentioned we need a coding forcing which 
harmonizes with the stationary sealing forcings we 
need to keep the saturation of $\NS$ low at $\aleph_2$.
Our ground model for the iteration will be $M_1$, 
the canonical inner model with one Woodin cardinal. 
We cut $M_1$ at the Woodin cardinal $\delta$ to build 
Steel's $K$ there. We have already seen that $K^{\mathcal{J}^{M_1}_{\delta}} = \mathcal{J}_{\delta}^{M_1}$.
The move towards $K$ has the advantage that 
$K$ has a nice first order definition, as opposed 
to $M_1$, and generic absoluteness below $\delta$ is a well known fact for $K$.
We use $K$-trees $T_{\alpha}(\xi):= ((\alpha^{+\xi})^{< (\alpha^{+\xi} )})^{K}$, 
where $\beta$ is a cardinal from $K$, and either force a 
cofinal branch through it or we ensure that there can not 
be any cofinal branches in $\omega_1$-preserving outer models.
This way we can code arbitrary 0,1-patterns into sequences of cofinal or bounded trees.
The big advantage of this method is that once we decided to 
write a certain pattern using these trees this information 
will prevail in all $\omega_1$-preserving outer models. 
Thus we don't have to reconsider earlier information during 
the iteration. These sequences of cofinal or bounded trees 
will code characteristic functions of stationary subsets of 
$\omega_1$. Of course the process of specializing a tree off 
or shooting a cofinal branch should preserve stationary subsets  
of $\omega_1$ to be useful for our purpose. 
Moreover we should be able to iterate these 
forcings in a nice way. That this is indeed 
the case is the content of the next proposition.

\begin{Proposition}[Coding via specializing trees]\label{codinglemma} 
Assume GCH, $\beta > \omega_1$ regular, and let $\Stat$ be a 
stationary class which is projective stationary, i.e.
for every stationary $S \subset \omega_1$, 
$\Stat(S) = \{ M \in \Stat \, : \, M \cap \omega_1 \in S \}$ is 
everywhere stationary. 
Suppose that $\mathbb{Q}$ is an $\Stat$-proper notion of 
forcing of size less than $\beta$ and $G$ is V-generic. Then:

\begin{itemize}
\item[1.] $T(\beta):= ((\beta^+)^{(< \beta)})^{V}$ viewed as a forcing is $\Stat[G]$-proper over $V[G]$.
\item[2.] There is a proper forcing $\mathbb{R} \in V[G]$ of size $\beta^{++}$  
that destroys the properness of $T(\beta)$. More specifically
if $H$ is $\mathbb{R}$-generic over $V[G]$, then in any $\omega_1$ preserving outer model of $V[G][H]$ there is no
branch through $T(\beta)$ which is $T(\beta)$-generic over $V$.
\end{itemize}

\end{Proposition}

\begin{proof}
We shall show first that $T(\beta):= ((\beta^+)^{(< \beta)})^{V}$ is $\Stat[G]$-proper over $V[G]$ for $G$ a generic
filter for the $\Stat$-proper forcing $\forceQ$. We will show that for any 
condition $p \in \forceQ$
\begin{multline*}
 p \Vdash \exists C \subset [H_{\theta}]^{\omega} (\forall M[\dot{G}] \in C \cap \Stat[\dot{G}] \, \forall t \in T(\beta) \cap M[\dot{G}]
\, \\ \exists
s < t (s \text{ is } (M[\dot{G}], T(\beta)) \text{-generic}))
\end{multline*}

So let us first fix $p \in \forceQ$. Then we let the club $C \subset [H_{\theta}]^{\omega}$ $(\in V[G])$ 
be the set of all models $M$ which contain $p$. We pick now an arbitrary $M \in C \cap S[\dot{G}]$ and assume,
using the $\Stat$-properness of $\forceQ$, that $M[\dot{G}]\cap V =M$.
We list all the dense subsets $D \subset T(\beta)$ in $M$. Additionally we list all the pairs in $M$ which are of the form
$(r, \dot{D})$ where $\dot{D}$ is a $\forceQ$-name, and $r\le p$ a condition such that $r \Vdash$\textquotedblleft
$\dot{D} \text{ is dense in } T(\beta)\textquotedblright$. 
Let $t \in T(\beta) \cap M$ be a condition. We shall show that $p$ forces that there is an $s < t$ which is 
$(M[\dot{G}], T(\beta))$-generic.
We recursively define a descending sequence of $T(\beta)$-conditions which are in $M$ like this.
First set $t_0:= t$. If $t_{n-1}$ is already defined we look at the $n$-th entry of our $\omega$-long list 
consisting of dense subset $D \subset T(\beta)$ and pairs $(r, \dot{D})$ of $\forceQ$-names for dense
subsets in $M$. We split into cases
\begin{enumerate}
 \item if the $n$-th entry of the list is some dense $D \subset T(\beta)$, $D \in M$, then we pick 
       a condition $t_n \in T(\beta) \cap M \cap D$ such that 
       $t_n < t_{n-1}$.
 \item if the $n$-th entry of the list is a pair $(r, \dot{D}) \in M$ such that 
       $r \Vdash$\textquotedblleft $\dot{D}$ is dense in $T(\beta)\textquotedblright$, then define recursively in $H_{\theta}$ a maximal 
       antichain $(a_{\gamma})_{\gamma < \kappa} =:A \subset \forceQ$ 
       below $r$ together with a descending sequence of $T(\beta)$-conditions $(t'_{\gamma})_{\gamma < \kappa}$ below $t_{n-1}$ as follows:
       Assume that $a_{\gamma}, t'_{\gamma}$ have already been constructed then 
       we let $a_{\gamma+1}$ be the $<$ least element below $r$ in $\forceQ \cap H_{\theta}$ (where $<$ denotes a previously
       fixed wellorder of $H_{\theta}$) 
       such that $a_{\gamma+1}$ is incompatible with all the $a_{\delta}, \delta < \gamma+1$,
       and such that there is a $t'_{\gamma+1} \in T(\beta)$ 
       such that $a_{\gamma+1} \Vdash_{\forceQ}$\textquotedblleft $t'_{\gamma+1} \in \dot{D} \land t'_{\gamma+1} \le t'_{\gamma}
\textquotedblright$.
       This antichain $A$ has size less that $\beta$ as $|\forceQ| < \beta$ thus we can always take 
       $t'_{\delta}$ to be the union of all the previous $t'_{i}$'s for limit stages $\delta < \kappa$.
       As $A$ is definable in $H_{\theta}$ and $M \prec H_{\theta}$, 
       $A \in M$ as well and finally we let the condition $t_n:= \bigcup_{\gamma < \kappa} t'_{\gamma}$
       which is an element of $M$ again. Note that by construction 
       \begin{itemize}
        \item[] for every $\gamma < \kappa$,
                $a_{\gamma} \Vdash_{\forceQ}$\textquotedblleft $t_n \Vdash_{T(\beta)} t'_{\gamma} \in \dot{H} \cap \dot{D}\textquotedblright$
       \end{itemize} 
        where $\dot{H}$ is the canonical name for the $T(\beta)$-generic filter. Therefore
    \begin{itemize}
     \item[($\ast$)] $r \Vdash_{\forceQ} \text{\textquotedblleft} t_n \Vdash_{T(\beta)} \dot{H} \cap \dot{D} \ne \emptyset.\textquotedblright$  
    \end{itemize}

\end{enumerate}
Finally we let $t_{\omega} := \bigcup_{n \in \omega} t_n$ and
we claim that indeed $$p \Vdash_{\forceQ} t_{\omega} \text{ is } (M[\dot{G}], T(\beta)) \text{-generic.}$$
But this is clear by construction of the $t_{\omega}$: 
if $\dot{D} \in M$ is the $\forceQ$-name of 
a dense subset of $T(\beta)$ and $p \Vdash_{\forceQ}$\textquotedblleft $\dot{D} \text{ is dense}\textquotedblright$
and so the pair $(p, \dot{D})$ will appear in our recursive construction from above,
say $\dot{D} = \dot{D}_{n}$ and then by ($\ast$) $p \Vdash_{\forceQ}$\textquotedblleft $
t_n \Vdash_{T(\beta)} \dot{H} \cap \dot{D} \ne \emptyset\textquotedblright$.
As $t_{\omega} < t_n$ we also have that $p \Vdash_{\forceQ}$\textquotedblleft $
t_{\omega} \Vdash_{T(\beta)} \dot{H} \cap \dot{D} \ne \emptyset\textquotedblright$,
so $p \Vdash_{\forceQ}$\textquotedblleft $t_{\omega} \text{ is } (M[\dot{G}], T(\beta)) \text{-generic.}\textquotedblright$

To prove the second statement, first add $\beta^{++}$ Cohen reals with a 
finite support product over $V[G]$, then L\'evy collapse $\beta^{++}$ to $\omega_1$ 
and let $V[G][H_1][H_2]$ denote the resulting model, which is a proper forcing 
extension of $V[G]$. By an observation of J.Silver we know that each 
$\beta$-branch through $T(\beta)$ which lies in $V[G][H_1][H_2]$, in 
fact lies already in $V[G][H_1]$. Indeed if $\dot{b}$ is a name in 
$V[G][H_1]$ for a new $T(\beta)$-branch, then we can build a binary 
$\omega$-tree of conditions in the L\'evy collapse and by its $\omega_1$ 
closure each branch has a lower bound, resulting in $2^{\aleph_0}=\beta^{++}$-many 
different interpretations of $\dot{b}$. Thus
$(T(\beta))^{V}$ has in $V[G][H_1]$ $2^{\aleph_0}=\beta^{++}$ many branches on a 
level $< \omega_1$, which is impossible as $V\models GCH$ and $|\forceQ|<\beta$.

Thus $T(\beta)$ has at most $\omega_1$-many branches in $V[G][H_1][H_2]$, and none of those branches is cofinal in $\beta^+$, therefore none
of the branches is $T(\beta)$-generic over $V$. Also every node is included in a $\beta$-branch. This enables us to use Baumgartners method
of ``specializing a tree off a small set of branches''.

\begin{Fact}
If T is tree of height $\omega_1$, such that each node is contained in a cofinal branch, and with at most $\aleph_1$ many cofinal branches,
then there is a ccc forcing $\mathbb{P}$ such that if G is $\mathbb{P}$-generic then whenever W is an outer model of V[G] with the same
$\omega_1$ then each cofinal branch through W belongs already to V.
\end{Fact}

Now we can use the forcing $\mathbb{P}$ from the Fact above to build the model $V[G][H_1][H_2][H_3]$ and obtain that whenever $W$ is an
$\omega_1$ preserving outer model of $V[G][H_1][H_2][H_3]$ then 
each branch through the tree $T(\beta)$ in $W$ is in fact in $V[G][H_1][H_2]$, and
therefore as already noted in $V[G][H_1]$. As no branch through $T(\beta)$ in $V[G][H_1]$ is cofinal in $\beta^+$ (by the ccc), and
$T(\beta)$ generic branches are necessarily cofinal in $\beta^+$, we are done.

\end{proof}

\section{The class $\Stat$}
Note that in the following we constantly use the fact 
that $M_1$ below the Woodin cardinal coincides with Steel's core model $K$ as build in $\mathcal{J}^{M_1}_{\delta}$.
The move towards $K$ has the advantage that, contrary to $M_1$ 
$K$ admits a first order definition which works uniformly in all size $<\delta$
generic extensions. The forcing iteration we will use in the 
proof of the theorem lives below the Woodin cardinal, thus every
intermediate model of the iteration is able to define its $K$ correctly. 
As $M_1$ is the ground model of our iteration, we do not have access to full condensation.
Knowing that a sufficiently elementary submodel $M$ of some 
$\mathcal{J}^{M_1}_{\eta}$ collapses to an $M_1$ initial 
segment is nevertheless of highest importance in our proof.
The reason for this is lies in the fact that we use $K$ 
definable trees, living on $K$ cardinals. It is therefore 
desirable that these trees still live on $K$ initial segments 
even after we transitively collapse, in order to not completely 
lose control of the things we are talking about. 
These considerations will become clearer once the 
everywhere stationary class $\Stat$, introduced for 
these very reasons is seen in action during the proof.

We introduce without much further ado:

\begin{Definition}
Let $\Stat$ be a class of countable sets. We say that $\Stat$ is everywhere stationary if
\begin{enumerate}
\item for every regular cardinal $\theta$, $\Stat \cap H(\theta)$ is a stationary subset of $[H(\theta)]^{\omega}$, and
\item $\Stat$ is closed under truncation, I.e $X \cap H_{\theta} \in \Stat$ whenever $X \in \Stat$ and $\theta$ a regular cardinal.
\end{enumerate}

\end{Definition}
\begin{Definition}

Let $M$ be a (sufficiently) elementary submodel
of some $K$ initial segment $\mathcal{J}^{K}_{\eta}$. 
We say that $M$ collapses nicely if the following demands are met: \begin{enumerate}
\item the transitive collapse
$\bar{M}$ is an initial segment of $K$, i.e. $\bar{M} \vartriangleleft K$, and
\item whenever $\omega_1^{\bar{M}}$ is
a cardinal in a $K$-initial segment $\mathcal{J}^{K}_{\eta}$ 
then already all cardinals of $\bar{M}$ remain cardinals in $\mathcal{J}^{K}_{\eta}$, and
\item $\bar{M}$ is $K$-correct, meaning that if $\Phi(x)$ denotes 
the first order formula defining $K$ and use it to define $K$ inside the transitive model $\bar{M}$  then
$K^{\bar{M}}= \bar{M}$.
\end{enumerate}
\end{Definition}
This list seems quite daring on first sight, nevertheless there are plenty of nicely collapsing submodels:
\begin{Lemma}
Let $\Stat$ denote the class of nicely collapsing submodels of $M_1$. Then $\Stat$ is everywhere stationary in $M_1$.
Moreover $\Stat$ is projective stationary, meaning that
for every stationary $X \subset \omega_1$, the set $\Stat(X):=\{ M \in \Stat \, : \, M \cap \omega_1 \in X \}$
remains an everywhere stationary class.
\end{Lemma}
\begin{proof}
The first item is just the Lemma \ref{condensation} on Condensation we have proved already.
For the second part we have to show that club often $\bar{M}=\mathcal{J}^{M_1}_{\eta}$ is $K$-correct, 
but this is easy as we could have just picked an $M$ which is elementary 
in the universe with respect to the $K$ defining formula $\Phi$.

$\Stat$ is trivially closed under truncation and all the 
considerations so far hold for $H(\theta) \cap M$ as well, 
thus $H(\theta) \cap M \in C \cap \Stat$ and $\Stat$ is everywhere stationary.

Finally to show that $\Stat$ is projective stationary
we assume the opposite, thus there is a stationary $X \subset \omega_1$ and a fixed regular cardinal $\theta$ 
for which there 
is a club $C \subset [H_{\theta}]^{\omega}$ such that every
$M \in C$ does not collapse to an initial segment of $M_1$ or
satisfies $M \cap \omega_1 \notin X$. Let $C$ be the least such club in
the wellorder. We pick a regular $\lambda > \theta$ and 
the $\Sigma_n$-Skolem hull denoted by $N$ of $\{X, \theta\}$ in $H_{\lambda}$. 
We can assume that $N \cap \omega_1 \in X$. 
It is clear that $N \cap H_{\theta} \in C$, however
we can arrange that $N$ collapses to an initial segment of $M_1$ and 
so does $N \cap H_{\theta}$ which is a contradiction.
To see that we can always assume that $N$ collapses to an
initial segment of $M_1$ we first note that $N$ projects to $\omega$
as it is not fully elementary in $H_{\lambda}$ and can be chosen to be $\omega$-sound.
If we pick a countable initial segment 
of $M_1$, say $\mathcal{J}^{M_1}_{\eta}$
which projects to $\omega$ and is $\omega$-sound as well we can start to compare these two models.
But now both models will not move during the comparison
and so $\bar{N} \vartriangleleft \mathcal{J}^{M_1}_{\eta}$ as desired.

\end{proof}

The role of $\Stat$ and its somewhat peculiar definition will become clearer during the advance of the proof.
The reason we have to constantly fall back on $\Stat$ is that in the end 
we want suitable countable models to be able to see the patterns of cofinal 
or bounded trees we inscribed during the iteration. As these trees are 
defined in $M_1$ on $M_1$-cardinals, the countable models should be 
correct about their $M_1$ and their sequence of $M_1$-cardinals.
As condensation in general fails in $M_1$ we are compelled to work with $\Stat$.

\section{The definition of the iteration}
We now describe the iteration in detail which is used to prove the main theorem. 
It utilizes a $\delta$-long countable support iteration of $\Stat$-semiproper 
forcings, the factors of it will be explained in the following subsections.
Our ground model is the inner model $M_1$ with one Woodin cardinal $\delta$, which is also Woodin with $\diamondsuit$ as was shown.
Actually we are working all the time below the Woodin
cardinal $\delta$, and we emphasize again that $\mathcal{J}^{M_1}_{\delta} = K$.
We will therefore use $M_1$ and $K$ synonymously from now on and hope
that it will not confuse the reader. Remember that
in $M_1$ $\GCH$ does hold and there exists a well-order. denoted by $<$ of $M_1$.
Fix a stationary, co-stationary subset $A \subset \omega_1$.
The goal is to arrange the sealing stages, where we seal a maximal antichain 
of stationary subsets of $\omega_1$ off and the coding stages, where we code 
the characteristic function $\chi_S$ of a stationary $S \subset \omega_1$
into an according pattern of cofinal and bounded (i.e. trees which have no branch of the
height of the tree) trees in such a 
way that they will not interfere with each other.

We use the $\diamondsuit$-sequence
$(a_{\alpha})_{\alpha < \delta}$ to determine with which forcing we should force at stage $\alpha$.
Thus assume that $\alpha < \delta$
and we have already constructed $\mathbb{P}_{\beta}$ for $\beta \le \alpha$.
We define the forcing $\dot{\forceQ}_{\alpha}$ in $V^{\forceP_{\alpha}}$ as follows:
\begin{enumerate}
\item if $a_{\alpha}$ is a $\forceP_{\alpha}$-name of a stationary subset 
$S$ of $\omega_1$ then we code the characteristic function of $S$ into 
a pattern of trees, followed by forcings which will localize the inscribed 
information. These forcings will be specified below in a detailed description.
\item if $a_{\alpha}$ is a $\forceP_{\alpha}$-name of a maximal antichain 
of stationary subsets of our fixed stationary, co-stationary 
$A \subset \omega_1$, then use the sealing forcing to seal it off, provided the
sealing forcing is $\Stat$-semiproper. If it is not force with the usual Levy collapse $Col(\aleph_1, 2^{\aleph_2})$.
\item else collapse $2^{\aleph_2}$ to $\aleph_1$ or force to create a default pattern of cofinal trees
\end{enumerate}
The third point of the definition deserves an explanation: our goal is that we only create $\omega_1$-length patterns
of cofinal and bounded trees which correspond to characteristic functions of stationary subsets of $\omega_1$.
We allow a default pattern however to fill in gaps in the sequence of patterns.
The reason that we do so is to avoid the case of accidentally forced patterns during our iteration.
Indeed if there is a gap in between two characteristic functions of stationary sets, it could happen that
the noise of the forcings we use create an unwanted pattern of trees in the gap, which we cannot kill off once it is produced.
This potentially ruins our argument. To avoid this degenerated case we set up the iteration in such a way that
a default pattern is forced in all gaps between characteristic functions of stationary sets. This is cumbersome
to explain in the recursive way we chose above, therefore we used the blurry words in the third case and
hope that this remark makes the point clear.

\section{The sealing forcings}
Whenever $a_{\alpha}$ at stage $\alpha$ is the $\forceP_{\alpha}$-name 
of a long maximal antichain in $P(A)/ \NS$ then we have to seal it off using the following notion of forcing:

\begin{Definition}
Assume that $\vec{S}_{A}$ is a maximal antichain of stationary subsets 
of $A$, where $A \subset \omega_1$
is stationary, co-stationary. Conditions of the sealing forcing $\seal_A$ 
are pairs $(p,c)$ where $p: \alpha+1 \rightarrow \vec{S}_A $ is a function 
and $c: \alpha+1 \rightarrow \omega_1$ is a function with range a closed subset of $\omega_1$ such that
$$ \forall \xi \le \alpha (c(\xi) \in A \rightarrow  c(\xi) \in \bigcup_{i \in \xi} p(i))$$ \end{Definition}
It is well known that $\seal_A$ is $\omega$-distributive, preserves stationary subsets of $\omega_1$ but 
is not necessarily semiproper. The definition of $\seal_A$ still makes sense 
if the antichain $\vec{S}_A$ is not maximal. In that case forcing with 
$\seal_A$ turns the antichain into a maximal one. If $A$ is not 
stationary then the definition of $\seal_A$ is again meaningful 
and the forcing $\seal_A$ has a dense subset which is countably closed.

\section{The coding forcings}
We turn now to the forcings mentioned in case 1.
They are defined as a three step iteration
of forcings $\forceQ_{\alpha}^0 \ast \forceQ_{\alpha}^1 \ast \forceQ_{\alpha}^2$ which will be defined now:

\subsection{$\forceQ_{\alpha}^0$}

For brevity we use in the following often the notion $(K)_{\eta}$ for an ordinal
$\eta$ which just means $\mathcal{J}^K_{\eta}$.
Assume we arrived in our iteration at stage $\alpha$, $\alpha$ a cardinal of the ground model $K$ which
is $M_1$ below the Woodin cardinal $\delta$,
and the $\diamondsuit$-sequence 
$(a_{\beta})_{\beta < \delta}$ at stage $\alpha$ is the $\forceP_{\alpha}$-name 
of a stationary subset $S \subset \omega_1$. Then we want to write the 
characteristic function $\chi_S$ of $S$ into a pattern of canonical trees for 
which we either shoot a cofinal branch through, or make sure that 
there will never be a cofinal branch in each $\omega_1$-preserving outer model.

More specifically we let $\forceQ^0_{\alpha}$ be an $\omega_1$-length iteration of
forcings $(\forceP_{\xi} \, : \, \xi < \omega_1)$ 
defined as follows: For $\xi < \omega_1$ we consider the
forcing $T_{\alpha}(\xi)$ consisting of the tree of $K$-sequences of 
elements of $(\alpha^{+(\xi+1)})^K$ of length less than $(\alpha^{+\xi})^K$, i.e.
we let $T_{\alpha}(\xi)= ((\alpha^{+(\xi+1)})^{< (\alpha^{+\xi})})^{K}$. This forcing remains $\Stat$-proper
over $K[G_{\alpha}]$, as was shown in Lemma \ref{codinglemma}.
However, again by Lemma \ref{codinglemma}
there exists a forcing $\mathbb{R_{\alpha} (\xi)}$ of size $\alpha^{+(\xi +2)}$ 
such that after forcing with $\mathbb{R}_{\alpha}(\xi)$, if $H$ denotes the generic for $\forceR_{\alpha}(\xi)$,
each outer model of $K[G_{\alpha}][H]$ which preserves $\omega_1$ 
can not contain any branch through $T_{\alpha}(\xi)$ which is cofinal in $\alpha^{+(\xi+1)}$.
Now we take the stationary set $S \subset \omega_1$ and code the pattern
of $S$ into an $\omega_1$-block of trees $T_{\alpha}(\xi) := ((\alpha^{+(\xi+1)})^{< (\alpha^{+\xi})})^{K}$,
using only every third cardinal successor of $\alpha$ in $K$ for a nontrivial forcing in order to guarantee 
$\Stat$-properness.
\begin{enumerate}
\item if $\xi \in S$ then we let $\forceP_{\xi \cdot 3}$ be the 
forcing which shoots a cofinal branch through the tree $T_{\alpha}(\xi^{+++})$ using the tree forcing
with conditions that are nodes in the tree $T_{\alpha}(\xi^{+++})$.
 
\item if $\xi \notin S$ then pick $\forceP_{\xi \cdot 3}$ to 
specialize the tree $T_{\alpha}(\xi^{+++})$, using the already defined $\Stat$-proper
specialization forcing.

\item if the index $\xi$ of $\forceP_{\xi}$ is not of the form $\eta \cdot 3$ for an $\eta < \omega_1$
      we let $\forceP_{\xi}$ be the trivial forcing. 
\end{enumerate}
We define $\forceQ^0_{\alpha}$ to be the countable support iteration of
the just defined $\forceP_{\xi}$, $\xi < \omega_1$.

Note that as $K$ satisfies the $\GCH$ and as we only used every third $K$ cardinal
we ensure that the condition on the size of the forcings in Lemma \ref{codinglemma}
is met, thus each iterand $\forceP_{\xi}$ is an $\Stat$-proper forcing
and so is $\forceQ_{\alpha}^0$.
Note further that if $G^0$ is $\forceQ^0_{\alpha}$-generic
then in $M_1[G_{\alpha}] [G^0]$ the model sees the stationary 
set $S$ via the pattern of the trees $(T(\alpha))^{K}$, having a cofinal branch or not.

To avoid unwanted patterns we can without any problems demand that all $\omega_1$-blocks are used, meaning that each block has a
pattern (possibly a dummy pattern) written on it.

\subsection{$\forceQ_{\alpha}^1$}

To define the next forcing $\forceQ_{\alpha}^1$ 
we first fix a sufficiently big initial segment 
$\mathcal{J}^{K}_{\alpha^1}$ of $K$ such that 
the generic filter $G_{\alpha} \ast G^0$ is also 
generic over $\mathcal{J}^{K}_{\alpha^1}$. We then 
collapse the size of the structure 
$\mathcal{J}^{K}_{\alpha^1}[G_{\alpha}][G^0]$ to $\omega_1$ 
via the usual Levy collapse, $\forceQ_{\alpha}^1:= Coll(\lambda, \omega_1)$, 
where $\lambda = |\mathcal{J}^{K}_{\alpha^1}|$.
If $G^1$ is $\forceQ_{\alpha}^1$-generic then in $K[G_{\alpha}][G^0][G^1]$ 
there exists a subset $X_{\alpha} \subset \omega_1$ which codes the structure
$\mathcal{J}^{K}_{\alpha^1}[G_{\alpha}][G^0]$. Whenever $M$ 
is a transitive model of $\ZFC$, which is $K$-correct and $X_{\alpha} \in M$ 
then $M$ will say that:
\begin{enumerate}
\item[] there is a $K$-cardinal $\alpha$ such that the previously coded stationary 
set $S$ can be read off from a pattern of cofinal or bounded branches through 
canonical $K$-trees starting at $\alpha$.
    
\end{enumerate}

\subsection{$\forceQ_{\alpha}^2$ (Localization)}

Our goal is to add a set $Y_{\alpha} \subset \omega_1$ which codes the same information as the
just added set $X_{\alpha}$, and additionally reflects the 
desired property down to all suitable models of size $\aleph_0$. 
To be more specific we want a $Y_{\alpha} \subset \omega_1$ for 
which the following holds:

\begin{itemize}
\item[$\heartsuit$]  $\forall M$ countable, 
transitive model of
 $\ZFP$, $(Y_{\alpha} \cap \omega_1) \in M$, $(K)^M \subset \mathcal{J}^{K}_{\eta}$ 
for some $\eta < \omega_1$, $(\omega_1)^M=(\omega_1)^{(K)^M}$ and which satisfies that
 ($\forall \beta \in M \forall X$ (\textquotedblleft $X=\mathcal{J}^{K}_{\beta}$\textquotedblright 
$\rightarrow \mathcal{J}_{\beta}^K \subset K^M$)) then $M$ satisfies that 
it can decode out of $Y_{\alpha} \cap \omega_1$ an ordinal
$\bar{\alpha}$ such that for every $\xi < \omega_1$ $(\xi \in S$ if and only if $T_{\bar{\alpha}}(\xi^{+++})$ 
has a cofinal branch in $M$).
\end{itemize}

To force property $\heartsuit$ we just use approximations of size $\aleph_0$. Thus our forcing $\forceQ^2_{\alpha}$ consists
of the set of $\omega_1$-Cohen conditions $p: |p| \rightarrow 2$ in $K[G_{\alpha}][G^0][G^1]$ with the
properties that

\begin{itemize}
\item The domain of $p$ is a limit ordinal $< \omega_1$.
\item The even part of $p$ codes $X_{\alpha} \cap \omega_1$.
\item For any limit ordinal $\xi < \omega_1$, $\xi \le |p|$ and any transitive $\ZFP$-model $M$ of size $\omega_1$, with
$\omega_1^M = (\omega_1)^{{K}^{M}}$ and which contains $p \upharpoonright \xi$, 
moreover satisfies that $(K)^M \subset (K)_{\eta}$ for some $\eta < \omega_1$
we have that there is $\bar{\alpha}$ such that $\forall \xi < (\omega_1)^M$
$\xi \in S$ if and only if $(T_{\bar{\alpha}}(\xi^{+++}))^{(K)^M}$ has a branch in $M$ which is cofinal in $((\bar{\alpha})^{+(\xi+3)})^{K^{M}}$.

\end{itemize}

\begin{Claim}
The forcing
$\forceQ_{\alpha}^2$ is $\Stat[G_{\alpha}][G_0][G_1]$-proper.
\end{Claim}

\begin{proof}
Note that we can assume, via arguing by induction on the length of the 
iteration that $\Stat$ is stationary in the structure $[H_{\kappa} \cap K]^{\omega}$ evaluated in 
$K^{\forceP_{\alpha} \ast \forceQ^0_{\alpha} \ast \forceQ^1_{\alpha}}$ as $\Stat(A^c)$ is.
Note also that the forcing $\mathbb{Q}^2_{\alpha}$ has the extendibility property, meaning that given some condition
$p \in \mathbb{Q}^2_{\alpha}$
and some countable limit ordinal $\gamma > |p|$, we can always find a condition $q$ extending $p$, such that $|q|=\gamma$.
This is easily seen as given a condition $p \in \mathbb{Q}^2_{\alpha}$, one can extend its length to $\gamma$ in such
a way that the even entries still code $X_{\alpha}$ up to $\gamma$, and on the odd entries one codes an ordinal $\eta > \gamma$
into an $\omega$-block of the odd entries up to
$\gamma$. Let $q$ be such a sequence. Then no model $N$ which has $q$ as an element can have $\gamma$ as
its $\omega_1$ as it can use $q$ to see the countability of $\eta> \gamma$. Thus the third property is
automatically satisfied and the sequence $q$ of length $\gamma$ is a condition in $\mathbb{Q}^2_{\alpha}$.
We want to show the $\Stat[G_{\alpha}][G_0][G_1]$-properness now, thus we shall exhibit a club 
$C \subset [H(\theta)^{K[G_{\alpha}][G^0][G^1]}]^{\omega}$ such that 
whenever we pick an arbitrary countable elementary submodel $M \in C \cap \Stat[G_{\alpha}][G_0][G_1]$
and some condition $p \in M \cap \mathbb{Q}_{\alpha}^2$ we can show that there is a stronger $q$ which is
$(M,\mathbb{Q}_{\alpha}^2)$-generic.
Choose $C$ to be a club of elementary submodels of some $\mathcal{J}^{K}_{\sigma}[G_{\alpha}][G^0][G^1]$,
such that the set $X_{\alpha}$ is in every element of $C$, and let the ordinal $\sigma$ be 
large enough to enable $\mathcal{J}^{K}_{\sigma}[G_{\alpha}][G^0][G^1]$ to define its $K$ properly.
Decoding the information out of $X_{\alpha}$,  $(\mathcal{J}^{K}_{\sigma}[G_{\alpha}][G^0][G^1]$
will agree that \begin{enumerate}
\item the $K$-initial segment $\mathcal{J}^{K}_{\alpha^1}$ coded into $X_{\alpha}$ 
is an initial segment of the $K$-version which is computed in the model
$\mathcal{J}^{K}_{\sigma}[G_{\alpha}][G^0][G^1]$. Moreover the model 
$\mathcal{J}^{K}_{\sigma}[G_{\alpha}][G^0][G^1]$ will see that $\mathcal{J}^{K}_{\alpha^1}$
is an initial segment of its own version of $K$. To emphasize the fact that
we decode out of $X_{\alpha}$ we write 
$dec(X_{\alpha})$ to denote the mouse $\mathcal{J}^{K}_{\alpha^1}$ which is coded into $X_{\alpha}$. Thus 
$dec(X_{\alpha}) \vartriangleleft (K)^{(\mathcal{J}^{K}_{\sigma}[G_{\alpha}][G^0][G^1]}$ does hold.
 
\item As $\mathcal{J}^{K}_{\sigma}[G_{\alpha}][G^0][G^1]$ and $\mathcal{J}^{K}_{\alpha^1}$ do 
have the same $\omega_1$, the sequence of $K$-trees, as computed in 
$\mathcal{J}^{K}_{\sigma}[G_{\alpha}][G^0][G^1]$ and $\mathcal{J}^{K}_{\alpha^1}$ respectively, which starts at
the ordinal $\alpha$ (which is coded into $X_{\alpha}$), will coincide. 
To be more precise, the sequence of $K$-cardinals starting at $\alpha$, 
as computed in 
$(\mathcal{J}^{K}_{\sigma}[G_{\alpha}][G^0][G^1]$ and $\mathcal{J}^{K}_{\alpha^1}$ 
respectively coincide. The trees $T_{\alpha}(\xi^{+++})$ as computed 
in $\mathcal{J}^{K}_{\sigma}[G_{\alpha}][G^0][G^1]$ and $\mathcal{J}^{K}_{\alpha^1}$ 
respectively coincide, and they do or do not have cofinal branches simultaneously.
\end{enumerate}
Note that such a $\sigma$ always exists due to the Reflection Principle, and the definability of $K$ in small forcing extensions.

Now let $M$ be from $C \cap \Stat[G_{\alpha}][G^0][G^1]$. We list all the dense subsets $D_n$ which are elements of $M$ and construct
a descending
sequence of $\forceQ^2_{\alpha}$-conditions $p= q_0 > q_1 >...$.
We can demand that for every $k \in \omega$, $q_k$ lies in the corresponding dense set $D_k$.
Further we demand that $sup_{k \in \omega} |dom (q_k)|$ $=\xi= M \cap \omega_1$. If we can ensure 
that the limit of the $q_k$'s, denoted with $q_{\omega}$ is a condition
we would be finished, as $q_{\omega}$ is $(M,\mathbb{Q}_{\alpha}^2)$-generic.
As the first two properties of conditions in $\forceQ_{\alpha}^2$ hold automatically 
for $q_{\omega}$, we have to ensure the third property. Note first that if $\xi < |q_{\omega}|$ 
holds then property 3 will hold automatically as well.
Thus we shall show that the lower bound
$q_{\omega}$ satisfies
\begin{itemize}
\item[$(\ast)$] For any countable, transitive $\ZFP$ model $N$ such that $q_{\omega} \in N$ and $\mu$ an ordinal such that
$(K)^N \subset (K)_{\mu}$, and which is able to see that whenever a real $K$ initial segment $\mathcal{J}^K_{\beta}$ is an
element of $N$ then $N \models$ \textquotedblleft$\mathcal{J}^K_{\beta} \vartriangleleft K$\textquotedblright, further
if $\xi = \omega_1^{{N}} = (\omega_1^{K})^{{N}}$ and $(\omega_1)^{{N}}= (\omega_1)^{(K)_{\mu}}$ then
$q_{\omega}$ codes a $(K)^N$-cardinal $\bar{\alpha}$ and certain generic filters such that if one builds the sequence
of $(K)^N$-cardinals $(\bar{\alpha}^{+\tau})^{(K)^{N}}$, and the according trees 
$(T_{\bar{\alpha}}(\xi^{+++}))^{(K)^{N}}$ then it holds that $\xi \in \dot{S}[G_{\alpha}][G^0][G^1]$ iff
there is
a $\alpha^{+(\xi+3)}$-cofinal branch through $(T_{\bar{\alpha}}(\xi^{+++}))^{(K)^{N}}$ in $N$.
\end{itemize}
Now if we take our countable $M \prec \mathcal{J}^K_{\sigma}[G_{\alpha}][G^0][G^1]$, $M \in \Stat[G_{\alpha}][G_0][G_1]$ such that
$X_{\alpha} \in M$,
by elementarity we have
that the two properties of $\mathcal{J}^K_{\sigma}[G_{\alpha}][G^0][G^1]$ which we listed above still hold for M.
Thus
\begin{enumerate}
\item $M$ thinks that the real initial segment $\mathcal{J}^K_{\alpha^1}$ is coded into $X_{\alpha}$. Further $M$ is also
able to realize that this real initial segment is an initial segment of its own $K$.
\item  $(\omega_1)^{{M}}= ((\omega_1)^{K})^{{M}}$ and moreover $M$ computes the same pattern
of trees
having cofinal branches or not as its own version of $K$. To be more precise the sequence of $K$-cardinals 
starting at $\alpha$ does not depend on whether we compute it in $M$ or in $(K)^M$. 
And likewise for the pattern of trees having cofinal branches or not.
\end{enumerate}
Thus its transitive collapse $\bar{M}$ will decode the same information out of 
$\pi(X_{\alpha})= X_{\alpha} \cap (\omega_1)^{\bar{M}}$, where $\pi$ denotes the collapsing function.
I.e. $\bar{M}$ thinks that $X_{\alpha} \cap (\omega_1)^{\bar{M}}$ 
codes an initial segment of its own version of $K$ together 
with an ordinal $\bar{\alpha}$ and some generically added branches through 
trees, which have all the information to see the intended pattern of 
cofinal or bounded trees $T_{\bar{\alpha}}(\xi^{+++})$ as computed 
in $(K)^{\bar{M}}$. But as $M$ was assumed to be in $\Stat[G_{\alpha}][G^0][G^1]$, and the
iteration so far has been $\Stat(A^c)$-proper, $M \cap K$ will be an element of $\Stat$ and hence collapse to an
initial segment of $K$. 
Thus $\bar{M}$ can compute its own version of $K$ correctly, 
$K^{\bar{M}}= \pi(M \cap K) \in \Stat$ and $X_{\alpha} \cap (\omega_1)^{\bar{M}}$ codes an initial segment $\mathcal{J}^K_{\eta}$ of 
$K$.

Whenever $N$ is a countable, transitive $\ZFP$-model, which satisfies that $N$ 
is elementary with respect to the $K$ defining first order formula $\Phi(x)$,
$K^N \subset \mathcal{J}^K_{\mu}$ for a $\mu < \omega_1$,
$q_{\omega} \in N$, and $\xi = (\omega_1)^N= ((\omega_1)^{(K)})^{N}$, then it 
contains, as the even entries of $q_{\omega}$ code $X_{\alpha}$, the set $X_{\alpha} \cap \xi$.
So if ${N}$ decodes the information packed into
$X_{\alpha} \cap \xi$, it will obtain the $\ZFP$-model from above, namely 
$\mathcal{J}^K_{\eta} \subset (K)^{\bar{M}}$, the $\mathcal{J}^K_{\eta}$-cardinal $\bar{\alpha}$, 
and certain generically added cofinal branches through trees, such that $\mathcal{J}^K_{\eta}$ 
together with the generically added branches read off the intended pattern of bounded and 
cofinal trees. As $N$ was assumed to be able to see each
$K$ initial segment in it as such, $N$ will be able to realize that $\mathcal{J}^K_{\eta}$
 is in fact also an initial segment of its own $K$, thus $\mathcal{J}^K_{\eta} \subset (K)^N$.

The goal is to show that the model $(K)^N$ which the countable model $N$ creates will 
define exactly the same sequence of cardinals starting at $\bar{\alpha}$ and exactly the same 
pattern of cofinal or bounded trees, as the model $\mathcal{J}^K_{\eta}$.
For the pattern which is seen by $\mathcal{J}^K_{\eta}$ is the right one and we would be finished.
In order to show this, we crucially use the definition of $\Stat$.
First note that from our assumptions on $N$ we get that there is a $\mu < \omega_1$ such that 
$(K)^N \subset \mathcal{J}^K_{\mu}$ and $(\omega_1)^N = (\omega_1)^{\mathcal{J}^K_{\mu}}$.
Thus, lining up already shown things, we have that $$\mathcal{J}^K_{\eta} \subset (K)^N \subset \mathcal{J}^K_{\mu}$$ 
and their versions of $\omega_1$ all coincide.
Note now that as $M \in \Stat[G_{\alpha}][G^0][G^1]$, $K$ of the transitive collapse $\bar{M}$ is of the form  
$\mathcal{J}^K_{\zeta}$  for some $\zeta < \omega_1$.
By the definition of $\Stat$ we have that $(K)^{\mathcal{J}^K_{\zeta}} = \mathcal{J}^K_{\zeta}$, and as seen above we also have
$$\mathcal{J}^K_{\eta} \subset \mathcal{J}^K_{\zeta}.$$
Comparing $\zeta$ and $\mu$ we split into cases.

First assume that $\zeta \le \mu$.
Then remember that $\mathcal{J}^K_{\zeta}$ has the property that whenever a larger $K$-initial 
segment has the same $\omega_1$ then already all $\mathcal{J}^K_{\zeta}$-cardinals remain cardinals 
in the larger $K$-initial segment.
As $\mathcal{J}^K_{\mu}$ is of such a form and shares the same $\omega_1$, we infer that all 
cardinals of $\mathcal{J}^K_{\zeta}$ are still cardinals in $\mathcal{J}^K_{\mu}$. But by elementarity 
we have that all cardinals of $\mathcal{J}^K_{\eta}$ are still cardinals in $\mathcal{J}^K_{\zeta}$, 
thus the sequence of cardinals after $\bar{\alpha}$ is the same as evaluated in $\mathcal{J}^K_{\eta}$ and $\mathcal{J}^K_{\mu}$.
And as $\mathcal{J}^K_{\eta} \subset (K)^N \subset \mathcal{J}^K_{\mu}$, the sequence of cardinals after $\bar{\alpha}$
in $(K)^N$
is the same as the sequence of cardinals after $\bar{\alpha}$ in $\mathcal{J}^K_{\eta}$. Thus the sequence of trees
in both models live on the same ordinals. As $(\omega_1)^{\mathcal{J}^K_{\eta}} = (\omega_1)^{(K)^N}$
the pattern of cofinal or bounded trees must coincide as well as otherwise the $(\omega_1)^{(K^N}$ would have been collapsed.

Secondly assume that $\mu > \zeta$. Then $\mathcal{J}^K_{\eta} \subset (K)^N \subset \mathcal{J}^K_{\mu} \subset \mathcal{J}^K_{\zeta}$
and as already noted above, the sequence of cardinals after $\bar{\alpha}$ does not 
depend on the model $\mathcal{J}^K_{\eta}$ or $\mathcal{J}^K_{\zeta}$. But this also implies that 
the sequence of cardinals does not depend on whether we compute them in $(K)^N$ 
or $\mathcal{J}^K_{\eta}$. Thus the trees live on the same ordinals, independently from 
where we construct them, and again, even the pattern of bounded or cofinal 
trees agrees as $(\omega_1)^N = (\omega_1)^{\mathcal{J}^K_{\mu}}$.
This finishes the proof.

\end{proof}

This finishes the definition of the three step iteration we use whenever the $\alpha$-th entry of the $\diamondsuit$-sequence
$a_{\alpha}$ is the name of a stationary subset of $\omega_1$.
What is left is to show that
in the resulting model, the nonstationary ideal $\NS$ is indeed $\Delta_1$-definable and $\NSA$ is $\aleph_2$-saturated.

\section{The definability of $\NS$}
Goal of this section is the proof that in our final model $M_1[G_{\delta}]$ $\NS$ is $\undertilde{\Delta}_1$-definable
over. The parameter will be $K_{\omega_1}$. 
Remember that the only patterns of length $\omega_1$ which occur on 
trees $T(\xi)^K$ in $M_1[G_{\delta}]$ are the ones which code stationary sets.
\begin{Lemma}
If $G$ denotes the generic for the forcing 
notion defined at the beginning of the last 
section. Then in $M_1[G]$, the nonstationary 
ideal is $\Delta_1$-definable over $H(\omega_2)$ using the parameter $K_{\omega_1}$.
\end{Lemma}
\begin{proof}
Let $S \subset \omega_1$ be an arbitrary stationary subset 
of $M_1[G]$. By the inaccessibility of the Woodin cardinal 
$\delta$, which represents the length of the iteration, we 
know that the $\forceP_{\delta}$-name $\dot{S}$ of $S$ is in 
fact a $\forceP_{\beta}$-name for some $\beta < \delta$.
Thus there is a stage $\alpha < \delta$ such
that at stage $\alpha$ the name $\dot{S}$ is considered by the 
$\diamondsuit$-sequence. The rules of our iteration then force 
the characteristic function of $S$ into the sequence of $K$-trees 
starting at the $K$-cardinal $\alpha$, followed by a localization of this information.

We therefore know that there is a set $Y_{\alpha} \subset \omega_1$ such that

\begin{enumerate}
\item[$\heartsuit$] the pattern written on the trees
as computed in $K^{L[Y_{\alpha}][K_{\omega_1}]}$ is the same
for both $L[Y_{\alpha}][K_{\omega_1}]$ and the real world $H(\omega_2)$.
Additionally for every countable, transitive model $M (\in L_{\omega_1}[Y_{\alpha}][K_{\omega_1}])$ of $\ZFP$ which has 
$Y_{\alpha} \cap \omega_1^{M}$ as an element and which satisfies that
$K^M \subset \mathcal{J}^{K}_{\eta}$ and $\omega_1^{M} = \omega_1^{K^M}$ and finally satisfies that 
every $\mathcal{J}^K_{\gamma}$ which is in $M$ is seen by $M$ to be an initial segment of $K^M$, 
there exists an ordinal $\beta \in M$ such that for
every $\xi < \omega_1^{M}:$ $\xi \in S$ if and only if $M \models T_{\beta}(\xi^{+++})$ has a $\beta^{+(\xi+3)}$-cofinal branch.
\end{enumerate}
Note that property $\heartsuit$ can be written as $\Sigma_1$ over 
$H(\omega_2)$ with parameter $K_{\omega_1}$, as saying that the patters are the same
is in our situation equivalent to saying that in $L[Y_{\alpha}][K_{\omega_1}]$ every
$\omega_1$-block of $K$-cardinals has been used to code a subset of $\omega_1$.
As stationarity 
is automatically $\Pi_1$ over $H(\omega_2)$ we see that in $M_1[G_{\delta}]$ 
stationary sets have a $\undertilde{\Delta}_1$ definition over $H(\omega_2)$.
What is left is to show that the definition $\heartsuit$ characterizes stationarity in $M_1[G_{\delta}]$.

For that assume that $S$ is an arbitrary subset of $\omega_1$ and assume that 
there is a set $Y_{\alpha} \subset \omega_1$ such that $\heartsuit$ holds. We want to conclude that
$S$ is indeed stationary.
Note first that $\heartsuit$ also holds for transitive models $M$ which are 
uncountable. Indeed if $M$ would be an uncountable, transitive model such 
that the antecedens of $\heartsuit$ does hold but not its conclusion, 
then we can consider a countable elementary $N \prec M$ which is an element of $\Stat$ 
and collapse it to obtain a countable transitive $\bar{N}$. By the definition of $\Stat$, $\bar{N}$ collapses to a 
$M_1$-initial segment $\bar{N} = \mathcal{J}^{M_1}_{\eta}$ for 
some countable $\eta$. Thus $\bar{N}$ is a countable, transitive 
model $M$ for which $\heartsuit$ fails by elementarity, which is a contradiction. to our assumption that $\heartsuit$ is true.

Hence as $\heartsuit$ holds for models of arbitrary size we can consider large initial segments $\mathcal{J}^{M_1}_{\eta}$
of $M_1$, for which the antecedens of $\heartsuit$ holds trivially. Thus $\mathcal{J}^{M_1}_{\eta}$
witnesses that the pattern of the characteristic function of $S$ is written into trees which are build using $K$ of
$\mathcal{J}^{M_1}_{\eta}$, I.e. the real $K$. But remember that we have set up the iteration $\forceP_{\delta}$ in
such a way that the only patterns of cofinal or bounded trees $T_{\beta}(\xi)^K$ which can arise are the ones
which code stationary subsets of $\omega_1$. Thus the pattern which codes $S$ must stem from a stationary set which
is what we wanted.

\end{proof}

\section{$\NSA$ is $\aleph_2$-saturated}
What is left is to show that we indeed have the restricted nonstationary ideal saturated after forcing with the iteration.
\begin{Lemma}
 In the final model $M_1[G_{\delta}]$ of the iteration the restricted nonstationary ideal $\NSA$ on 
$\omega_1$ is $\aleph_2$-saturated.
\end{Lemma}
\begin{proof}
Note first that for every $\alpha < \delta$ the forcing $\forceP_{\alpha}$ is $\Stat[G_{\alpha}]$-semiproper.
Indeed each factor of the iteration is $\Stat[G_{\alpha}]$-proper, or $\Stat[G_{\alpha}]$-semiproper, 
thus if we iterate with countable support we obtain an $\Stat$-semiproper 
notion of forcing as we have argued at the beginning of the proof in section 2.1.
This suffices to see that the iteration is stationary sets preserving, which
is of highest importance for the following argument. 

Assume now for a contradiction that in $M_1[G_{\delta}]$ there is a 
$\delta =\omega_2^{M_1[G_{\delta}]}$-long antichain of stationary 
subsets of $A$. Denote this antichain with $S$ and let $\tau$ denote its $\forceP_{\delta}$-name.
We claim that it is possible to find
an inaccessible $\kappa$ below $\delta$ such that the following three properties hold: \begin{enumerate}
\item $\kappa$ is $\mathbb{P} \oplus \tau$-strong up to $\delta$ 
in $M_1$ \item $\kappa = \omega_2^{M_1[G\upharpoonright \kappa]}$ 
\item $\vec{S}\upharpoonright \kappa = (S_i \, : \, i < \kappa) = (\tau \cap (M_1)_{\kappa})^{G\upharpoonright \kappa}$
is the maximal antichain
in $M_1[G\upharpoonright \kappa]$ which is picked by the $\diamondsuit$-sequence at stage $\kappa$.
\end{enumerate}

This is clear as we can assume that our $\diamondsuit$-sequence lives on the stationary subset of inaccessible cardinals below $\delta$, and
for all inaccessible $\kappa$ property $2$ automatically holds. Moreover the sets $$C_1:=\{ \kappa < \delta \, :\,
\vec{S}\upharpoonright \kappa
= (S_i \, :\, i< \kappa)= (\tau \cap (M_1)_{\kappa})^{M_1[G\upharpoonright \kappa} \}$$ 
and $$C_2:=\{ \kappa < \delta \, : \forall \alpha < \kappa \forall S \in P(\omega_1) 
\cap M_1^{\mathbb{P_{\alpha}}} \text{ stationary } \exists \bar{S} \in \vec{S}\upharpoonright
\kappa (S \cap \bar{S} \notin NS\}$$ are both clubs, therefore hitting the stationary set $T$ consisting of the points $\kappa < \delta$
where $\tau \cap (M_1)_{\kappa} = a_{\kappa}$ (remember: $a_{\kappa}$ is the $\kappa$-th element of the $\diamond$-sequence
$(a_{\alpha} : \alpha
< \delta )$) and $\kappa$ is $\tau$-strong up to $\delta$. Thus if $\kappa$ is in the nonempty intersection
$C_1 \cap C_2 \cap T$ then $1$ and $2$ are satisfied, and the recursive definition of our forcing $\mathbb{P}$ yields that at stage
$\kappa$, as $a_{\kappa}= \tau \cap (M_1)_{\kappa}$, the sealing forcing 
$\mathbb{S} ((\tau \cap (M_1)_{\kappa} )^{G\upharpoonright \kappa})$ is 
at least considered, and in order to show property $3$, it suffices to 
show that $(\tau \cap (M_1)_{\kappa} )^{G\upharpoonright \kappa})= \bar{S}\upharpoonright \kappa$ 
is maximal in $M_1[G\upharpoonright \kappa]$. But this is clear as by the RCS iteration properties 
we take at inaccessible $\kappa$'s the direct limit of the $\mathbb{P_{\alpha}}$'s, thus 
each stationary $S\subset \omega_1$ in $M_1^{\mathbb{P_{\kappa}}}$ is already included in 
a $M_1^{\mathbb{P_{\alpha}}}$ for $\alpha < \kappa$. So we have ensured the existence of 
a $\kappa$ with all the 3, above stated properties.

Now the forcing $\mathbb{S}(\vec{S}\upharpoonright \kappa)$ can not be $\Stat[G \upharpoonright\kappa]$-semiproper 
at stage $\kappa$, as otherwise we would have to
force with it, therefore killing the antichain $\vec{S}$. So there exists a 
condition $(p,c) \in \mathbb{S}(\vec{S}\upharpoonright \kappa)$
such that the set
\begin{multline*}
\bar{T}:= \{ X \prec (H_{\kappa^+})^{M_1[G\upharpoonright \kappa]} \,:\, 
|X|=\aleph_0 \land (p,c) \in X \land \nexists Y \supset X (Y \prec
(H_{\kappa^+})^{M_1[G\upharpoonright \kappa]} \\ \land |Y|= \aleph_0 
\land (X \cap \omega_1 = Y \cap \omega_1) \land \exists (q,d) \le (p,c)
\, ((q,d) \text{ is Y-semigeneric }))\} \\
\cap \Stat[G \upharpoonright\kappa].
\end{multline*}
is a stationary subset of $\Stat[G \upharpoonright\kappa]$ in
$M_1[G\upharpoonright \kappa]$, and by construction of our iteration, 
the $\kappa$-th forcing in $\mathbb{P}$ is
$Col(\omega_1, 2^{\aleph_2})$, so in $M_1[G\upharpoonright \kappa+1]$ 
there is a surjection $f: \omega_1 \rightarrow (H_{\kappa^+})^{M_1[G\upharpoonright \kappa]}$.
As $Col(\omega_1, 2^{\aleph{_2}})$ is proper the set $\bar{T}$ remains 
stationary in $M_1[G\upharpoonright \kappa+1]$ which implies that 
$$ T:= \{ \alpha < \omega_1 \, : \, f \text{\textquotedblright}\alpha 
\in \bar{T} \land \alpha= f\text{\textquotedblright}\alpha \cap \omega_1\}$$ 
is stationary in $M_1[G\upharpoonright \kappa+1]$. As the tail 
$\mathbb{P}_{[\kappa +2 , \delta]}$ remains $\Stat[G \upharpoonright\kappa+1]$-semiproper, seen as an iteration
with $M_1[G\upharpoonright \kappa+1]$ as ground model, we can infer that $T$ remains stationary in $M_1[G]$ and therefore either
$\lnot A \cap T$ is stationary or there
exists an $i_0< \delta$ such that
$$(\ast\ast) \quad T \cap S_{i_0} \text{ is stationary in } M_1[G].$$

Assume first that $\lnot A \cap T$ is stationary. We want to derive a contradiction.
Fix an $\alpha \in \lnot A \cap T$, and let 
$Z:= f$\textquotedblright$ \alpha \in \bar{T}$. 
We can list inside $Z \prec (H_{\kappa^+})^{M_1[G\upharpoonright \kappa]}$ all 
the $\mathbb{S}(\vec{S}\upharpoonright \kappa)$-names for countable ordinals and 
a descending sequence of $\mathbb{S}(\vec{S}\upharpoonright \kappa)$-conditions 
$(p,c)>(p_0,c_0)>(p_1,c_1),..$ deciding more and more of them.
We can assume that $sup (dom(p_i)) = \alpha$ and thus the lower bound of the 
sequence $(p_i,c_i)_{i \in \omega}$ denoted by $(p',c')$ with
$dom(p')= \alpha +1$ and $p'(\alpha) \in \vec{S} \upharpoonright \kappa$ is in fact already a condition in
$\mathbb{S}(\vec{S}\upharpoonright \kappa)$ as $\alpha \notin A$. Thus we have found a 
$(Z,\mathbb{S}(\vec{S}\upharpoonright \kappa))$-semigeneric condition below $(p,c)$ contradicting $Z \in \bar{T}$.

Thus $(\ast\ast)$ must hold. Let us shortly reflect the situation we are in. The idea is to mimic the short argument from above in
our new case, i.e. we want to find a model $X \in \bar{T}$ such that we $can$ find a $(X,\mathbb{S}(\vec{S}\upharpoonright \kappa))$-
semigeneric condition $(q,d) < (p,c)$, thus arriving at a contradiction.
For that it would be sufficient to find an $X$ such that \begin{enumerate}

\item $X \cap \omega_1 \in T \cap S_{i_0}$ for some $i_0 < \kappa$ 
\item 
$f \textquotedblright (X \cap \omega_1) \subset X$ \item $S_{i_0} \in X$
\end{enumerate}
The first thing to note here is that already the first item can be impossible to fulfill, 
for we can not restrict the size of $i_0$ and it might as well happen that $i_0 > \kappa$,
 and thus there is no chance of finding an $\mathbb{S}(\vec{S}\upharpoonright \kappa)$-condition 
which is $(X, \mathbb{S}(\vec{S}\upharpoonright \kappa))$-semigeneric in the 
way described above.
Here we use the large cardinal property of $\delta$ being Woodin for rescue.
We can find an
elementary embedding $j: M_1 \rightarrow N$ such that $j(\kappa) > i_0$, and which fixes 
the predicate $\vec{S}$ up to a cardinal $\lambda > i_0$.
The task has now changed to find a $(j(X), j(\mathbb{S}(\vec{S}\upharpoonright \kappa))$-semigeneric condition
below $(c,p)$ in order to get a contradiction.

First let $\lambda < \delta$, $\lambda >$ max$(i_0, \kappa +1)$ be such that 
$(\tau \cap (M_1)_{\lambda})^{G\upharpoonright \lambda}= \vec{S}
\upharpoonright \lambda$, so we have $(\tau \cap (M_1)_{\lambda})^{G\upharpoonright \lambda} (i_0) = S_{i_0}$. As $\kappa$
was chosen to be $\mathbb{P}\oplus \tau$-strong up to $\delta$ we let $j: M_1 \rightarrow N$ be an
elementary embedding with critical point $\kappa$, such that $N$ is transitive, $N^{\kappa} \subset N$, $(M_1)_{\lambda + \omega} \subset
N$, $j(\mathbb{P}) \cap (M_1)_{\lambda} = \mathbb{P} \cap (M_1)_{\lambda}$, and $j(\tau) \cap (M_1)_{\lambda} = \tau \cap (M_1)_{\lambda}$.

$H$ should denote the generic filter for the segment 
$(\mathbb{P}_{[\lambda +1, j(\kappa)]})^{N[G\upharpoonright \lambda]}$ of $j(\mathbb{P})$ 
over $N[G\upharpoonright \lambda]$. Then we lift $j$ to an elementary embedding 
$$j^{\ast} : M_1[G\upharpoonright \kappa] \rightarrow N[G\upharpoonright \lambda, H].$$ 
Notice that $((M_1)_{\lambda + \omega})^{M_1[G\upharpoonright \lambda]} = ((M_1)_{\lambda+ \omega})^{N[G\upharpoonright \lambda]}$.

Now we let $(X_k \,:\, k < \omega_1) \in M_1[G\upharpoonright \kappa +1]$ be an increasing continuous chain of
countable elementary substructures
of $(H_{j(\kappa)^{+}})^{N[G\upharpoonright \kappa +1]}$ with $\{ \tau \cap (M_1)_{\lambda}, i_0\} \subset X_0$ satisfying for all
$k < \omega_1$ the following three properties:
\begin{enumerate}
\item[(a)] $k \in X_{k+1}$
\item[(b)] $f$\textquotedblright $(X_k \cap \omega_1) \subset X_k$ \item[(c)] 
$j^{\ast}$\textquotedblright$(X_k \cap (H_{\kappa^{+}})^{M_1[G\upharpoonright \kappa]} \subset X_k$ \end{enumerate}
Let $\bar{G} := G \upharpoonright [\kappa +2, \lambda]$, and consider the set
$$ \{ X_k[\bar{G}] \cap \omega_1 \, : \, k < \omega_1\}$$ which is a club in $M_1[G\upharpoonright \lambda]$.
We can demand that for all $k < \omega_1$, $X[\bar{G}] \cap \omega_1 = X \cap \omega_1$.
And by the stationarity of $T \cap S_{i_0}$ we
obtain an index $k \in \omega_1$ such that
$$X_k[\bar{G}] \cap \omega_1 = X_k \cap \omega_1 \in T \cap S_{i_0}.$$

Write $X:=X_k, \alpha:= X \cap \omega_1$. As at stage $\kappa$ we had to 
force with the $\omega$-closed $Col(2^{\aleph_2}, \aleph_1)$
we know that $X \cap (H_{\kappa^{+}})^{M_1[G\upharpoonright \kappa]} \in M_1[G\upharpoonright \kappa]$.
Remember that $f \in M_1[G\upharpoonright \kappa+1]$ was chosen as a surjection of $\omega_1$ onto 
$(H_{\kappa^{+}})^{M_1[G\upharpoonright \kappa]}$, so as $\alpha \in T$ by definition of 
$T$ $f$\textquotedblright $\alpha
\in \bar{T}$ and $\alpha = f$\textquotedblright $\alpha \cap \omega_1$, and hence by (b) 
$$f\text{\textquotedblright}\alpha \subset X \cap (H_{\kappa^{+}})^{M_1[G\upharpoonright \kappa]} \in M_1[G\upharpoonright \kappa].$$
As $\alpha = f$\textquotedblright$\alpha \cap \omega_1$, $f$\textquotedblright$\alpha \in \bar{T}$ and
$f$\textquotedblright$\alpha \subset X \cap (H_{\kappa^{+}})^{M_1[G\upharpoonright \kappa]}$ we get that
$X \cap (H_{\kappa^{+}})^{M_1[G\upharpoonright \kappa]} \in \bar{T} \subset \Stat[G \upharpoonright\kappa]$ and therefore

$$(\ast\ast\ast) \quad j^{\ast}(X \cap H_{\kappa^{+}})^{M_1[G\upharpoonright \kappa]}) \in j^{\ast} (\bar{T})\subset 
j^{\ast}(\Stat[G \upharpoonright\kappa]).$$

Note that our second generic $H$, denoting the generic filter for the segment 
$(\mathbb{P}_{[\lambda +1, j(\kappa)]})^{N[G\upharpoonright \lambda]}$ of 
$j(\mathbb{P})$ over $N[G\upharpoonright \lambda]$ has not been specified yet.
The segment $(\mathbb{P}_{[\lambda+1, j(\kappa)]})^{N[G\upharpoonright \lambda]}$ 
of $j(\mathbb{P}\upharpoonright \kappa)[G \upharpoonright \lambda]$ over $N[G\upharpoonright \lambda]$ 
is $j^{\ast}(\Stat)[G\upharpoonright \lambda]$-semiproper.
By $(\ast \ast \ast)$, $j^{\ast}(X \cap H_{\kappa^{+}}^{M_1[G\upharpoonright \kappa]}) [\bar{G}] \in j^{\ast}(\bar{T})[\bar{G}]$.
So we can say that as $\bar{T}$ is closed under supersets with the same intersection 
with $\omega_1$ so is $j(\bar{T})$. Hence $X \cap (H_{\kappa^{+}})^{M_1[G\upharpoonright \kappa]} \in \bar{T}$ 
implies that $j^{\ast}(X \cap (H_{\kappa^{+}})^{M_1[G\upharpoonright \kappa]}) \in j^{\ast} (\bar{T})$ 
which in turn implies that $X \in j^{\ast} (\bar{T})$ by property (c), and finally we see 
that $X[\bar{G}] \in j^{\ast}(\bar{T})[G \upharpoonright [\kappa+1, \lambda]] \subset j (\Stat) [G \upharpoonright \lambda]$.
By the $j(\Stat)[G\upharpoonright \lambda]$-semiproperness 
of $(\mathbb{P}_{[\lambda+1, j(\kappa)]})^{N[G\upharpoonright \lambda]}$ we can now 
conclude that any condition $r \in \forceP_{[\lambda+1,j(\kappa)]} \cap X[\bar{G}]$ has 
a stronger condition $q \in \forceP_{[\lambda+1,j(\kappa)]})$ which is $(X[\bar{G}], \forceP_{[\lambda+1,j(\kappa)]})$-semigeneric.
If we pick $H$ such that $q \in H$ then by semigenericity of $q$ we obtain
$X[\bar{G},H] \cap \omega_1 = X[\bar{G}] \cap \omega_1 = X \cap \omega_1 = \alpha \in S_{i_0}= (\tau \cap (M_1)_{\lambda})^{G\upharpoonright
\lambda} (i_0) \in X[\bar{G}, H]$.
But also due to (c) we have that

$$j^{\ast} (X \cap (H_{\kappa^{+}})^{M_1[G\upharpoonright \kappa]}=  j^{\ast} \text{\textquotedblright} (X \cap
(H_{\kappa^{+}})^{M_1[G\upharpoonright \kappa]} \subset X[\bar{G},H].$$

This gives us the desired contradiction as we can find an $(X[\bar{G},H],j(\mathbb{S}(\vec{S}\upharpoonright \kappa)))$-semigeneric
condition below $j(p,c)=(p,c)$. Indeed we can just list the countably many names for countable ordinals in $X[\bar{G},H]$ along with
conditions of $j(\mathbb{S}(\vec{S}\upharpoonright \kappa))$ deciding them below $(p,c)$ and let $(p',c') \in
j(\mathbb{S}(\vec{S}\upharpoonright \kappa))$ be just the condition with dom$(c')=$dom$(d')= \alpha + 1$, $c'(\alpha)=\alpha$ and
$p'(i) = S_{i_0}$ for some $i< \alpha$. So $X[\bar{G},H]$ together with $(p',c')<(p,c)$ witness that
$j^{\ast}(X \cap H_{\kappa^{+}})^{M_1[G\upharpoonright \kappa]}) \notin j^{\ast} (\bar{T})$, contradicting $(\ast\ast\ast)$.
\end{proof}

\chapter{$\NS$ saturated and a projective well-order.}

Goal of this section is the proof of the following result:

\begin{thm}
 Assume that $M_1^{\#}$ exists, then there exists a generic extension $M_1[G]$ of $M_1$ such that
in $M_1[G]$ $\NS$ is $\aleph_2$-saturated and there is a lightface $\Sigma^1_{4}$-definable wellorder 
on the reals.
\end{thm}

Its proof is organized as follows. First we introduce a method which codes reals into triples of ordinals $(\beta, \gamma, \delta)$.
This method was invented by A. Caicedo and B. Velickovic \cite{CV} building on the work of J. Moore on his set mapping reflection principle
(see \cite{Moore}), 
which he used to show that $\BPFA$ implies that the continuum is $\aleph_2$. Their coding is introduced thoroughly in the first
section of this chapter. The downside of this coding is that, as opposed to the coding we used in the last chapter, we do not have any 
control of when exactly a certain code gets written down in the universe. However this difficulty does not prevent
a nice definition of a wellorder on the reals which is localizable, meaning that already small fragments of the universe,
provided they contain a certain amount of information, are able to witness the coding. The relevant defintions
for that are done in the second section and
the theorem is finally proved in the last section, where we combine the coding forcings with 
forcings which seal off the long antichains in $P(\omega_1) \slash \NS$.

\section{The coding method}

The definition of the coding is a rather convoluted one. We have to introduce a couple of definitions and start with

\begin{Definition}
A $\vec{C}$-sequence, or a ladder system, is a sequence $(C_{\alpha} \, : \, \alpha \in \omega_1, \alpha \text{ a limit ordinal })$, such
that for every $\alpha$, $C_{\alpha} \subset \alpha$ is cofinal and the ordertype of $C_{\alpha}$ is $\omega$.
\end{Definition}

For three subsets $x,y,z \subset \omega$ we can consider the oscillation function. First turn use the set $x$ 
into an equivalence relation $\sim_x$, defined on the set $\omega- x$ as follows: for natural numbers in the 
complement of $x$ satisfying $n \le m$ let $n \sim_x m$ if and only if $[n,m] \cap x = \emptyset$.
This enables us to define:
\begin{Definition}
For a triple of subset of natural numbers $(x,y,z)$ list the intervals $(I_n \, :\, n \in k \le \omega)$ of equivalence classes of
$\sim_x$ which have nonempty intersection with both $y$ and $z$. Then the oscillation map $o(x,y,z,):
k \rightarrow 2$ is defined to be the function satisfying

\begin{equation*}
o(x,y,z)(n) = \begin{cases}
0  & \text{ if min}(I_n \cap y) \le \text{min}(I_n \cap z) \\ 1 & \text{ else}
\end{cases}
\end{equation*}

\end{Definition}

Next we want to define how suitable countable subsets of ordinals can be used to code reals. 
For that suppose that $\omega_1 < \beta < \gamma < \delta$ are fixed limit ordinals, and that 
$N \subset M$ are countable subsets of $\delta$.
Assume further that $\{ \omega_1, \beta, \gamma\} \subset N$ and that for every 
$\eta \in \{ \omega_1, \beta, \gamma\}$, $M \cap \eta$ is a limit ordinal and $N \cap \eta < M \cap \eta$.
We can use $(N,M)$ to code a finite binary string. Namely let $\bar{M}$ denote the transitive collapse of 
$M$, let $\pi : M \rightarrow \bar{M}$ be the collapsing map and let 
$\alpha_M := \pi(\omega_1)$, $\beta_M := \pi(\beta), \, \gamma_M := \pi(\gamma) \, \delta_M:= \bar{M}$. 
These are all countable limit ordinals.
Further set $\alpha_N:= sup(\pi``(\omega_1 \cap N))$ and let the height $n(N,M)$ of $\alpha_N$ in $\alpha_M$ 
be the natural number defined by

$$n(N,M):= card (\alpha_N \cap C_{\alpha_M})$$ where $C_{\alpha_M}$ is an element of our previously fixed ladder system. 
As $n(N,M)$ will appear quite often in the following we write shortly $n$ for $n(N,M)$. Note that
as the ordertype of each $C_{\alpha}$ is $\omega$, and as $N \cap \omega_1$ is bounded below $M \cap \omega_1$,
$n(N,M)$ is indeed a natural number.
Now we can assign to the pair $(N,M)$ a triple $(x,y,z)$ of finite subsets of natural numbers as follows:
$$x:= \{ card(\pi(\xi) \cap C_{\beta_M}) \, : \, \xi \in \beta \cap N \}.$$ Note that $x$ again is finite as $\beta \cap N$ is 
bounded in the cofinal in $\beta_M$-set $C_{\beta_M}$, which has ordertype $\omega$. Similarly we define 
$$y:= \{ card(\pi(\xi) \cap C_{\gamma_M}) \, : \, \xi \in \gamma \cap N \}$$ and
$$z:= \{ card(\pi(\xi) \cap C_{\delta_M} \, : \, \xi \in \delta \cap N \}.$$ Again it is easily 
seen that these sets are finite subsets of the natural numbers.
We can look at the oscillation $o(x \backslash n, y \backslash n, z \backslash n)$ (remember we let $n:= n(N,M)$)
and if the oscillation function at these points has a domain bigger or equal to $n$ then we write
\begin{equation*}
s_{\beta, \gamma, \delta} (N,M):= \begin{cases}
o(x \backslash n, y \backslash n, z \backslash n)\upharpoonright n & \text{ if defined } \\ \ast \text{ else}
\end{cases}
\end{equation*}
Similarly we let $s_{\beta, \gamma, \delta} (N,M) \upharpoonright l = \ast$ when $l > n$.
The seemingly arbitrary move in considering the oscillation of the triple $(x \backslash n, y \backslash n, z\backslash n)$ 
instead of just the oscillation of $(x,y,z)$ will become clear after the next definition.
Finally we are able to define what it means for a triple of ordinals $(\beta, \gamma, \delta)$ to code a real $r$.

\begin{Definition}
For a triple of limit ordinals $(\beta, \gamma, \delta)$, we say that it codes a real $r \in 2^{\omega}$
if there is a continuous increasing sequence $(N_{\xi} \, : \, \xi < \omega_1$ of countable sets of ordinals whose 
union is $\delta$ and which satisfies that whenever $\xi < \omega_1$ is a limit ordinal then there is a $\nu < \xi$ such that
$$ r = \bigcup_{\nu < \eta < \xi} s_{\beta, \gamma, \delta} (N_{\eta}, N_{\xi}) $$ \end{Definition}

As a short remark, note that if we would define the $s_{\beta, \gamma, \delta}$ function with the full oscillation 
of $(x,y,z)$ instead, then our above definition would become useless. Indeed if $N_{\xi}$ is a continuous increasing 
sequence whose union is $\delta$ then for some fixed limit $\xi < \omega_1$, and if we construct the sets $x,y,z$ for 
the pair $(N_{\nu}, N_{\xi})$, $\nu < \xi$, the set $x$ would for increasing $\nu$ eventually
just become an enumeration of $card( \eta \cap C_{\beta_M})$ for $\eta \in \beta \cap N_{\nu}$. Likewise the sets $y$ 
and $z$ would just become eventually an enumeration of the natural numbers, and thus completely useless for any coding purpose.
The idea to overcome this, is to constantly throw away already coded information, namely the $n$, and code
the real $r$ over and over again into the oscillation pattern.

Our next task is to combine the forcings which make $\NS$ $\aleph_2$-saturated, and the forcings which will 
generically add witnesses for ordinal triples coding reals $r$ in such a way that in the end a projective well-order. 
of the reals is possible.
For that we should take a closer look at the forcings which add coding witnesses.
These forcings have in common that their generics exist, assuming the so called Mapping Reflection Principle ($\MRP$), 
which was introduced by Justin Moore. The $\MRP$ is a forcing axiom whose consistency strength lies in between the proper 
forcing axiom $\PFA$ and the bounded proper forcing axiom $\BPFA$. Moore famously used the $\MRP$ to show that already the 
bounded proper forcing axiom suffices to decide the value of the continuum, namely $\ZFC + \BPFA \vdash 2^{\aleph_0} = \aleph_2$.
We follow Caicedo Velickovic in proving that
$\MRP$ implies that for every real $r$ there is a triple of ordinals $(\beta, \gamma, \delta)$ such that 
$\omega_1 < \beta < \gamma < \delta < \omega_2$ which code the real $r$. In fact the witnesses for the
 coding can be added by small proper forcings, which will allow us to combine these with the forcings 
which seal off long antichains in $P(\omega_1) \backslash \NS$.
For the definition of the $\MRP$ we need the following local version of stationarity: \begin{Definition}
Let $\theta$ be a regular cardinal, $X$ be an uncountable set, let $M \prec H_{\theta}$ be a countable 
elementary submodel which contains $[X]^{\omega}$. Then $S \subset [X]^{\omega}$ is $M$-stationary if 
for every club subset $C$ of $[X]^{\omega}$, $C \in M$ it holds that $$C \cap S \cap M \ne \emptyset.$$
\end{Definition}

\begin{Definition}
Let $X$ be an uncountable set, $N \in [X]^{\omega}$ and $x \subset N$ finite. Then the Ellentuck 
topology on the set $[X]^{\omega}$ is generated by base sets of the form
$$ [x,N]:= \{ Y \in [X]^{\omega} \, : \, x \subset Y \subset N \}.$$ From now on whenever we say 
open we mean open with respect to the Ellentuck topology.
\end{Definition}

\begin{Definition}
Let $X$ be an uncountable set, let $\theta$ be a large enough regular cardinal so that
 $[X]^{\omega} \in H_{\theta}$. Then a function $\Sigma$ is said to be open stationary 
if and only if its domain is a club $C \subset [H_{\theta}]^{\omega}$ and for every 
countable $M \in C$, $\Sigma(M) \subset [X]^{\omega}$ is open and $M$-stationary.
\end{Definition}
Equipped with these notions we can introduce the mapping reflection principle:

\begin{Definition}
Let $\Sigma$ be an open stationary function with domain some club $C \subset [H_{\theta}]^{\omega}$
and range $P([X]^{\omega})$.
Then there is a continuous sequence of models $(N_{\xi} \, : \, \xi < \omega_1)$ in $dom(\Sigma)$ 
such that for every limit ordinal $\xi$ there is a $\nu < \xi$ such that for every $\eta$ with 
$\nu < \eta < \xi$, $N_{\eta} \cap X \in \Sigma(N_{\xi})$.
\end{Definition}

Key here is that these sequences of models which witness the truth of the $\MRP$ can always be forced with a proper forcing:

\begin{Proposition}
$\PFA$ proves $\MRP$.
\end{Proposition}
\begin{proof}

The goal is to show that the natural forcing which adds a continuous sequence 
of models witnessing the $\MRP$ for a stationary set mapping $\Sigma$ is 
always proper. So given such a function $\Sigma$ with domain 
$C \subset [H_{\theta}]^{\omega}$ and range $P([X]^{\omega})$ 
we let $\forceP_{\Sigma}$ be the partial order whose elements 
are functions $p: \alpha +1 \rightarrow dom(\Sigma)$, $\alpha$ countable, 
which are continuous and $\in$-increasing, and which additionally satisfy the 
$\MRP$-condition on its limit points, namely that for every $0< \nu < \alpha$ 
there is a $\nu_0 < \nu$ such that $p(\xi) \cap X \in \Sigma(p(\nu))$ for every 
$\xi$ in the interval $(\nu_0, \nu)$. The order is by extension.
The first thing to note is that sets of the form
$ D_{\alpha} := \{ p \in \forceP_{\Sigma} \, : \, \alpha \in dom(p) \}$ are always dense. This is true as the trivially dense
sets $D_{x} := \{ p \in \forceP_{\Sigma} \, : \, \exists \beta \in dom(p) (x \in p(\beta))$ 
ensure that whenever we force with $\forceP_{\Sigma}$ there will be a surjection from 
$\{ \alpha \, : \, \exists p \in G (\alpha \in dom(p) \}$ onto the uncountable $X$. 
Thus once we show that the forcing $\forceP_{\Sigma}$ is proper, and 
therefore $\omega_1$-preserving the $\omega_1$-many dense sets 
$D_{\alpha}$ and $\PFA$ will give the desired reflecting sequence. 
Note that we will not use the density of the $D_{\alpha}$'s to show that $\forceP_{\Sigma}$ is proper, so we avoid a circular reasoning.

To see that $\forceP_{\Sigma}$ is proper we fix a large enough 
cardinal $\lambda$ and a countable, elementary submodel 
$M \prec H_{\lambda}$ which contains $\Sigma$, $\forceP_{\Sigma}$, a 
condition $p \in \forceP_{\Sigma}$, and the structure $H_{|\forceP_{\Sigma}|^+}$. 
We list all the dense subsets $(D_0, D_1,..)$ of $\forceP_{\Sigma}$
which we can find in $M$ and build by recursion a descending sequence of conditions 
$(p_i \, : \, i \in \omega)$ in $M$, starting at $p_0:= p$ hitting the corresponding $D_{i-1}$. 
Assume that we have already built conditions up to $i \in \omega$. 
We let $N'_i$ be a countable elementary submodel of $H_{|\forceP_{\Sigma}|^+}$ 
containing $H_{\theta}$, $D_i$, $\forceP_{\Sigma}$ and $p_i$, and 
build the club of countable structures $C_i:= \{ N'_i \cap X \, : \, N'_i \text{ as just described} \}$. 
Note that this club will be in $M$, and further that for every club on $[H_{\theta}]^{\omega}$ 
which is in $M$, the set $M \cap H_{\theta}$ will be in the club.
Thus the set $M\cap H_{\theta}$ will be in
the domain of $\Sigma$ and by the definition of 
$\Sigma$, the set $\Sigma(M \cap H_{\theta})$ is $M\cap H_{\theta}$-stationary and open. 
So there is an $N_i \in C_i \cap \Sigma(M \cap H_{\theta}) \cap M$, and by the definition of the Ellentuck topology, 
there is a finite subset of $N_i$ called $x_i$ such that $[x_i, N_i] \subset \Sigma(M \cap H_{\theta})$.
We first extend the condition $p_i$ to $q_i := p_i \cup \{ \zeta_i +1, hull^{H_{\theta}} (p_i(\zeta_i) \cup x_i))\}$, 
for $\zeta_i$ the maximum of the domain of $p_i$. This condition $q_i$ will also be in $N'_i$ as all its defining 
parameters are, thus as $N'_i$ also contains $D_i$ we can extend the condition $q_i$ to a $p_{i+1} \in N'_i \cap D_i$. 
Note that as we are working in $N'_i$, no matter how we extend $q_i$, the range of the extended condition intersected with $X$ will always
be contained in $N_i=N'_i \cap X$, and as $\Sigma(M \cap H_{\theta} \supset [x_i, N_i]$, 
it will also be contained in $\Sigma(M \cap H_{\theta}$. 
Then if we set $p_{\omega} := \bigcup_{i \in \omega} p_{i} \cup (\omega, (M\cap H_{\theta}))$ 
then this will be a condition in $\forceP_{\Sigma}$, which is by construction 
below $p$ and $(M,\forceP_{\Sigma})$-generic, thus the forcing is proper.

\end{proof}
We can use the $\MRP$ to show that in its presence the function $s_{\beta \gamma \delta}$ eventually becomes stable in the following way:

\begin{Proposition}
Assume that the $\MRP$ does hold and let $\omega_1<\beta< \gamma< \delta<\omega_2$ be limit ordinals with
uncountable cofinality. Then there is a continuous sequence of countable sets of ordinals $(N_{\xi} \, : \, \xi < \omega_1)$ 
such that $\bigcup_{\xi < \omega_1} N_{\xi} = \delta$ and such that for every limit ordinal $\xi < \omega_1$ and 
every $n \in \omega$ there is a $s_{\xi}^{n} \in 2^n \cup \{ \ast \}$ and
there is a $\nu < \xi$ such that
$$s_{\beta \gamma \delta} (N_{\eta}, N_{\xi}) \upharpoonright n = s_{\xi}^n$$ for every $\eta$ in the interval $(\nu, \xi)$.
Moreover already $\BPFA$ suffices for the conclusion.
\end{Proposition}

The proof is a consequence of the following
\begin{Lemma}
Again assume the $\MRP$ and suppose that for every $\alpha < \omega_1$ we have a partition of 
$\alpha$ into two sets $K^i_{\alpha}$, $i \in 2$ such that both are clopen sets in 
the usual order topology on $\omega_1$. Then there is a club $C \subset \omega_1$ such 
that for every limit point $\xi$ of $C$ there is an $i \in 2$ such that $C \backslash K^i_{\xi}$ is bounded below $\xi$.
Already $\BPFA$ does suffice for the conclusion.
\end{Lemma}
\begin{proof}
Consider $\omega_1$ as a subset of $[\omega_1]^{\omega}$, then it is easy to see that the property ''not an ordinal 
in $[\omega_1]^{\omega}$`` is an open set in the Ellentuck topology, hence $\omega_1$ is closed in $[\omega_1]^{\omega}$.
The restriction of the Ellentuck topology to $\omega_1$ yields the usual order topology on $\omega_1$.

If $M \prec H_{\theta}$ is countable and $\alpha = M \cap \omega_1$ then 
one of the $K^i_{\alpha}$'s is $M$-stationary. Indeed assume not then there 
are two clubs $C_0$ and $C_1$ in $M$ which have empty intersection with 
$K^0_{\alpha}$ and $K^1_{\alpha}$ respectively. But then $C_0 \cap C_1 = \emptyset$ in $M$ and 
by elementarity in $H_{\theta}$ which is a contradiction.
Now set
$$\Sigma(M) := [\omega_1]^{\omega} \backslash K^i_{\alpha}$$ where $K^{1-i}_{\alpha}$ is the stationary part of the partition.
As this is an open stationary map we can apply the $\MRP$ to obtain a sequence $(N_{\xi} \, : \, \xi < \omega_1)$ 
such that for all limits $\xi$ there is a $\nu < \xi $ such that
$N_{\eta} \cap \omega_1 \in \Sigma(N_{\xi})$ for every $\eta \in (\nu, \xi)$.
Then $C:=\{ N_{\xi} \cap \omega_1 \, : \, \xi < \omega_1 \}$ is a club as required.

To see that already $\BPFA$ suffices for the conclusion note that the statement of the Lemma is of the 
form $\Sigma_1$ with parameters $K^i_{\alpha} \subset \omega_1$. We have already seen that each 
instance of the $\MRP$ can be forced by a proper forcing.
As $\BPFA$ is equivalent to the assertion that
$H_{\aleph_2} \prec V^{\forceP}$ for every proper $\forceP$ we see that already $\BPFA$ is sufficient for the statement.
\end{proof}
\begin{proof}[proof of the Proposition]
Consider a triple $(\omega_1 < \beta< \gamma < \delta< \omega_2)$ of cofinality $\omega_1$ and 
let $(N_{\xi} \, : \, \xi < \omega_1)$ be a continuous sequence of countable sets of ordinals whose union is $\delta$.
We can assume without loss of
generality that $\{ \omega_1, \beta, \gamma\} \subset N_0$ and that $N_{\eta} \cap \omega_1 < N_{\xi} \cap \omega_1$ 
holds for $\eta < \xi$ and similar inequalities hold for $\beta$ and $\gamma$.
For a fixed natural number $n \in \omega$ and an ordinal $\xi < \omega_1$ we can apply the previous Lemma 
to obtain a partition of $\xi$ into $2^n +1$ many
sets $\{ K_{\xi}^{i} \, : \, i \in 2^{n+1} \}$
via setting
$$ \eta \in K^{i}_{\xi} \text{ if and only if } s_{\beta \gamma \delta}(N_{\eta}, N_{\xi}) \upharpoonright n = s_i$$
for $s_i$ the $i$-th element of $2^{n} \cup \{\ast\}$ in some fixed well-order.
This partition consists even of clopen sets as already finitely many elements $\{n_0. n_1,...,n_k\}$ of 
$N_{\xi}$ suffice that $s_{\beta \gamma \delta} (\{n_0,...,n_k \}, N_{\xi}) = s_{\beta \gamma \delta} (N_{\eta}, N_{\xi})$ 
and thus each $K^{i}_{\xi}$ is clopen.
We apply the previous Lemma to obtain for every $n \in \omega$ a club $C_n \subset \omega_1$ such that for every 
limit point $\xi \in C$ eventually all points of $C$ below $\xi$ are in some fixed $K^{i}_{\xi}$. To $K^{i}_{\xi}$ 
corresponds an $s_{\xi}^{n} \in 2^{n} \cup \{\ast\}$ and we have that for every limit $\xi \in C_n$ there is an 
$\eta < \xi$ such that $s_{\beta \gamma \delta} (N_{\eta}, N_{\xi}) \upharpoonright n = s_{\xi}^{n}$ for every $\eta \in (\nu,\xi)$.
If we set $C:= \bigcap_{n \in \omega} C_n$ then the sequence $(N_{\xi} \, : \, \xi \in C )$ is as desired.
\end{proof}

We will use the $\MRP$ to show that for every real $r$ there is a proper notion of forcing which will introduce a 
sequence of countable models which witnesses that the real $r$ is coded by a triple $(\beta, \gamma, \delta)$ 
of limit ordinals below $\aleph_2$. The procedure is the same as before, namely we show first that $\MRP$ 
implies that every real $r$ is coded by a triple of ordinals, and then show that in fact $\BPFA$ suffices for the conclusion.
We start with fixing an $F : [\omega_4]^{<\omega} \rightarrow \omega_4$.
We will define a sequence of two player games $\mathcal{G}_{\nu}^{F}$ for $\nu < \omega_1$.

\begin{table}[h]
\begin{tabular}{l|l}
I & $\beta_0$ \qquad \qquad $\gamma_0$ \qquad \qquad $\delta_0$ \qquad \qquad $\beta_1$ \qquad \qquad  $\gamma_1$... \\
\hline
II & \qquad  $\kappa_0, \epsilon_0$ \qquad \quad  $\lambda_0, \vartheta_0$ \qquad $\mu_0, \varpi_0$ \qquad  $\kappa_1, \epsilon_1$...
   
\end{tabular}
\end{table}
The rules of the game $\mathcal{G}_{\nu}^{F}$ are as follows: player I starts by playing an ordinal
$\beta_0 < \omega_2$, then player II responds with a pair of ordinals $\kappa_0 \le \epsilon_0 < \omega_2$,
$\beta_0 \le \kappa_0$.
In the next round player I plays an ordinal $\gamma_0 < \omega_3$ and II responds with playing 
ordinals $\lambda_0 \le \vartheta_0 < \omega_3$ such that $\gamma_0 \le \lambda_0$.
This is followed by player I picking
a $\delta_0 < \omega_4$ and II playing $\mu_0 \le \varpi_0 < \omega_4$ with $\delta_0 \le \mu_0$.
Then I goes back to play below $\omega_2$: he picks a $\beta_1< \omega_2$ greater than $\epsilon_0$
and II responds with a pair $\kappa_1, \epsilon_1$ such that $\beta_1 \le \kappa_1$, and the 
game continues with I playing below $\omega_3$, II responding, I playing below $\omega_4$, II 
responding and so on. The first player who violates the rules loses. Otherwise the game will 
determine an infinite sequence of ordinals and we set
$$ X := cl_{F} (\{ \kappa_n, \lambda_n, \mu_n \, : \, n\in \omega \} \cup \nu \})$$ where 
$cl_{F}$ should denote the closure under the function $F$ and $\nu < \omega_1$ is the index of the $\nu$-th game $\mathcal{G}_{\nu}^{F}$.
We say that player II wins if the following conditions hold: \begin{itemize}
\item $X\cap \omega_1 = \nu$
\item $X \cap [\beta_n, \beta_{n+1}) \subset [\beta_n, \epsilon_n)$ for all $n \in \omega$.
\item $X \cap [\gamma_n, \gamma_{n+1}) \subset [\gamma_n, \vartheta_n)$ for all $n \in \omega$.
\item $X \cap[\delta_n, \delta_{n+1}) \subset [\delta_n, \varpi_{n})$ for all $n \in \omega$.
\end{itemize}
If player I wins then he has won already after finitely many stages so by the Gale-Stewart 
theorem $\mathcal{G}_{\nu}^{F}$ is determined for every $\nu < \omega_1$.
Player I wins $\mathcal{G}_{\nu}^{F}$ only for few $\nu < \omega_1$ as is shown now.
Let
$$A_{F}:= \{ \nu < \omega_1 \, : \, \text{ player $\I$ has a winning strategy in } \mathcal{G}_{\nu}^{F}  \}$$
then
\begin{Lemma}
$A_F$ is nonstationary.
\end{Lemma}

\begin{proof}
The proof is by contradiction, thus assume that $A_F$ is stationary for some function 
$F: [\omega_4]^{\omega} \rightarrow \omega_4$. First notice that there is a winning 
strategy for $\I$ which works for all $\nu \in A_F$ simultaneously.
Indeed we can define a strategy $\sigma$ for $\I$ which maps a given state $s$ of the game to the supremum of all the ordinals of
the form $\sigma_{\nu}(s)$, $\nu \in A_F$, where $\sigma_{\nu}$ is a winning strategy for $\I$. 
This strategy $\sigma$ is again a winning strategy for $\I$ as a supposed play of the 
game where $\I$ follows $\sigma$ but nevertheless yields a victory of $\II$ will 
immediately give $\II$ a way of winning against any $\sigma_{\nu}$, contradicting 
the fact that $\sigma_{\nu}$ is a winning strategy for $\I$.

A similar consideration shows that we can further assume that whenever player $\I$ 
has to pick an ordinal $\gamma_{n+1} < \omega_3$, he can do so without looking at 
what player $\II$ decided to play for $\kappa_n, \epsilon_n$. Indeed note that after 
all there are only $\aleph_2$-many choices for $\II$ to make and $\I$ can simply play 
the supremum of all the answers, which again is a winning strategy for $\I$.
Likewise for playing $\delta_{n+1} < \omega_4$, $\I$ does not have to care about the choices of $\II$ for $\lambda_n$ and $\vartheta_n$.
We will assume that the winning strategy $\sigma$ for player $\II$ has the above described properties

Now we describe a way for $\II$ to win against the strategy $\sigma$, thus giving us the desired contradiction.
We first fix a large enough regular $\theta$
and an increasing, continuous elementary chain of
$\aleph_1$-sized models of length $\omega_1 \cdotp \omega$

$$P_0 \prec P_1 \prec ...P_{\xi} \prec...\prec H_{\theta}, \, \, \,\,\, \xi < \omega_1 \cdotp \omega$$ 
such that $F, A_F, \sigma \in P_0$. Let $N_n := P_{\omega_1 \cdotp n}$ and let $N$ be the union of the $N_n$'s. Let
\begin{enumerate}
\item[] $\zeta_n := sup(N_n \cap \omega_2)$
\item[] $\eta_n := sup(N_n \cap \omega_3)$
\item[] $\theta_n := sup(N_n \cap \omega_4)$
\end{enumerate}
then all of these ordinals have cofinality $\omega_1$ and the club subsets of $N_n$ which 
climb up to them are elements of $N_{n+1}$ as the sequence
is an elementary chain. We fix a countable $M \prec N$ containing all the just mentioned clubs and such that
$\{\zeta_n \, : \, n \in \omega\} \subset M$,
$\{ \eta_n  \, : \, n \in \omega\} \subset M$,
$\{ \theta_n   \, : \, n \in \omega\} \subset M$ and
$F, A_F, \sigma \in M$ holds. Further by the
assumed stationarity of $A_F$ we can demand that $\alpha := M \cap \omega_1 \in A_F$.
We will describe now a game where $\I$ follows $\sigma$ which will nevertheless result in a victory for $\II$.
Assume that we are in the $n$-th round and the
position $p_{n-1}$ of the game looks like this
$p_{n-1} = (\beta_0, (\kappa_0, \epsilon_0), \gamma_0,...,\delta_{n-1}, (\mu_{n-1}, \varpi_{n-1}))$.
Assume further that $p_{n-1} \in N_n$.
Then $\I$ plays according to $\sigma$, thus $\beta_n=\sigma(p_{n-1})$ 
and as both $\sigma$ and $p_{n-1}$ are in $N_n$, $\beta_n$ will also be in $N_n$. Then $\II$ will follow by choosing
$\kappa_n \in M \cap N_{n+1}$ such that $\zeta_n \le \kappa_n < \omega_2$.
Let $\epsilon_{n} := sup (M \cap N_{n+1} \cap \omega_2)$. Recall that 
$\zeta_{n+1} = sup (N_{n+1} \cap \omega_2)$ has cofinality $\omega_1$ and as $M$ is countable, $\epsilon_{n} < \zeta_{n+1}$.
As $N_{n+1}$ contains a club set which converges to $\zeta_{n+1}$ and this club is in $M$ 
we conclude that $\epsilon_n \in N_{n+1}$.
Player $\II$ plays $(\kappa_n, \epsilon_n)$.
Player $\I$ then answers with $\sigma(p_{n-1}\, (\beta_n, (\kappa_n, \epsilon_n))) = \gamma_n$.
As $\sigma$ was assumed to be not dependent on the previous choice of player $\I$ we know that
$\gamma_n \in N_{n}$. $\II$ responds
with picking a $\lambda_n \in M \cap N_{n+1}$ with $\eta_n \le \lambda_n < \omega_3$ and 
$\vartheta_n := sup(M \cap N_{n+1} \cap \omega_3).$ The same reasoning as above yields that
$\vartheta_n < \eta_{n+1}$ and $\vartheta_n \in N_{n}$.
Player $\I$ the answers with $\delta_{n}$ with
the help of the strategy $\sigma$. Again, as $\sigma \in N_n$, and as $\sigma$ does 
not depend on player $\I$'s choice of $\lambda_n$ and $\vartheta_n$, $\delta_n \in N_n$.
Player $\II$ then follows with playing $\mu_n \in M \cap N_{n+1}$ such that 
$\theta_n \le \mu_n < \omega_4$. If we let $\varpi_n := sup (M \cap N_{n+1} \cap \omega_4)$ 
then again $\varpi_n < \theta_{n+1}$ and $\varpi_n \in N_{n+1}$. $\II$ then picks $(\mu_n, \varpi_n)$.
This defines the next stage $p_n$ of the game and by summing up already shown things we see that
$p_n \in N_{n+1}$.
If we let
$$X := cl_F (\{ \kappa_i, \lambda_i, \mu_i \, : \, I \in \omega \} \cup  \alpha )$$ 
then $X \subset M$ as all the relevant information is present in $M$.
This implies that $X \cap \omega_1 = \alpha$. By construction we also have that 
$\beta_{n+1} \in N_{n+1}$ and $\epsilon_n= sup(M \cap N_{n+1} \cap \omega_2)$, 
thus $X \cap [\beta_n, \beta_{n+1}) \subset X \cap [\beta_n, \epsilon_n)$ for every $n$.
Likewise we have that for every $n$, $X \cap [\gamma_n, \gamma_{n+1}) = X \cap [\gamma_n, \vartheta_n)$
and $X \cap [\delta_n, \delta_{n+1}) = X \cap [\delta_n, \varpi_n)$.
Thus $\II$ has won the game $\mathcal{G}_{\alpha}^{F}$, however $\alpha \in A_F$ which is a contradiction.

\end{proof}

If we assume that $\MRP$ holds then arbitrary reals $r$ can be coded into triples of ordinals, as is shown now:

\begin{Lemma}
If $\MRP$ holds then every real $r$ is coded into a triple $(\beta, \gamma, \delta)$ such 
that the each element of the triple has cofinality $\omega_1$ and $\omega_1 < \beta < \gamma < \delta < \omega_2$.
\end{Lemma}
\begin{proof}
Let $r$ be an arbitrary real. We let $\theta$ be
a regular cardinal which is large enough, and let the open stationary set mapping $\Sigma^r$ to be
defined on the club of countable elementary submodels $M \prec H_{\theta}$ with values set as follows:

$$ \Sigma^r (M) := \{ N \in [M \cap \omega_4]^{\omega} \, :\, s_{\omega_2,\omega_3,\omega_4}(N, M \cap
\omega_4) \text{ is an initial segment of } r \} $$
We have to check that $\Sigma^r$ is indeed open and stationary. Openness 
is clear as already finite information  $T \subset N$ suffices to determine the value of $s_{\omega_2,\omega_3,\omega_4}(N, M \cap
\omega_4)$. Thus if $N \in \Sigma^r (M)$ then
already an open set of the form $[m,N]$ will be a subset of $\Sigma^r (M)$.
To show that $\Sigma^r$ is stationary needs more work and the proof of this will be postponed. 
Instead we finish the proof of the lemma under the assumption that $\Sigma^r$ is open stationary.

Using $\MRP$ we obtain a reflecting sequence $(M_{\xi} \, :\, \xi < \omega_1 )$ for $\Sigma^r$. 
Let $M := \bigcup_{\xi < \omega_1} M_{\xi}$ and let $\bar{M}$ be its transitive collapse with $\pi$ the collapsing map. 
Let $\beta= \pi(\omega_2), \gamma = \pi(\omega_3), \delta= \pi(\omega_4)$. All these ordinals are in the interval 
$(\omega_1, \omega_2)$ and their cofinalities are $\omega_1$. Let $N_{\xi}:= \pi(M_{\xi} \cap \omega_4)$ and we get 
that $$s_{\beta, \gamma, \delta} (N_{\eta}, N_{\xi}) = s_{\omega_2, \omega_3, \omega_4} (M_{\eta} \cap \omega_4, M_{\xi} \cap \omega_4)$$ 
for every $\eta < \xi <\omega_1$.

Moreover for every $\xi$ limit we have continuity in that $N_{\xi} = \bigcup_{\eta < \xi} N_{\eta}$.
Which implies that the $n(N_{\eta}, N_{\xi})$ converges to $\omega$ for $\eta \rightarrow \xi$, where as a 
short reminder $n(N_{\eta}, N_{\xi})$ was defined to be the size of the set 
$C_{\pi(\omega_1)} \cap \pi \textquotedblright (\omega_1 \cap N_{\eta})$ for $\pi$ the collapsing function of $N_{\xi}$. As
$s_{\beta, \gamma, \delta} (N_{\eta}, N_{\xi}) =
s_{\omega_2, \omega_3, \omega_4} (M_{\eta} \cap \omega_4, M_{\xi} \cap \omega_4)$, and as 
$(M_{\xi} \, : \, \xi < \omega_1$ is a reflecting sequence for $\Sigma^r$ we know that there is a $\nu < \xi$ such that
$$r = \bigcup_{\nu < \eta < \xi} s_{\beta, \gamma, \delta} (N_{\eta}, N_{\xi}). $$ 
This ends our proof as the triple $(\beta, \gamma, \delta)$ code $r$ as desired.

\end{proof}

\begin{Lemma}
Already $\BPFA$ suffices for the conclusion of the last Lemma \end{Lemma}
\begin{proof}
We will use Bagaria's characterization of $\BPFA$ again, namely that $\BPFA$ is equivalent 
to the assertion that $H_{\aleph_2} \prec_1 V^{\forceP}$ for every proper $\forceP$. 
Note that in the proof of the previous Lemma a proper forcing will add a reflecting 
sequence for the open stationary function $\Sigma^r$ and the existence of a triple of 
ordinals coding the reals $r$ is a $\Sigma_1$ statement with parameters the ladder system 
$C$ and the real $r$. Thus already $\BPFA$ suffices to guarantee such a triple.

\end{proof}

What is still left to show is that the function $\Sigma^r$ is $M$-stationary for every $M$ in its domain.

\begin{Lemma}
Let $\Sigma^r$ be the function from the proof of Lemma 54,
$$\Sigma^r (M) := \{ N \in [M \cap \omega_4]^{\omega} \, :\, s_{\omega_2,\omega_3,\omega_4}(N, M \cap
\omega_4) \text{ is an initial segment of } r \}.$$
Then $\Sigma^r$ is $M$-stationary for every $M \prec H_{\theta}$, where $\theta$ is a sufficiently large regular cardinal.
\end{Lemma}
\begin{proof}
The set $X_{\Sigma^r}$ which comes along with $\Sigma^r$ is $\omega_4$. As the closed 
filter on $[\omega_4]^{\omega}$ is generated by sets which are closed under functions 
$F:[\omega_4]^{<\omega} \rightarrow \omega_4$ it suffices to show that whenever 
$M \prec H_{\theta}$ and $F:[\omega_4]^{<\omega} \rightarrow \omega_4$ is a function in $M$, 
then there is an $X \in M$ which is closed under $F$ and for which $X \in \Sigma^r(M)$ does hold. 
We will do this using the previously defined games $\mathcal{G}_{\nu}^{F}$ and the already proved fact 
that there is a club $C\subset \omega_1$ of ordinals $\nu$ for which player II has a winning strategy.

Thus let $M$ be an arbitrary elementary submodel of $H_{\theta}$. We will argue almost 
entirely inside $M$. Let $\bar{M}$ as always denote the transitive collapse of $M$, let $\pi$ be the 
collapsing function and set $\alpha_M:= \pi(\omega_1)$,
$\beta_M:= \pi(\omega_2)$, $\gamma_M := \pi(\omega_3)$ and $\delta_M:= \pi(\omega_4)$. 
Let $C$ be our fixed ladder system, and for an ordinal $\rho < \omega_4$ let
$ht_{\omega_i}(\rho)$, the height of $\rho$ in $\omega_i$ be defined as 
$ht_{\omega_i}(\rho) := card(\pi(\rho) \cap C_{\pi(\omega_i)})$.
Our goal is to show that we can find a $\nu \in C \cap M$ and play 
finitely many rounds of the game $\mathcal{G}^{F}_{\nu}$ inside $M$ such 
that if $T$ is the finite set of relevant ordinals played by II and $X:= cl_F(T\cup \nu)$
then $s_{\omega_2 \omega_3 \omega_4} (X, M \cap \omega_4)$ is an initial segment of $r$ and thus in $\Sigma^r(M)$.
The two players I and II will collaborate to achieve this.

We consider the set of winning strategies $\sigma_{\nu}$ for II and what 
they do at stage one. Let us assume that I plays 0 in all of his three first moves.
II will respond with pairs 
$(\kappa_0^{\nu}, \epsilon_{0}^{\nu})$, $(\lambda_{0}^{\nu}, \vartheta_{0}^{\nu})$, and $(\mu_{0}^{\nu}, \varpi_{0}^{\mu})$.
We can consider the suprema of the sets $\{\epsilon_{0}^{\nu} \, : \, \nu \in C \}$ 
$\{ \vartheta_{0}^{\nu} \, : \, \nu \in C \}$ and $\{ \varpi_{0}^{\nu} \, : \, \nu \in C \}$ and 
call them $\epsilon$, $\vartheta$ and $\varpi$ respectively.
These suprema are elements of $M$ but we leave $M$ for a moment to compare the height of $\epsilon$ 
below $\omega_2$, the height of $\vartheta$ below $\omega_3$ and the height of $\varpi$ below $\omega_4$. 
We pick the maximum of these three natural numbers, call it $k$ and fix a $\nu \in C \cap M$ be such 
that its height below $\omega_1$, say $n$ is bigger than $k$.
We ensured this way that for the set $X$ we will construct inductively $s_{\omega_2 \omega_3 \omega_4} (X, M \cap \omega_4)$ 
will be defined on the first $n$ digits.

The first round of the game is played as follows:
player I simply plays 0 three times and player II responds using his winning strategy $\sigma_{\nu}$. 
His answers, say $(\kappa_0, \epsilon_0)$, $(\lambda_0, \vartheta_0)$ and $(\mu_0, \varpi_0)$ are such that the height of
$\epsilon_0$ below $\omega_2$, the height of
$\vartheta_0$ below $\omega_3$ and the height
of $\varpi_0$ below $\omega_4$ are less than
$n$ and thus if we let $X_0$ be
the $cl_F (\{ \kappa_0, \lambda_0, \mu_0 \} \cup \nu)$ then 
$s_{\omega_2 \omega_3 \omega_4} (X_0, M\cap \omega_4)= o(x\backslash n, y\backslash n, z\backslash n)\upharpoonright n$ 
(where the sets $x,y$ and $z $ are defined as usual) will be $\ast$. Thus we start our construction 
without accidentally having coded some information already in the first step.

Assume now inductively that rounds $0,1,...,i-1$ have already been played and that 
$X_{i-1}:= cl_{F} (\{ \kappa_0, \lambda_0, \mu_0,..., \kappa_{i-1}, \lambda_{i-1}, \mu_{i-1}\} \cup \nu)$ 
is such that $s_{\omega_2 \omega_3 \omega_4}(X_{n-1}, M\cap \omega_4)$ codes the previously fixed real $r$ on the first
$i-2$ many digits. Our goal is to code $r$'s $i-1$-th digit.
For that purpose I plays an ordinal $\beta_i$ in $M$ such that 
$ht_{\omega_2} (\beta_i) > max \{ ht_{\omega_3} (\vartheta_{i-1}), ht_{\omega_4} (\varpi_{i-1}), n \}$.
Player II responds according to his winning strategy $\sigma_{\nu}$ with the pair $(\kappa_{i}, \epsilon_{i})$
We have by the rules of the game that $\beta_i \le \kappa_i \le \epsilon_i$.
We now split into cases according the $(i-1)$-th entry of $r$.
\par
\begin{flushleft}
\textbf{Case 1.}
\end{flushleft}
$r(i-1) =0$. Then player I picks a $\gamma_i$ such that 
$ht_{\omega_3}(\gamma_i) > ht_{\omega_2} (\epsilon_i)$ and player II 
answers with $(\lambda_i, \vartheta_i)$ according to $\sigma_{\nu}$.
Then I plays $\delta_i$ such that $ht_{\omega_4} (\delta_i) > ht_{\omega_3} (\vartheta_i)$ 
and II plays again what $\sigma_{\nu}$ tells him to, say $(\mu_i, \varpi_i)$.
By the rules of the game $\mathcal{G}_{\nu}^{F}$ we have that $\gamma_i \le \lambda_i \le \vartheta_i$ and
$\delta_i \le \mu_i \le \varpi_i$. Hence the height of any point in the interval $(\beta_i, \epsilon_i)$ 
is less than the height of any point in the interval $(\gamma_i, \vartheta_i)$ which are less than 
the height of any point of the interval $(\delta_i, \varpi_i)$.
Thus when going to the images under the $M$ collapse and constructing the according sets $x,y$ and
$z$ we see that
for $X_i :=X_{i-1} \cup \{ \kappa_i, \lambda_i, \mu_i \} $, $s_{\omega_2 \omega_3 \omega_4} (cl_F (X_{i} \cup \nu) (i-1) = 0$ as desired.
As $\nu$ is a winning strategy for II the already coded information remains untouched when passing from $X_{i-1}$ to $X_i$.
\\

\begin{flushleft}\textbf{Case 2.}\end{flushleft}
If $r(i-1) =1 $ then we argue very similar to case 1, we just have to ensure that we switch 
the heights of the ordinals in the intervals. To be more precise we make player I and II 
play in such a way that the height below $\omega_2$ of any point in the interval 
$(\beta_i, \epsilon_i)$ is smaller than the height below $\omega_4$ of any point 
in the interval $(\delta_i, \varpi_i)$ which in turn is smaller than the height 
below $\omega_3$ of any point in the interval $\gamma_i, \vartheta_i)$.
It is easy to see that such a play exists.

\par

If I and II play the game as described above they will code the real $r$ up to 
$n$ in $n+1$ many rounds. Let $X := X_{n+1}= cl_F (\{\kappa_i, \lambda_i, \mu_i \, : \, i \le n+1 \} \cup \nu)$ 
then as $\nu$ is a winning strategy in $\mathcal{G}^F_{\nu}$ it 
follows that $s_{\omega_2 \omega_3 \omega_4} (X, M \cap \omega_4) \upharpoonright n = r \upharpoonright n$ 
and as we played entirely within the elementary submodel $M \prec H_{\theta}$ and since $X$ is 
closed under $F$ we have an $F$-closed witness for $\Sigma^r(M) \cap M  \ne \emptyset$.
Thus $\Sigma^r$ is $M$ stationary.

\end{proof}
Note that the coding method just described does
go well with the sealing forcings we use to seal off long antichains in $P(\omega_1) \backslash \NS$
as we only use them whenever the forcing is
semiproper.
We finally have the techniques to prove that,
assuming the existence of $M_1^{\sharp}$, there is a model of $\ZFC$ where $\NS$ is $\aleph_2$-saturated 
and which has a $\Sigma_{4}^{1}$-well-order on the reals.

\section{Coding the reals}

We briefly summarize the main results of the last section:

\begin{enumerate}
\item[$(\dagger)$] Given ordinals $\omega_1 < \beta < \gamma < \delta < \omega_2$ of cofinality $\omega_1$,
there exists a proper notion of forcing $\forceP_{\beta \gamma \delta}$ such that after forcing with it the following holds:
There is an increasing continuous sequence $(N_{\xi} \, : \, \xi < \omega_1)$ such that $N_{\xi} \in [\delta]^{\omega}$ 
whose union is $\delta$ such that for every limit $\xi < \omega_1$ and every $n \in \omega$ there 
is $\nu < \xi$ and $s_{\xi}^n \in 2^n$ such that 
$$s_{\beta \gamma \delta}(N_{\eta}, N_{\xi}) \upharpoonright n = s_{\xi}^{n}$$ holds for every $\eta$ in the interval $(\nu, \xi)$.
We say then that the triple $(\beta, \gamma, \delta)$ is stabilized.
 
\item[$(\ddagger)$] Further if we fix a real $r$
there is a proper notion of forcing $\forceP_r$ such that the forcing will produce for a triple of ordinals 
$(\beta_r, \gamma_r, \delta_r)$ of size and cofinality $\aleph_1$ a continuous, increasing sequence 
$(P_{\xi} \, : \, \xi < \omega_1)$, $P_{\xi} \in [\delta_r]^{\omega}$ such that $\bigcup P_{\xi} = \delta_r$ 
and such that for every limit $\xi < \omega_1$ there is a $\nu < \xi$ such that 
$$\bigcup_{\nu < \eta < \xi} s_{\beta_r \gamma_r \delta_r} (P_{\eta}, P_{\xi}) = r.$$ 
We say then that the real $r$ is determined by the triple $(\beta_r, \gamma_r, \delta_r)$.
\end{enumerate}
We can use these two forcing notions to set up a well-order on the reals, which is suitable for our purposes.
First we define a partial function $f$ assigning triples of ordinals to reals: \begin{Definition}
For a real r we let $f(r)$ be the antilexicographically least triple $(\beta, \gamma, \delta)$  
for which there is a sequence of models 
$(N_{\xi} \, : \, \xi < \omega_1)$, $N_{\xi} \in [\delta]^{\omega}$, $\bigcup (N_{\xi}) = \delta$ such that the set
$$\{ \xi < \omega_1 \, : \, \exists \nu <\xi (\bigcup_{\nu < \eta< \xi} s_{\beta \gamma \delta} (N_{\eta}, N_{\xi}) = r ) \}$$ 
is club containing (if there is such a triple).
\end{Definition}
Note that this function $f$ is only a partial function. However we can force any given 
real to be in the domain of $f$ using the proper forcing described in $(\ddagger)$. We use
the function $f$ to define a partial well-founded order, which will become a well-order during our iteration :
\begin{Definition}
Let $r,s$ be two reals for which $f(r)$ and $f(s)$ is defined. Then we write $$r<s \text{ iff } f(r) <_{antilex} f(s).$$
\end{Definition}
It is clear that $<$ will become a well-order as soon as every real $r\in \omega^{\omega}$ is in the domain of $f$.
The crucial property of $<$ is the following:
\begin{Lemma}
Suppose that $M$ is a transitive model
such that
every triple of ordinals $\omega_1< (\beta, \gamma, \delta) < \omega_2$ is stabilized. Suppose further that  
$M\models x<y$, for two reals $x,y \in M$ then $x<y$ will hold in every stationary subset-preserving forcing extension $M[G]$ of $M$.
\end{Lemma}
\begin{proof}
Assume for a contradiction that
there is a stationary set preserving notion of forcing $\forceP$ such that $M[G] \models y<x$, for a generic $G$ for
$\forceP$.
Thus $G$ must have added a continuous, increasing sequence of countable sets of ordinals $(N_{\xi} \, : \, \xi < \omega_1)$ 
such that for a triple $(\beta, \gamma, \delta)$ the set 
$\{ \xi < \omega_1 \, : \, \exists \nu < \xi (\bigcup_{\nu < \eta< \xi} s_{\beta \gamma \delta} (N_{\eta}, N_{\xi}) = y ) \}$ 
is club containing, and $(\beta, \gamma, \delta)$ is antilexicographically less than the least triple 
$(\beta_x, \gamma_x, \delta_x)$ for $x$. But we have assumed that every triple of ordinals $< \omega_2$ in $M$ 
is already stabilized, I.e. there exists in $M$ an increasing continuous sequence of countable sets of 
ordinals $(P_{\xi} \, : \, \xi < \omega_1 )$ such that $\bigcup_{\xi < \omega_1} P_{\xi} = \delta$. 
By a hand by hand argument carried out in $M[G]$ we obtain a club $C= \{ \xi < \omega_1 \, : \, P_{\xi}= N_{\xi} \}$ 
and for every limit point $\xi$ of $C$ we have that $\bigcup_{\nu < \eta < \xi} s_{\beta \gamma \delta} (N_{\eta}, N_{\xi})$
is the same as $\bigcup_{\nu< \eta< \xi} s_{\beta \gamma \delta} (P_{\eta}, P_{\xi})$.
Thus $$M[G] \models \{ \xi< \omega_1 \, : \, \bigcup s_{\beta \gamma \delta} (P_{\eta}, P_{\xi}) = y \}
\text{ is club containing}.$$ But as $y$ and $(P_{\xi} \, : \, \xi < \omega_1)$ are objects in $M$ and the 
function $s_{\beta \gamma \delta}$ is absolute we can consider the set also in $M$. As $M[G]$ is a 
stationary set-preserving generic extension of $M$ we have that already $M$ sees that the set 
$\{ \xi < \omega_1 \, : \, \exists \nu < \xi ( \bigcup_{\nu < \eta < \xi} s(P_{\eta}, P_{\xi})= y) \}$ 
is club containing in $M$ as otherwise the complement would be stationary and therefore stationary in the extension.
This is a contradiction to $M \models x< y$.
\end{proof}
Thus the order $<$, once witnessed in a transitive model such that every triple of ordinals is stabilized, 
will not change in all outer models which preserve stationary subsets of $\omega_1$. This has as a 
consequence that we can build up the well-order $<$ gradually during a forcing iteration and 
once we have a model of the just described type which sees that $x<y$, then $x<y$ in our final 
model of the iteration, as long as the iteration preserves stationary subsets of $\omega_1$.
This enables us to localize the well-order $<$, thus arriving at a projective well-order of small complexity.

\section{Definition of the iteration}

Next we describe how to code reals nicely while making $\NS$ $\aleph_2$-saturated.
In order to get $\NS$ $\aleph_2$-saturated we need an $RCS$-iteration of length $\kappa$, where $\kappa$ is the Woodin cardinal.
Again we let a $\diamondsuit$-sequence decide what to do in our forcing 
iteration. We fix such a $\diamondsuit$-sequence $(a_{\alpha} \, : \, \alpha < \kappa)$ in the ground model
$V=M_1$, and let $\vec{C}$ be the ladder system we obtain using the $\Sigma^1_3$-well-order of the reals of $M_1$, 
which will still be a ladder system in small forcing extensions of $M_1$ which preserve $\omega_1$.
We describe first informally how the iteration looks like.
As always we have stages which are used to code information yielding the definable well-order and stages where
we seal off long antichains in $P(\omega_1)/ \NS$. We ensure that we code all the reals we generate 
during the iteration into triples of ordinals $(\beta, \gamma, \delta)$ using the proper forcing of $(\ddagger)$. At the same time
we ensure that all the triples of ordinals below $\omega_2$ stabilize using the forcing described in $(\dagger)$. 
Additionally whenever our $\diamondsuit$-sequence hits the name of a long antichain in $P(\omega_1)/\NS$ we seal it off.
As we have stationarily many inaccessible cardinals below the Woodin $\kappa$ we will hit stationarily often 
inaccessible stages $\alpha$ such that the model $(M_1)_{\alpha}[G_{\alpha}]$ satisfies the following: \begin{enumerate}
\item[(1)] $(M_1)_{\alpha}[G_{\alpha}] \models \ZFP$ 
\item[(2)] $(M_1)_{\alpha}[G_{\alpha}] \models \forall \beta \gamma \delta <\omega_2 ((\beta, \gamma, \delta) \text{ is stabilized}$)
\item[(3)] $(M_1)_{\alpha}[G_{\alpha}] \models \forall r\in \omega^{\omega} \exists (\beta_r, \gamma_r, \delta_r)$
($r$ is determined by $(\beta_r, \gamma_r, \delta_r))$.
\end{enumerate}
Whenever we hit such a stage everything
$(M_1)_{\alpha}[G_{\alpha}]$ sees about $<$ will be preserved in all future extensions in our iteration 
by Lemma 61. Thus we will additionally localize the information which is seen by $(M_1)_{\alpha}[G_{\alpha}]$ 
concerning $<$, I.e. for every pair of reals $x<y$ in $(M_1)_{\alpha}[G_{\alpha}]$ we add a subset $Y_{x,y}$ of $\omega_1$ such that every
countable transitive model $N$ which contains $Y_{x,y} \cap \omega_1^{N}$ will also see that $x<y$. This uses a proper forcing again.
As all the iterands are proper or semiproper, using
an $RCS$-iteration will yield a semiproper notion of forcing.
In the end we will argue that indeed $\NS$ is saturated and there is a $\Sigma_4^{1}$-definable well-order on the reals.

We start now with a more detailed description of how the iteration should look like. We will 
construct the iteration recursively, so assume that $\alpha < \kappa$
and we have already constructed $\mathbb{P}_{\beta}$ for $\beta \le \alpha$.
We define the forcing $\dot{\forceQ}_{\alpha}$ in $V^{\forceP_{\alpha}}$ as follows:

\begin{itemize}
\item[(i)] Assume that $a_{\alpha}$ is the $\forceP_{\alpha}$-name of a real $r_{\alpha}$. 
Then we let $\dot{\forceQ}_{\alpha}$ be the $\forceP_{\alpha}$-name of the forcing
which codes $r_{\alpha}$ into a triple of
ordinals $(\beta_{r_{\alpha}}, \gamma_{r_{\alpha}}, \delta_{r_{\alpha}})$, such 
that $\beta_{r_{\alpha}}, \gamma_{r_{\alpha}}, \delta_{r_{\alpha}} < \omega_2$ and
using the already fixed $\vec{C}$-sequence.
This forcing is followed by considering all the triples of ordinals $(\beta', \gamma', \delta')$ 
which are
antilexicographically below  $(\beta_{r_{\alpha}}, \gamma_{r_{\alpha}}, \delta_{r_{\alpha}})$ and 
which have not been stabilized yet. We use an $RCS$-iteration of forcings which stabilize each 
such triple $(\beta', \gamma', \delta')$.
As a summary $\dot{\forceQ}_{\alpha}$ is an $\omega_1$-long $RCS$ iteration of proper forcings 
resulting in a proper forcing, and we obtain a model where the real $r_{\alpha}$ is coded into 
the triple $(\beta_{r_{\alpha}}, \gamma_{r_{\alpha}}, \delta_{r_{\alpha}})$ with the help of the 
ladder system $\vec{C}$, and each other triple of ordinals below it will be stabilized.

Otherwise force with $Col(2^{\aleph_2}, \aleph_1)$, the usual L\'evy collapse which collapses $2^{\aleph_2}$ down to $\aleph_1$.

\item[(ii)] Assume that $\alpha$ is an inaccessible, further that $a_{\alpha}$ is the 
$\mathbb{P}_{\alpha}$-name of a maximal antichain $S_{\alpha}$ in $P(\omega_1)$/$\NS$, 
and assume that the sealing forcing $\mathbb{S}(S_{\alpha})$ is semiproper. 
Then force with it, I.e. let $\dot{\forceQ}_{\alpha}$ be $\mathbb{S}(S_{\alpha})$.

Otherwise force with $Col(2^{\aleph_2}, \aleph_1)$.

\item[(iii)] If $\alpha$ is an inaccessible and if $(M_1)_{\alpha}[G_{\alpha}]$ 
is a model such that property $(1)$ $(2)$ and $(3)$ from above holds, then we first collapse its size down to
$\aleph_1$. We consider
a pair of reals $x,y \in (M_1)_{\alpha}[G_{\alpha}]$ such that $(M_1)_{\alpha} [G_{\alpha}] \models x<y$ and 
construct a set $Y_{x,y}\subset \omega_1$ which should code in a nice way the information
that $(M_1)_{\alpha} [G_{\alpha}] \models x<y$ (this will be specified below).
Now code the set $Y_{x,y}$ into a real $s_{x,y}$ using almost disjoint coding forcing relative to 
an almost disjoint family of reals in $M_1$.
We pick the pair $(x<y)$ in such a way that every such pair we create during our iteration will be
considered at some inaccessible stage $\alpha$. This can be easily achieved using some bookkeeping function.

\end{itemize}
The points (i) and (ii) are clear, thus we
shall discuss (iii) in detail:
Note first that it is straightforward
to code the previously fixed ladder system $\vec{C}$ into a subset $X$ of $\omega_1$ such that for every limit ordinal
$\xi < \omega_1$, if we decode $X \cap \xi$ we end up with $\vec{C} \upharpoonright \xi$.
Assume now that we are in the situation described in (iii), thus $\alpha$ is an inaccessible and
$(M_1)_{\alpha}[G_{\alpha}]$ is a model of $\ZFP$, (1),(2) and (3).
We collapse its size to $\aleph_1$ using L\'evy-collapse and let $H$ be the generic filter. Let $x,y$ be two reals in
$(M_1)_{\alpha}[G_{\alpha}]$ such that
$(M_1)_{\alpha}[G_{\alpha}] \models x<y$. We describe now the set $X_{x,y} \subset \omega_1$ which codes the information:

\begin{Claim}
In $M_1[G_{\alpha}][H]$ there is a set $X_{x,y} \subset \omega_1$ such that $X_{x,y}$ can 
be recursively partitioned into 6 subsets such that the following holds:
\begin{enumerate}
\item $dec_1(X_{x,y})= \vec{C}$
\item $dec_2(X_{x,y})= x$
\item $dec_3(X_{x,y})=y$
\item $dec_4(X_{x,y})$ is the set consisting of the triple of ordinals 
$(\beta_x, \gamma_x, \delta_x)$ and the continuous sequence of 
models $(N_{\xi}^{x} \, : \, \xi < \omega_1)$ such that $(\beta_x, \gamma_x, \delta_x)$ is the 
antilexicographically least triple which codes $x$ with the help of $\vec{C}$ and witnessing 
sequence $(N_{\xi}^{x} \, : \, \xi < \omega_1)$
\item $dec_5(X_{x,y})$ is the set similar defined as in the previous point but with $y$ instead of $x$.
\item $dec_6(X_{x,y})$ is the set of all the continuous increasing sequences $(N_{\xi}^{(\beta',\gamma', \delta')}
\, : \, \xi < \omega_1)$ for $(\beta', \gamma', \delta')< (\beta_x,\gamma_x,\delta_x)$ of sets of ordinals
which witness that the triple $(\beta', \gamma', \delta')$ stabilizes at a real which is neither $x$ nor $y$.
       
\end{enumerate}

\end{Claim}
The construction of such a set $X_{x,y}$ is straightforward.
As a consequence, for every transitive model $M$ of $\ZFC$ which contains $X_{x,y}$, $M$ 
will see that $x<y$ as it contains all the relevant information.
The goal now is to rewrite $X_{x,y}$ into a set $Y_{x,y} \subset \omega_1$ such that 
every countable, transitive $M$ which contains $Y_{x,y} \cap \omega_1^{M}$ already 
sees that $x<y$ using the local ladder $\vec{C}\upharpoonright \omega_1^{M}$.
We can force the existence of such a set $Y_{x,y}$ with a proper notion of forcing.

\begin{Lemma}
There is a proper notion of forcing $\forceR$ which introduces a set $Y_{x,y} \subset \omega_1$
such that if $H'$ is an $\forceR$ generic filter
$M_1[G_{\alpha}][H][H']$ satisfies that there is subset $Y_{x,y}$ of $\omega_1$ such that 
if $\xi = \omega_1^N$ for a countable transitive model $N$ then there are recursively definable
decoding functions $dec_i$ such that the following holds: \begin{enumerate}
\item $dec_1(Y_{x,y})= \vec{C}$ and for every
limit ordinal $\xi < \omega_1$, with $\xi \omega=\xi$, $dec_1(Y_{x,y} \upharpoonright \xi) = \vec{C}\upharpoonright \xi$.
\item $dec_2(Y_{x,y})= x$
\item $dec_3(Y_{x,y})= y$
\item $dec_4(Y_{x,y} \upharpoonright \xi)$ is the set consisting of a triple of ordinals 
$(\beta_x^{\xi}, \gamma_x^{\xi}, \delta_x^{\xi})$ and a continuous sequence of sets of 
ordinals $(N_{I}^{x,\xi} \, : \, I \in \xi)$ such that $x$ is coded by the triple 
$(\beta_x^{\xi}, \gamma_x^{\xi}, \delta_x^{\xi})$ using the local ladder system 
$\vec{C} \upharpoonright \xi$ and witnessed by $(N_{I}^{x,\xi} \, : \, I \in \xi)$.
\item $dec_5(Y_{x,y})$ is defined similar as 4 but with x replaced by y.
\item $dec_6(X_{x,y})$ is the set of all the continuous increasing sequences $(N_{I}^{\xi,(\beta',\gamma', \delta')}
\, : \, I < \xi)$ for $(\beta', \gamma', \delta')< (\beta_x^{\xi},\gamma_x^{\xi},\delta_x^{\xi})$ of sets of ordinals
which witness that the triple $(\beta', \gamma', \delta')$ stabilizes at a real 
which is neither $x$ nor $y$ using the local ladder $\vec{C} \upharpoonright \xi$.
\end{enumerate}

\end{Lemma}
\begin{proof}
 
Working in $M_1[G_{\alpha}][H]$ we have that $(M_1)_{\alpha}[G_{\alpha}]$ is a 
model of $\ZFP$ and (1) and (2) of size $\aleph_1$. By the previous Claim we 
know that there is a set $X_{x,y}$ such that all the points in the claim are true.
Fix a model of the form $(M_1)_{\eta}[G_{\alpha}][H]$ for $\eta > \alpha$ which 
contains $X_{x,y}$ and consider the club $C$ of countable, elementary submodels containing $X_{x,y}$.
We can easily arrange
that $X_{x,y}$ satisfies already (1) of the Claim.
For $M$ an arbitrary element of $C$ the following holds: 
$$(M_1)_{\eta} [G_{\alpha}][H] \models X_{x,y} \text{ satisfies }(1)-(6)\text{ of the previous Claim}$$ thus for each $M \in C$:
$$M \models X_{x,y} \text{ satisfies } (1)-(6)$$
thus
$$\bar{M} \models \bar{X}_{x,y} \text{ satisfies } (1)-(6)$$ But 
$\bar{X}_{x,y}=X_{x,y} \upharpoonright \omega_1^{M}$ and as the decoding functions 
$dec_i$ are absolute for transitive models we get that $X_{x,y} \upharpoonright \xi$ 
satisfies (1)-(6) for $\xi$'s which are the $\omega_1$ of some $M \in C $. 
In order to get the full statement of the Claim we add additional information 
to $X_{x,y}$ which yields $Y_{x,y}$ such that any countable transitive model 
$N$ which contains $Y_{x,y} \cap \omega_1^{N}$ must have its $\omega_1$ to be an $\omega_1^{M}$ for some $M \in C$.
To achieve this we use forcing.

Let $\forceR$ be the following partial order:
conditions $p\in \forceR$ are $\omega_1$-Cohen conditions, i.e.
functions from limit ordinals $\xi < \omega_1$
to $2$ which satisfy:
\begin{enumerate}
\item the even ordinals of $\{ \eta< \xi \, : \, p(\eta)=1\}$, where $\xi = dom(p)$ code the set $X_{x,y} \cap \xi$.

\item for every limit ordinal $\xi$, if $M$ is countable and transitive 
and $\xi= \omega_1^M$ and $(p \upharpoonright \xi) \in M$ then $p(\xi)$ 
seen as a subset of $\xi$ satisfies points $(1)-(6)$ for every limit $\eta < \xi$.

\end{enumerate}

Note that whenever we do have a condition $p \in \forceR$, and $\xi < \omega_1$ 
is a limit ordinal
we can extend $p$ to $q<p$ such that $\xi \in dom(q)$.
This is clear as we can write into the first odd $\omega$-block of $q$ following 
$dom(p)$ a surjection of $\xi$ to $\omega$. Then no countable transitive model $M$ 
which contains $q$ can have its $\omega_1$ at $\xi$, thus the second property for 
being a condition in $\forceR$ is satisfied automatically.
Thus the set $D_{\eta}:= \{ p \in \forceR \, : \, \eta \in dom(p) \}$ is dense for 
every $\eta < \omega_1$ and the generic will produce a subset of $\omega_1$, $Y_{x,y}$ 
with the desired properties $(1)-(6)$ for countable, transitive models in $M_1[G_{\alpha}][H]$. 
This already suffices as we will see below that the forcing $\forceR$ is also $\omega$-distributive.

What is left is to show that the forcing $\forceR$ is proper: for that we pick the
$(M_1)_{\eta}[G_{\alpha}][H]$ from above and let the club $C$ be the set of all countable 
elementary submodels of it which contain the set $X_{x,y}$ from above.
If $M \in C$,  and $p \in \forceR \cap M$ then
we shall construct a $q< p$ which is $(M,\forceR)$-generic.
We list all the dense sets $D_n$ in $M$ and
recursively construct a descending sequence of conditions starting at $p= p_0>p_1>...$ 
such that $p_n \in D_n$ and such that $sup (dom(p_n)) = \omega_1 \cap M$.
If we can show that the limit $p_{\omega}$ is a condition in $\forceR$, we are done. 
Thus we have to show that whenever $p_{\omega} \cap \xi$ is contained in a countable, 
transitive model $N$ such that $\xi = \omega_1^{N}$
then it will satisfy conditions $(1)-(6)$. This is clear by definition of $\forceR$ for every 
$\xi < dom(p_{\omega})$. If $\xi = dom(p_{\omega})$ then as $\xi = \omega_1^{\bar{M}}$, $M \in C$ 
and $p_{\omega}$ codes $X_{x,y} \cap \xi$ we know by the above that $X_{x,y} \cap \xi$ 
satisfies $(1)-(6)$, and so does $Y_{x,y}\cap \xi$ if we change the decoding functions 
in the obvious way. Thus $p_{\omega}$ is a condition in $\forceR$ as desired and the forcing $\forceR$ is proper.

Note that the same argument shows that $\forceR$ is also $\omega$-distributive, 
and so the added set $Y_{x,y}$ has the desired properties not only for countable sets in $M_1[G_{\alpha}][H]$ but also for sets
in $M_1[G_{\alpha}][H][H']$ as desired.

\end{proof}
Now that we have constructed the localized set $Y_{x,y}$ for two reals $x<y$ in 
$(M_1)_{\alpha}[G_{\alpha}]$ we code $Y_{x,y} \subset \omega_1$ into a real 
$r_{x,y}$ using almost disjoint coding. We fix an 
$M_1$-family $(r_{\alpha} \, : \, \alpha < \omega_1)$ of almost disjoint reals and add a real 
$r_{x,y}$ such that the following holds:
$$\forall \xi < \omega_1 \,( \xi \in Y_{x,y}
\text{ iff } r_{x,y} \cap r_{\xi} \text{ is almost disjoint.})$$ This forcing is ccc, 
therefore proper and for the real $r_{x,y}$ the following holds: \begin{enumerate}
\item[$\heartsuit$] Every countable, transitive model $M$ which contains $r_{x,y}$ 
and $(M_1)_{\omega_1^{M}}$ sees that $x<y$ using the ladder system $\vec{C} \upharpoonright \omega_1^{M}$.
\end{enumerate}
This ends the discussion of the forcing we use in case (iii).

\section{$\NS$ is saturated and a projective well-order. Open Questions.}

In this section we show that the model we obtain following the definition 
given in the previous section will indeed satisfy that $\NS$ is saturated and 
there is a $\Sigma_4^{1}$ definable well-order on the reals. We start with the well-order first.
\begin{Lemma}
Let $G$ be a generic filter for the
forcing defined in the last section.
Then $$M_1[G] \models \text{ there is a } \Sigma_4^{1} \text{ definable well-order on the reals.}$$
\end{Lemma}
\begin{proof}
We can use the cofinal set of $M_1$-initial segments in $M_1[G]$ which is $\Pi^{1}_{2}$ definable to construct a
$\Sigma^{1}_{3}$-definable ladder system $\vec{C}$ in $M_1[G]$.
Simply let $(\alpha, C_{\alpha}) \in \vec{C}$ if and only if there is a countable 
$M_1$ initial segment $M$ which contains $(\alpha,C_{\alpha}) $ and which sees that 
$C_{\alpha}$ is the $<_M$-least set in $M$ (where $<_M$ denotes the usual definable well-order 
on the mouse $M$) which is cofinal and has ordertype $\omega$. This definition is $\Sigma^{1}_{3}$ and
$vec{C} \upharpoonright \alpha$ is uniformly definable over $\mathcal{J}^{M_1}_{\alpha}$. 
We let $\psi(x)$ be the defining
formula, i.e. for every $C \subset \omega_1$ we have that $\psi(C)$ if and only if $C \in \vec{C}$.  
We fix the ladder system $\vec{C}$ defined via $\psi$ and code, using the machinery described and 
explained above, everything relative to it.
We claim that the following defines the well-order $<$ in $M_1[G]$: \begin{enumerate}
\item[($\ast$)] $x<y$ if and only if $\exists r \forall M$($M$ countable and transitive and $\mathcal{J}^{M_1}_{\omega_1} \in M$
then $M \models (\psi(dec(r,(M_1)_{\omega_1^{M}}))^{\mathcal{J}^{M_1}_{\omega_1}}$ and
$M \models$ $x<y$ using the local ladder system $\vec{C} \cap \omega_1^{M})$.
\end{enumerate}
Here we write $(dec(r,(M_1)_{\omega_1^{M}})$ for the set we obtain when we decode the real using 
the almost disjoint family of $(M_1)_{\omega_1^{M}}$-reals and obtain a
subset of $\omega_1^{M}$. Thus the statement $\psi(dec(r,(M_1)_{\omega_1^{M}})$ tells us that 
the ladder system coded into the real is the previously fixed ladder system $\vec{C}$ relativized to $\omega_1^{M}$.
The direction from left to right is clear as whenever $x<y$ in $M_1[G]$ then the names of $x$ 
and $y$ will appear at some earlier stage $\beta < \kappa$ in our iteration.
Thus $x<y$ will be witnessed already at some inaccessible stage $\alpha < \kappa$ as we used a 
bookkeeping function which looks at each name of a pair of reals unboundedly often on the 
inaccessibles below $\kappa$.
Note that in this case the definition of our iteration guarantees us that we have added a 
real $r_{x,y}$ such that $\heartsuit$ holds which is exactly what we want.

For the direction from right to left note first that
the right hand side of
$(\ast)$ holds also for large enough uncountable models $M$ containing $r$ and $(M_1)_{\omega_1}$
and which satisfy that each triple of ordinals
below $\omega_2^{M}$ is stabilized.
Indeed if we assume that $M$ is uncountable, $r\in M$ and $(M_1)_{\omega_1}$ $\in M$ and every triple of ordinals is stabilized and
$M$ does decide that $y<x$, then a countable elementary submodel $N \prec M$ containing $r$ will see as well
that $y<x$ using the ladder system which is coded into $r$.
This ladder system must be our fixed $\vec{C}$. But now the transitive collapse $\pi(N)=\bar{N}$ 
contains $r$ and we can assume that also $(M_1)_{\omega_1^{\bar{N}}}$ is in $\bar{N}$.
It will still think that $y<x$ using the
ladder system $\pi(\vec{C})=\vec{C} \cap \omega_1^{\bar{N}}$, but $\bar{N}$ is just as in 
the assumption for the right hand side for $(\ast)$. This gives us the contradiction. as $\bar{N}$ cannot see both $x<y$ and $y<x$.
Thus if we let $M$ be big enough such that every triple of ordinals below $\omega_2$ is stabilized then
it will see that $x<y$. But by Lemma 61 $x<y$ must hold in $M_1[G]$.

Note further that the definition of the well-order is of the form 
$\exists r \forall M( \Pi_{2}^1 \rightarrow \Pi_{2}^1)$ and therefore $\Sigma^{1}_4$.

\end{proof}

What is left is to show that in $M_1[G]$ the nonstationary ideal $\NS$ is 
indeed $\aleph_2$-saturated. But this does not cause any problems as the 
coding forcings were all seen to be proper, the sealing forcings were 
only used when semiproper and we used $RCS$-iteration for the limit steps.
Therefore the iteration yields a semiproper,
thus stationary set preserving extension of $M_1$
and we can just repeat Shelah's proof that
$\NS$ is $\aleph_2$-saturated in the final model.
Hence the following is true, which ends the proof of the theorem:
\begin{Lemma}
If $G$ denotes the generic filter for the
iteration then in $M_1[G]$ the nonstationary
ideal $\NS$ is $\aleph_2$-saturated.
\end{Lemma}

We end this thesis with a couple of remarks. At first glimpse it might seem that
the techniques introduced in the second chapter, i.e. isolating a suitable class of models
$\Stat$ and use $K$-definable objects to code patterns would as well yield a solution to the 
problem of projectively definable wellorders on the reals in the presence of $\NS$ saturated. Indeed this was our 
strategy for quite some time, it turned out however that this approach is fruitless.
The reason for this is that sealing the antichains in $P(\omega_1) \slash \NS$
is a semiproper forcing only, hence our stationary class $\Stat$
will lose its stationarity along the iteration as seen as a subclass of sets of the form
$M \cap K$ where $M\prec H_{\theta}$ for regular $\theta$. This fact makes the quest for
$K$ inside countable transitive models of the form $M[G]$ impossible, even though $M \in \Stat$ itself
can define its $K$ properly. $M[G] \cap K$ will not be $M$ anymore, thus
trying to define $K$ causes utter chaos.
This problems led me to consider other methods of coding which finally resulted in the 
proof described in this chapter.
Note however that the methods developed here do not seem to get an easy proof of the following 
natural extension of the problem in chapter 2:

\begin{Question}
How to get a model of $\ZFC$ in which $\NS$ is $\aleph_2$-saturated and
the full nonstationary ideal $\NS$ is $\delta_1$-definable with parameter $K_{\omega_1}$?
\end{Question}

Another natural question is whether the $\delta^{1}_4$ wellorder is optimal while having 
$\NS$ saturated. By the results of Hjorth and Woodin, 
under the assumption that there is a measurable
cardinal, the $\Sigma^{1}_4$ wellorder is optimal but are there ways to do better in the
absence of a measurable?

\begin{Question}
 Is there a model of $\ZFC$ in which $\NS$ is saturated and there is 
a $\Delta^1_{3}$-definable wellorder on the reals?
\end{Question}

and finally the very interesting problem

\begin{Question}
 Is it possible to have a model of $\ZFC$ in which $\NS$ is saturated and $\CH$ holds?
\end{Question}

\end{document}